\documentclass[12pt,reqno,a4paper]{amsart}
\usepackage[doi=true,isbn=false,style=alphabetic,sorting=nyt,backend=biber,maxnames=5,maxalphanames=5,giveninits=true]{biblatex}
\AtEveryBibitem{\clearlist{language}}
\bibliography{main}
\usepackage[breaklinks,colorlinks,plainpages,hypertexnames=false,plainpages=false]{hyperref}
\hypersetup{urlcolor=blue, citecolor=blue, linkcolor=blue}

\usepackage[utf8]{inputenc}
\usepackage{amsmath}
\usepackage{amsthm}
\usepackage{amssymb}
\usepackage{amsfonts}
\usepackage{enumerate}
\usepackage{mathtools}
\usepackage{tikz-cd}
\usepackage{hhline}
\usepackage{stmaryrd}
\usepackage{enumitem}
\usepackage{centernot}
\usepackage[textwidth=25mm]{todonotes}
\usepackage{mathrsfs}
\usepackage{comment}
\usepackage{listings}
\usepackage{multicol}
\usepackage[capitalise, noabbrev]{cleveref} %
\usepackage[noend]{algorithmic} %

\usepackage{mathdots}
\usepackage{cleveref} %
\usepackage{nicematrix} %
\usepackage{mathalpha} %
\DeclareMathAlphabet{\dutchcal}{U}{dutchcal}{m}{n}
\newcommand{\dcu}{\dutchcal{u}}
\newcommand{\dcv}{\dutchcal{v}}
\newcommand{\dcw}{\dutchcal{w}}

\usepackage{array}
\newcolumntype{L}[1]{>{\raggedright\let\newline\\\arraybackslash\hspace{0pt}}m{#1}}
\newcolumntype{C}[1]{>{\centering\let\newline\\\arraybackslash\hspace{0pt}}m{#1}}
\newcolumntype{R}[1]{>{\raggedleft\let\newline\\\arraybackslash\hspace{0pt}}m{#1}}

\usepackage{tikz}
\usetikzlibrary{arrows}
\usetikzlibrary{decorations.pathreplacing}
\usetikzlibrary{matrix}
\usetikzlibrary{calc}
\usetikzlibrary{shapes}
\usetikzlibrary{patterns}
\usetikzlibrary{fit,backgrounds,scopes}
\newtheoremstyle{theoremstyle}
{10pt}      %
{5pt}       %
{\itshape}  %
{}          %
{\bfseries} %
{}         %
{ }      %
{}          %

\newtheoremstyle{algorithmstyle}
{10pt}      %
{5pt}       %
{}  %
{}          %
{\bfseries} %
{}         %
{ }      %
{}          %

\newtheoremstyle{examplestyle}
{10pt}      %
{5pt}       %
{}          %
{}          %
{\bfseries} %
{}         %
{ }      %
{}          %

\makeatletter
\newtheorem*{rep@theorem}{\rep@title}
\newcommand{\newreptheorem}[2]{%
\newenvironment{rep#1}[1]{%
 \def\rep@title{#2 \ref{##1}}%
 \begin{rep@theorem}}%
 {\end{rep@theorem}}}
\makeatother

\makeatletter %
\newcommand{\subalign}[1]{%
  \vcenter{%
    \Let@ \restore@math@cr \default@tag
    \baselineskip\fontdimen10 \scriptfont\tw@
    \advance\baselineskip\fontdimen12 \scriptfont\tw@
    \lineskip\thr@@\fontdimen8 \scriptfont\thr@@
    \lineskiplimit\lineskip
    \ialign{\hfil$\m@th\scriptstyle##$&$\m@th\scriptstyle{}##$\hfil\crcr
      #1\crcr
    }%
  }%
}
\makeatother

\newreptheorem{theorem}{Theorem}
\newreptheorem{corollary}{Corollary}
\newreptheorem{proposition}{Proposition}

\theoremstyle{theoremstyle}
\newtheorem{theorem}{Theorem}[section]
\newtheorem{lemma}[theorem]{Lemma}
\newtheorem{proposition}[theorem]{Proposition}
\newtheorem{corollary}[theorem]{Corollary}
\theoremstyle{examplestyle}
\newtheorem{example}[theorem]{Example}
\newtheorem{definition}[theorem]{Definition}

\newtheorem{notation}[theorem]{Notation}
\newtheorem*{notation*}{Notation}

\newtheorem{remark}[theorem]{Remark}

\newtheorem{assumption}[theorem]{Assumption}

\theoremstyle{algorithmstyle}

\newcommand{\CC}{\mathbb{C}}
\newcommand{\RR}{\mathbb{R}}

\newcommand{\ZZ}{\mathbb{Z}}

\newcommand{\PP}{\mathbb{P}}

\newcommand{\suchthat}{\;\ifnum\currentgrouptype=16 \middle\fi|\;}
\newcommand{\bigmid}{\left.\vphantom{\Big\{} \suchthat \vphantom{\Big\}}\right.}
\newcommand{\bigslant}[2]{{\raisebox{.2em}{$#1$}\left/\raisebox{-.2em}{$#2$}\right.}}

\newcommand{\an}{{\mathrm{an}}}
\newcommand{\lin}{{\mathrm{lin}}}
\newcommand{\nlin}{{\mathrm{nlin}}}

\newcommand{\st}{{\mathrm{st}}}

\DeclareMathOperator{\codim}{codim}
\DeclareMathOperator{\Conv}{Conv}
\DeclareMathOperator{\initial}{in}

\DeclareMathOperator{\mult}{mult}
\DeclareMathOperator{\MV}{MV}

\DeclareMathOperator{\rec}{rec}
\DeclareMathOperator{\Span}{Span}
\DeclareMathOperator{\Spec}{Spec}
\DeclareMathOperator{\Star}{Star}
\DeclareMathOperator{\trop}{trop}
\DeclareMathOperator{\Trop}{trop}
\DeclareMathOperator{\val}{val}

\newcommand\restr[2]{{\left.\kern-\nulldelimiterspace #1 \right|_{#2}}}

\makeatletter
\newcommand{\oset}[3][0ex]{%
  \mathrel{\mathop{#3}\limits^{
    \vbox to#1{\kern-2\ex@
    \hbox{$\scriptstyle#2$}\vss}}}}
\makeatother

\makeatletter
\newcommand{\uset}[3][0ex]{%
  \mathrel{\mathop{#3}\limits_{
    \vbox to#1{\kern-7\ex@
    \hbox{$\scriptstyle#2$}\vss}}}}
\makeatother
\makeatletter
\newcommand{\customlabel}[2]{%
   \protected@write \@auxout {}{\string \newlabel {#1}{{#2}{\thepage}{#2}{#1}{}} }%
   \hypertarget{#1}{#2}%
}
\makeatother

\begin{document}

\title[Generic root counts and flatness in tropical geometry]{Generic root counts and flatness\\ in tropical geometry} %

\author{Paul Alexander Helminck}
\address{Mathematical Institute, Tohoku University, Japan.
}
\email{paul.helminck.a6@tohoku.ac.jp}
\urladdr{https://paulhelminck.wordpress.com}
\author{Yue Ren}
\address{Department of Mathematics, Durham University, United Kingdom.}
\email{yue.ren2@durham.ac.uk}
\urladdr{https://yueren.de}

\subjclass[2020]{14T90,14M25,13P15}

\date{\today}

\keywords{tropical geometry, non-archimedean geometry, Berkovich analytifications, families of polynomial equations, root counts, tropical intersection numbers.}

\begin{abstract}
  We use tropical and non-archimedean geometry to study the generic number of solutions of families of polynomial equations over a parameter space $Y$. In particular, we are interested in the choices of parameters for which the generic root count is attained.
  Our families are given as subschemes $X\subseteq T$ where $T$ is a relative torus over $Y$.
  We generalize Bernstein's theorem from an intersecting family of hypersurfaces $X=V(f_1)\cap\dots\cap V(f_n)$ to an intersecting family of higher-codimensional schemes $X=X_1\cap\dots\cap X_k$, replacing the mixed volume by a tropical intersection product.
  Central to our work is the notion of tropical flatness of $X$ around a point $P\in Y$, which allows us to transfer tropical properties of the  fiber over $P$ to generic properties.  We show that tropical flatness holds over a dense open subset of the Berkovich analytification $Y^\an$, and that the tropical intersection number is attained as a root count at all $P\in Y^\an$ around which the $X_i$'s are tropically flat and the tropical prevariety of the fibers $\bigcap_{i=1}^k\Trop(X_{i,P})$ is bounded.
  We then study the generic root count of a wide class of parametrized square polynomial systems. This in particular gives tropical formulas for the volumes of Newton-Okounkov bodies, and the number of complex steady states of chemical reaction networks.
\end{abstract}

\maketitle

\section{Introduction}\label{sec:Introduction}
Consider a parametrized family of polynomial equations over an algebraically closed field, such as
\begin{equation}
  \label{eq:introSystem2}
  a_{1} x^3 +a_{2}x+ a_{3} y^2 = 0 \quad\text{and}\quad b_{1} x^3+b_{2} y^2+b_{3}=0
\end{equation}
over $\CC$ with variables $x,y$ and parameters $a_{i},b_{i}$ for $1\leq{i}\leq{3}$. It has $6$ solutions for a generic choice of parameters, i.e., for $a_{i}$ and $b_{i}$ in a dense open subset of $\mathbb{C}^{6}\subseteq \Spec(\mathbb{C}[a_{1},a_{2},a_{3},b_{1},b_{2},b_{3}])\eqqcolon Y$.  We refer to that number as its \emph{generic root count}, and it is the main object of interest for this paper. Two important examples of such parametrized families come from the theory of Newton-Okounkov bodies \cite{KavehKhovanskii2010,KavehKhovanskii2012} and chemical reaction networks \cite{Dickenstein16}. These systems can be obtained from generic systems as in \cref{eq:introSystem2} by imposing a set of \emph{algebraic relations} on the parameters, which we view as a Zariski closed subspace $Z$ of the parameter space $Y$. By restriction, every irreducible closed subspace $Z\subseteq Y$ gives a new parametrized system with its own (restricted) generic root count. For instance, the linear relations $a_{1}=a_{2}=b_{1}$ and $a_{3}=b_{2}=b_{3}$ give the Zariski closed subspace $Z=V(a_{1}-a_{2},a_{2}-b_{1},a_{3}-b_{2},b_{2}-b_{3})$ of $Y$, and the generic root count of the system in \cref{eq:introSystem2} over $Z$ is $2$. Establishing explicit combinatorial formulas for these generic root counts is one of the main goals of this paper.

These parametrized families of polynomial equations are ubiquitous in mathematics and beyond. They describe the $27$ lines on a smooth cubic surface, the dynamics of the Wnt signaling pathway \cite{GHRS16}, and the motion of the Gough-Steward platform \cite{SW05}. Many applications require solving polynomial systems, which can for example be done using homotopy continuation. For a single solution, Lairez has shown this to be possible in average polynomial time \cite{Lairez17}, thus solving Smale's 17th problem. If all solutions are needed, a major difficulty is predicting the number of solutions in the first place, which can be very high in theory but is often surprisingly small in practice. This forms one of the main motivations for finding combinatorial formulas for generic root counts in this paper.

An archetypal example of such a formula is given by \emph{B\'{e}zout's theorem}, which states that $n$ polynomials in $n$ variables of degrees $d_{1},\dots,d_{n}\in\ZZ_{\geq 0}$ over an algebraically closed field $K$ generically have $\prod_{i=1}^n d_{i}$ solutions in $K^{n}\subseteq \mathbb{A}^{n}_{K}=$ \linebreak $\Spec(K[x_{1},\dots,x_{n}])$.  For example, two bivariate polynomials of degree $3$ generically have $9$ solutions, which shows that generic instances of System~\eqref{eq:introSystem2} are not generic in the sense of B\'ezout.

Another example is the \emph{Bernstein-Kushnirenko theorem} \cite{Bernstein1975,Kushnirenko1976}, which states that $n$ Laurent polynomials in $n$ variables with fixed monomial supports $S_{1}, \dots, S_{n}\subseteq \mathbb{Z}^{n}$ generically have as many solutions in the torus $(K^\ast)^n\subseteq $\linebreak $\Spec(K[x_{1}^{\pm},\dots,x_{n}^{\pm}])$ as the (normalized) mixed volume of their Newton polytopes $\Conv(S_1), \dots, \Conv(S_n)\subseteq\RR^n$.  For example, the two bivariate polynomials in \cref{eq:introSystem2} with monomial supports $S_1=\{(2,0),(0,2),(1,0)\}$ and $S_1=\{(2,0), (0,2),$ $(0,0)\}$ generically have $6$ solutions, which is the mixed volume of their corresponding Newton polytopes, see \cref{fig:IntroExample}.

\begin{figure}[ht]
\centering
\begin{tikzpicture}
  \node (mixedSubdivision) at (-3.5,0)
  {
    \begin{tikzpicture}
      \fill[blue!20]
      (1,0) -- (3,0) -- (0,2) -- cycle;
      \fill[red!20]
      (3,0) -- (6,0) -- (3,2) -- cycle;
      \draw[blue!50!black,thick]
      (1,0) -- (3,0) -- (0,2) -- cycle
      (3,2) -- (0,4);
      \draw[red!50!black,thick]
      (3,0) -- (6,0) -- (3,2) -- cycle
      (0,2) -- (0,4);
      \fill
      (1,0) circle (2pt)
      (3,0) circle (2pt)
      (6,0) circle (2pt)
      (0,2) circle (2pt)
      (3,2) circle (2pt)
      (0,4) circle (2pt);
    \end{tikzpicture}
  };

  \node (tropicalIntersection) at (3.5,0)
  {
    \begin{tikzpicture}
      \coordinate (v1) at (-2,-1);
      \node[opacity=0.2,xshift=0.5mm,yshift=1mm] at (v1)
      {
        \begin{tikzpicture}[scale=0.3]
          \fill[blue]
          (1,0) -- (3,0) -- (0,2) -- cycle;
        \end{tikzpicture}
      };
      \draw[blue!50!black,thick]
      (v1) -- ++(2,3)
      (v1) -- ++(0,-1) node[right,yshift=2mm,font=\footnotesize] {$2$}
      (v1) -- ++(-2,-1);
      \fill[blue!50!black] (v1) circle (2pt);
      \coordinate (v2) at (0,0);
      \node[opacity=0.2,xshift=2mm,yshift=1mm] at (v2)
      {
        \begin{tikzpicture}[scale=0.3]
          \fill[red]
          (3,0) -- (6,0) -- (3,2) -- cycle;
        \end{tikzpicture}
      };
      \draw[red!50!black,thick]
      (v2) -- ++(1.333,2)
      (v2) -- ++(0,-2) node[right,yshift=2mm,font=\footnotesize] {$3$}
      (v2) -- ++(-4,0) node[above,xshift=2mm,font=\footnotesize] {$2$};
      \fill[red!50!black] (v2) circle (2pt);
      \fill[draw=black,fill=white] (-1.333,0) circle (3pt);
      \node at (-1.333,0)
      {
        \begin{tikzpicture}[scale=0.3]
          \draw[blue!50!black,densely dotted]
          (3,0) -- (0,2)
          (3,2) -- (0,4);
          \draw[red!50!black,densely dotted]
          (3,0) -- (3,2)
          (0,2) -- (0,4);
        \end{tikzpicture}
      };
      \node[anchor=south east,yshift=6mm,font=\scriptsize] at (-1.333,0)
      {
        $\textcolor{red!50!black}{2}\cdot \left|\det \left(
         \begin{smallmatrix}
             \textcolor{red!50!black}{-1} & \textcolor{blue!50!black}{2} \\
              \textcolor{red!50!black}{0} & \textcolor{blue!50!black}{3}
         \end{smallmatrix}\right)\right|$
      };
    \end{tikzpicture}
  };
\end{tikzpicture}
\caption{A mixed subdivision of the Minkowski sum of the Newton polytopes of $f_{1}$ and $f_{2}$. The white cell is the mixed cell, which has normalized volume $6$. This is also the tropical intersection number of the two dual tropical curves, which intersect in one point with multiplicity $6$. The integers next to the edges are their tropical multiplicities. }
\label{fig:IntroExample}
\end{figure}

A third prominent example can be found in the works of Kaveh and Khovanskii \cite{KavehKhovanskii2010,KavehKhovanskii2012}, who consider polynomials with fixed polynomial supports. Their generic root count is the birational intersection index, and in special cases it is the volume of the associated \emph{Newton-Okounkov body}. For instance, if we impose the linear relations $a_{1}=a_{2}$ and $b_{2}=b_{3}$ on the parametrized  System~\eqref{eq:introSystem2}, then we obtain two generic bivariate polynomials with polynomial support $\{x^2+x,y^2\}$ and $\{x^2,y^2+1\}$. In this case, the generic root count is the same as the one we obtain from the Bernstein-Kushnirenko theorem: $6$. However, if we further impose the linear relations $a_{1}=b_{1}$ and $a_{2}=b_{2}$ to go to the closed subspace $Z$, then the root count drops to $2$, so that it is {not} generic in the sense of Kaveh-Khovanskii.

This paper is a continuation of the aforementioned works, and our main goal is to obtain new formulas for generic root counts of more general classes of parametrized systems, including System~\eqref{eq:introSystem2} for linear subspaces of $Y$, but also systems arising in applications such as chemical reaction networks \cite{Dickenstein16}.  Our main tool is tropical geometry, which studies piecewise-linear objects arising from polynomial equations, and our formulas are given in terms of tropical intersection numbers.
One can regard our work as a generalization of the Bernstein-Kushnirenko-Theorem, as the mixed volume is a tropical intersection number of hypersurfaces \cite[Theorem 4.6.8]{MS15}.  The difference is that we make essential use of tropicalizations of higher codimension, relying on results by %
Osserman and Payne \cite{OssermanPayne2013} and Osserman and Rabinoff \cite{OR2013}.

Our general set-up is as follows:  Let $Y=\Spec(A)$ be a parameter space, which is integral and of finite type over a field $K$, let $T\coloneqq \Spec(A[x_{1}^{\pm},\dots,x_{n}^{\pm}]) \rightarrow Y$ be a relative torus over $Y$, and consider a set of closed subschemes $X_{i}\subseteq T$. One can informally view  the $X_{i}$ as families of closed subschemes of $\Spec(K[x_{1},\dots,x_{n}])$ parametrized over $Y$.  Let $X=\bigcap_i X_i$ be their intersection, which we assume to be generically zero-dimensional, although our techniques can also show that specific schemes are generically zero-dimensional.
For System~\eqref{eq:introSystem2}, the above translates to the following. We have %
the coordinate ring $A=\mathbb{C}[a_{1},a_{2},a_{3},b_{1},b_{2},b_{3}]$ with parameter space $Y=\Spec(A)$, %
the relative torus $T=\Spec(A[x^{\pm},y^{\pm}])$, and the two subschemes $X_1,X_2$ given by the two equations.  

There are two \emph{key questions} we wish to address in this paper:

\medskip

\begin{description}[leftmargin=5mm]
\item[Key Question 1]\label{enumitem:Question1}
  For which choices of parameters $P\in Y(L)$ with values in a non-trivially valued field extension $K\to L$ %
  is the generic root count of $X$ equal to the tropical intersection number of the fibers $X_{i,P}$?
\item[Key Question \customlabel{enumitem:Question2}{2}]
  With a view towards applications, which systems of hypersurfaces $X_i$ can be re-embedded so that the answer to Question \ref{enumitem:Question1} applies?
\end{description}

\medskip

We discuss our answers for these questions in Sections \ref{sec:IntroKQ1} and \ref{sec:IntroKQ2}, together with where they can be found in the paper. We have made an effort so that the corresponding passages can be read independently.

\subsection{Specialization and generization for tropical varieties}\label{sec:IntroKQ1}
Question~\ref{enumitem:Question1} is answered in Sections \ref{sec:FiberwiseTropicalizations} and \ref{sec:TropGeneric}.  We use the language of Berkovich spaces, as it offers a natural framework in which one can tropicalize over non-classical points of the parameter space $Y$.
Moreover, some of the technical results we require are only found in Berkovich theory, and many applications naturally have analytic parameter spaces, see \cref{rem:AnalyticFamilies}.

Our main answer is \cref{thm:MainThm1} in \cref{sec:GRCTropicalIntersectionNumber}, which states that the generic root count of $X$ equals the tropical intersection product of the $X_{i,P}$'s under two conditions:
\begin{enumerate}
\item the intersection of the tropicalizations of the $X_{i,P}$'s is bounded,
\item the $X_i$'s are \emph{tropically flat} around $P$.
\end{enumerate}
The boundedness of the intersection is a weaker form of transversality, and we have already seen that some sort of transversality condition may be necessary. For example, on the closed subspace $Z\subseteq Y$ considered before, the generic root count of System~\eqref{eq:introSystem2} does not equal the tropical intersection number of the two hypersurfaces, and, by the Transverse Intersection Theorem~\cite[Theorem~3.4.12]{MS15}, this means that the tropicalizations of the two hypersurfaces cannot intersect transversally.

Tropical flatness is the main topic of \cref{sec:FiberwiseTropicalizations}, and we use it in this paper to connect properties of various fibers, as in classical algebraic geometry. %
Informally, a family of schemes $X\to Y$ is flat around a point $P\in Y$ if the fibers vary nicely around $P$.  This is for instance exemplified by \cite[Chapter III, Corollary 9.10]{Hartshorne1977}, which says that the Hilbert polynomials of the fibers of a closed subscheme of relative projective space are the same if a family is flat. Similarly, if $X$ is tropically flat around a point $P\in Y$, then the tropicalizations of the fibers %
vary nicely around $P$.  As before, this implies that information which is preserved under tropicalization also varies nicely around $P$. This includes invariants such as the dimension, but also quantities such as tropical intersection numbers that can be used to obtain generic root counts.

The following examples illustrate the points at which the family is tropically flat in two important cases:
\begin{enumerate}
\item If $X_{i}=V(f)$ is a parametrized hypersurface, then it is tropically flat over all parameters for which no coefficient of $f$ vanishes, see \cref{cor:TropConstant}.  The two parametrized hypersurfaces of System~\eqref{eq:introSystem2} for instance are tropically flat over all parameters in the torus $(K^\ast)^2$.
\item If $X_i$ is a parametrized linear space, then it is tropically flat at all parameters around which its matroid does not change, see \cref{lem:FlatnessMatroids}.
\end{enumerate}
Under our assumptions on $Y$, we show in \cref{thm:TropicalGenericFlatness2} that any $X_{i}$ is tropically flat over a dense open subset of the parameter space, giving an analogue of Grothendieck's generic flatness theorem \cite[Th\'{e}or\`{e}me 6.9.1]{EGA42} \cite[\href{https://stacks.math.columbia.edu/tag/052A}{Proposition 052A}]{stacks-project}. Conversely, if $X_{i}$ is tropically flat at $P$, then many properties of $X_{i,P}$ that are preserved in its tropicalization hold for $X_i$ generically.
  For example, \cref{pro:DimensionsTransverseIntersections} shows that the dimension of $X_{i,P}$ is the generic dimension of $X_i$, and \cref{thm:MainThm1} shows that the tropical intersection number of the $X_{i,P}$ is the generic root count of $X$. %
The necessity of tropical flatness for \cref{thm:MainThm1} is illustrated in Examples
\ref{exa:TropicallyFlatIllustration} and \ref{ex:nonSquare}.

Finally, we introduce the notions of \emph{torus-equivariance} and \emph{parametric independence} in \cref{sec:TorusEquivariant}.  Together, they provide a simple criterion for the existence of transverse intersections.  To be precise, in a torus-equivariant family $X_i$, one can translate the fibers $X_{i,P}$ torically using a torus action on the parameter space. Tropically, this means that we can freely translate the $\trop(X_{i,P})$'s by adjusting the parameter $P$.  If the $X_i$'s are furthermore parametrically independent, then these translations can be done independent of each other.  Combining both notions yields an easy criterion for which \cref{thm:MainThm1} holds.

\subsection{Generic root counts of square systems}\label{sec:IntroKQ2}
A partial answer to Question~\ref{enumitem:Question2} can be found in Sections \ref{sec:modifications} and \ref{sec:linearDependencies}.  Systems that we discuss in-depth include  generalizations of the systems studied by Kaveh and Khovanskii, and many systems arising from applications such as chemical reaction networks, the Kuramoto model, and Duffing oscillators.

In \cref{sec:modifications}, we introduce the notion of a \emph{tropically rectifiable} square system, which are systems that can be reembedded to produce transverse intersection.   These reembeddings are also called \emph{tropical modifications}, and they are commonly used to ``repair'' tropicalizations \cite{CuetoMarkwig2016}.  We prove in \cref{thm:GRCModification} that the resulting tropical intersection number $\trop(\hat X_{\lin,P})\cdot \trop(\hat X_{\nlin,P})$ for generic $P$ equals the generic root count in an open set of the torus $T$.  The open subset arises from intrinsic obstructions similar to those in the works of Kaveh and Khovanskii \cite{KavehKhovanskii2010,KavehKhovanskii2012}.  We then provide descriptions of said open set and when it is equal to $T$. Moreover, we show that $P$ is generic if the matroid of $\hat X_{\lin,P}$ is generic.  Using a result by Jensen and Yu \cite[Corollary 5.2]{AndersJosephine2016}, this gives a decomposition of the generic root count into mixed volumes.

In \cref{sec:linearDependencies}, we discuss linearly parametrized systems and simplify the results from \cref{sec:modifications} for two classes of parametrized systems.

In \cref{sec:verticalDependencies}, we focus on \emph{systems with vertical parameter dependencies}, or \emph{vertical systems} for short, which for example arise from the steady state equations of chemical reaction networks. In \cref{prop:verticalDependenciesRootCount}, we express their generic root count as the tropical intersection product of a tropical linear space and a tropicalized binomial variety.

In \cref{sec:horizontalDependencies}, we focus on \emph{systems with horizontal parameter dependencies}, or \emph{horizontal systems} for short, which for example include systems studied by Kaveh and Khovanskii. In \cref{prop:horizontalDependenciesRootCount}, we express their generic root count of a tropical variety dependent on the polynomial support and tropical hyperplanes.  This in particular gives a formula for the birational intersection indices from \cite{KavehKhovanskii2010}, and hence also for the volume of Newton Okounkov bodies from \cite{KavehKhovanskii2012}, in terms of tropical intersection numbers.  We demonstrate our technique in three examples:  the stationary equations of the Kuramoto model \cite{Kuramoto2019}, Duffing oscillators \cite{BMMT2022} and steady-state equations of chemical reaction networks \cite{Dickenstein16}.

\subsection*{Acknowledgments}
The authors would like to thank the anonymous referees for %
carefully reading an initial version of this paper and for providing numerous comments that greatly helped in improving the exposition and readability.  
The authors would like to thank Mateusz Michałek and Tianran Chen for providing valuable feedback and examples. The authors would like to thank Alex Fink and Felipe Rinc\'on for helpful discussions on tropical linear spaces. The authors would further like to thank the participants of the Durham tropical geometry seminar, Jeffrey Giansiracusa, James Maxwell, and Stefano Mereta, for many productive discussions. Both authors are supported by UK Research and Innovation under the Future Leaders Fellowship program (MR/S034463/2). The first author is moreover supported by a JSPS Postdoctoral Fellowship for Research in Japan (ID No P23769) and KAKENHI 23KF0187 as a Postdoctoral Fellow at Tohoku University and the University of Tsukuba.

\section{Preliminaries}\label{sec:preliminaries}
In this section, we briefly review some basic concepts and fix our notation.  In particular, we define generic root counts, fiberwise tropicalizations and local tropical bases.

\subsection{Generic properties}\label{sec:genericProperties}
In this section, we fix the main setting of our paper and introduce the main properties of interest.

\begin{notation}\label{con:mainConvention}
  For the remainder of the paper, let $K$ be an algebraically closed field with a non-archimedean absolute value $|\cdot|_K\colon K\rightarrow\RR_{\geq 0}$.

  Let $A$ be a $K$-algebra of finite type and let $Y=\Spec(A)$ be its associated scheme.  We will refer to $A$ as the \emph{parameter ring} and $Y$ as the \emph{parameter space}.  We will assume $Y$ to be integral and use $\eta$ to denote its unique generic point.  Moreover, abbreviating $A[x^\pm]\coloneqq A[x_1^\pm,\dots,x_n^\pm]$, let $T\coloneqq \Spec(A[x^\pm])$ be an $n$-dimensional torus over $Y$, and denote the projection by $p\colon T\rightarrow Y$. If $P\in Y$, then we write $k(P)$ for the residue field of $P$. There is a natural ring homomorphism $A\to k(P)$ and for $f\in{A}$, we write $f(P)\in k(P)$ for the image of $f$ under this homomorphism.

  Let $B=A[x^\pm]/I$ for some ideal $I\subseteq A[x^\pm]$. We identify $X\coloneqq \Spec(B)$ with a closed subspace of $T=\Spec(A[x^{\pm}])$ through the closed immersion induced by the natural ring homomorphism $A[x^{\pm}]\to A[x^{\pm}]/I$. We will not distinguish between $X$ and its image in $T$.  %
  By composing the inclusion $X\to T$ with $p$, we obtain a morphism $p_{X}\colon X\to Y$, and we will often abbreviate $p=p_X$ if the context is clear.
  For a point $P\in Y$, we denote $A[x^\pm]_P\coloneqq A[x^\pm]\otimes_K k(P)$ and $T_P\coloneqq \Spec(A[x^\pm]_P)$ as well as $B_P\coloneqq B\otimes_K k(P)$ and $X_P\coloneqq \Spec(B_P)$. We refer to $X_P$ as the \emph{specialization} of $X$ at $P$.
\end{notation}

\begin{definition}
  The \emph{root count} of $X$ at $P\in Y$, denoted by $\ell_{X,P}\in\ZZ_{\geq 0}\cup\{\infty\}$, is the $k(P)$-vector space dimension  of $B_{P}$.
  The \emph{generic root count} of $X$ is the root count $\ell_{X,\eta}$ at the generic point $\eta\in Y$.  We say $X$ is \emph{generically finite} if $\ell_{X,\eta}<\infty$.
\end{definition}

\begin{example}\label{ex:runningExampleGenericRootCount}
  Let $K=\mathbb{C}\{\!\{t\}\!\}$ be the field of complex Puiseux series, $Y=\mathrm{Spec}(A)$ for $A=K[a_1,a_2,a_3,b_1,b_2,b_3,c_1,c_2]$, and $T=\mathrm{Spec}(A[x^{\pm},y^{\pm},z^{\pm}])$.  Consider the subscheme $X \subseteq T$  given by the ideal
  \begin{equation*}
    I\coloneqq \big( a_1x^2+a_2y^2+a_3y,\; b_1x^2+b_2y^2+b_3z,\; c_1z+c_2 \big).
  \end{equation*}

  One can show that the generic root count $\ell_{X,\eta}$ is $4$, while over $P\in Y$ with $(a_1b_2-a_2b_1)(P)=0$ the root count $\ell_{X,P}$ drops to $2$.  Over $P\in Y$ with $c_1(P)=0$, $c_2(P)=0$, or $a_1(P)=0=a_2(P)$ the root counts drops to $0$.
\end{example}

Besides generic root counts, other important generic properties of $X$ are:

\begin{definition}
  \label{def:genericProperties}
  We say that $X$ is
  \begin{enumerate}
  \item \emph{generically Cohen-Macaulay}, if $X_\eta$ is Cohen-Macaulay,
  \item \emph{generically pure}, if $X_\eta$ is pure,
  \item \emph{generically $d$-dimensional}, if $X_\eta$ is $d$-dimensional,
  \item \emph{generically $k$-codimensional}, if $X_\eta$ is $k$-codimensional in $T_\eta$ or equivalently $X$ is generically $(n-k)$-dimensional.
  \end{enumerate}
\end{definition}

While all generic properties are defined via the generic fiber $X_\eta$, note that they indeed reflect the behavior of $X$ over a dense open subset of $Y$.

\begin{lemma}
  \label{lem:genericRootCount}
  Let $X$ be generically finite with generic root count $\ell_{X,\eta}=k$. Then there is a dense open subset $U\subseteq Y$ such that $\ell_{X,P}=k$ for all $P\in{U}$.
\end{lemma}
\begin{proof}
  Let $A$ and $B$ be the coordinate rings of $Y$ and $X$ respectively. By Grothendieck's generic freeness theorem \cite[\href{https://stacks.math.columbia.edu/tag/051S}{Lemma 051S}]{stacks-project}, we can find an $f\in{A}$ such that $B_{f}\cong A^{k}_{f}$. Here $B_{f}$ and $A_{f}$ are the localizations of $B$ and $A$ with respect to $f$. By taking $U=D(f)$, we then directly find the desired result.
\end{proof}

\begin{lemma}\label{lem:GenericDimensions}
  If $X$ has generic dimension $d$, then there is a dense open subset $U\subseteq Y$ such that $X_P$ has dimension $d$ for all $P\in U$. Similarly, if $X$ has generic codimension $k$, then there is a dense open subset $U\subseteq Y$ such that $X_P$ is of codimension $k$ for all $P\in U$.
\end{lemma}
\begin{proof}
    This follows from \cite[\href{https://stacks.math.columbia.edu/tag/05F7}{Lemma 05F7}]{stacks-project}.
\end{proof}

\begin{lemma}\label{lem:GenericallyCM}
  If $X$ is generically Cohen-Macaulay, then there is a dense open subset $U\subseteq{Y}$ such that the restricted morphism $p_X|_{p_X^{-1}(U)}\colon p_X^{-1}(U)\to{U}$ is Cohen-Macaulay.
\end{lemma}
\begin{proof}
  We can assume by generic flatness that $p$ is flat. Consider the open subset $W$ from \cite[\href{https://stacks.math.columbia.edu/tag/045U}{Lemma 045U}]{stacks-project}. Its complement $Z=X\backslash{W}$ is closed and the fiber of $Z\to{Y}$ over the generic point is empty by assumption. This gives an open set $U_{\eta}$ containing $\eta$ for which  $p^{-1}(U_{\eta})\cap{Z}=\emptyset$ by \cite[\href{https://stacks.math.columbia.edu/tag/02NE}{Lemma 02NE}]{stacks-project}. The induced morphism $p^{-1}(U_{\eta})\to{U_{\eta}}$ is automatically Cohen-Macaulay.
\end{proof}

\begin{lemma}\label{lem:LemmaPurity}
  If $X$ is generically pure, then there is a dense open subset $U\subseteq{Y}$ such that $X_P$ is pure for all $P\in Y$.
\end{lemma}
\begin{proof}
  We follow the proof of \cite[\href{https://stacks.math.columbia.edu/tag/055A}{Lemma 055A}]{stacks-project}, where the notation $f$ is used for our morphism $p$. %
  By \cite[\href{https://stacks.math.columbia.edu/tag/0551}{Lemma 0551}]{stacks-project}, we can find a dense open subset $V\subseteq Y$ and a surjective finite \'{e}tale map $g\colon Y'\to V$ with induced commutative diagram
  \begin{equation*}
    \begin{tikzcd}
      X'=X\times_{V}Y' \arrow{r}{g'} \arrow{d}{p'}  & X_{V}=V\times_{Y}X \arrow{d}  \arrow{r} & X \arrow{d}{p}\\
      Y' \arrow{r}{g}& V \arrow{r}{} & Y
    \end{tikzcd}
 \end{equation*}
  such that:
  \begin{itemize}
  \item The squares are Cartesian.
  \item $Y'$ is irreducible and affine.
  \item The morphisms $g$ and $g'$ are surjective finite \'{e}tale.
  \item All irreducible components of the generic fiber of $p'$ are geometrically irreducible.
  \end{itemize}
  There is one point in $Y'$ lying over $y$, which we denote by $y'$.
  As in the proof of \cite[\href{https://stacks.math.columbia.edu/tag/055A}{Lemma 055A}]{stacks-project}, we may assume that the number of geometrically irreducible components $X'_{i,y'}$ of the fibers $X'_{y'}$ is constant over a dense open $V'\subseteq Y'$.  These components are necessarily pure of dimension $d$ again, see \cite[\href{https://stacks.math.columbia.edu/tag/04KX}{Lemma 04KX}]{stacks-project} and \cite[\href{https://stacks.math.columbia.edu/tag/07NB}{Section 07NB}]{stacks-project}. We set $Y'=V'$ and $V=g(V')$, which is again open since \'{e}tale maps are open.

  Let $X'_{i}$ be the closure of $X'_{i,y}$ in $X'$.  The proof of \cite[\href{https://stacks.math.columbia.edu/tag/055A}{Lemma 055A}]{stacks-project} then shows that we can find an open $V'\subseteq{Y'}$ such that the fibers of the $X'_{i}$'s over $V'$ are geometrically irreducible and $X'=\bigcup_i X'_{i}$ over $V'$. We can furthermore assume by shrinking $V'$ that all the fibers of the $X'_{i}$'s are of dimension $d$ by \cite[\href{https://stacks.math.columbia.edu/tag/05F7}{Lemma 05F7}]{stacks-project}. We again set $Y'=V'$ and $V=g(V')$.

  We now have a morphism $g\colon Y'\to V$ such that the base change $p'$ of $p$ is relatively pure of dimension $d$. Consider the proof of \cite[\href{https://stacks.math.columbia.edu/tag/0556}{Lemma 0556}]{stacks-project}, giving a bijection between the geometrically irreducible components of $p$ over $V$ and $p'$ over $V'$. Under this bijection the Krull dimensions of the irreducible components are not changed, since they are obtained by base change over a field extension. We thus obtain the desired result.
\end{proof}

\begin{remark}
    We note here that generically Cohen-Macaulay morphisms are almost generically pure. %
Namely, if $X_{\eta}$ is %
Cohen-Macaulay, then there exist open and closed subschemes $X_{\eta,i}$ with $\bigsqcup_{i=0}^{r} X_{\eta,i}=X_{\eta}$ such that the $X_{\eta,i}$ are pure of dimension $i$ by \cite[\href{https://stacks.math.columbia.edu/tag/02NM}{Lemma 02NM}]{stacks-project}. %
This is the reason for the additional purity condition in \cref{thm:MainThm1}.   %
\end{remark}

An important case for us are relative global complete intersections \cite[\href{https://stacks.math.columbia.edu/tag/00SP}{Definition 00SP}]{stacks-project}. We recall their definition in the generic case here. %

\begin{definition}
We say that $X=\Spec(B)$ is a \emph{generic global complete intersection} of dimension $d$ if $X_{\eta}$ is of dimension $d$ and there are $f_1,\dots,f_{n-d}\in A[x_{1}, \dots, x_{n}]$ such that $B\otimes_{A}K(A)\cong K(A)[x_{1},\dots ,x_{n}]/(f_{1},\dots ,f_{n-d})$. If $d=0$, then we say that $X$ is \emph{generically square}.

\end{definition}

\begin{lemma}\label{lem:GlobalCompleteIntersection}
Suppose that $X$ is a generic global complete intersection. Then there is a dense open subset $U\subseteq Y$ such that $p^{-1}(U)\to U$ is a relative global complete intersection.
\end{lemma}
\begin{proof}
    This follows from \cite[\href{https://stacks.math.columbia.edu/tag/00ST}{Lemma 00ST(2)}]{stacks-project}.
\end{proof}

Sparse polynomial systems with fixed monomial supports form an important class of examples.  These are generic global complete intersections, but not global complete intersections.

\begin{definition}
  \label{def:UniversalFamilies}
  Let $n,k\in\mathbb{Z}_{\geq 0}$, and $[k]\coloneqq\{1,\dots,k\}$. A \emph{fixed monomial support} is a tuple $S=(S_1,\dots,S_k)$ of finite subsets $S\subseteq \ZZ^n$.  We will generally assume that $|S_i|>1$.
  Let $A=K[c_{i,\alpha}| i\in [k], \alpha\in S_i]$, and consider the (Laurent) polynomials
  \begin{equation*}
    f_{i}\coloneqq\sum_{\alpha\in S_i} c_{i,\alpha} x^{\alpha}\in A[x_1^{\pm},\dots,x_n^{\pm}] \qquad\text{for }i\in [k].
  \end{equation*}
  We refer to $X\coloneqq\mathrm{Spec}(A[x_1^{\pm},\dots,x_n^{\pm}]/(f_{1},\dots,f_{k}))$ as the \emph{universal family} with monomial support $S$. %
\end{definition}

Suppose $k=n$. Then there are 
 simple combinatorial conditions on the monomial supports that characterize or guarantee whether the universal family is generically non-empty and thus square, see %
 \cite[Theorem 2.2]{Esterov2019} or \cite[Theorem 3]{Yu16}. 

\subsection{Root counts and specialization}
In this section, we prove a folklore result on root counts that is well known in applied algebraic geometry. It says that the root count of a generically finite $X$ can only go down under specialization, provided the morphism $p\colon X\rightarrow Y$ is sufficiently flat.  The fact that this is not generally true can be seen in the following two examples where $p$ is not flat:

\begin{example}
  \label{exa:CounterExampleRC}
  In the following two examples, the generic root count $\ell_{X,\eta}$ is lower than the specialized root count $\ell_{X,P}$ for some $P\in Y$. These examples are translated versions of well-known affine examples. The translations make sure the interesting phenomena occur in the torus, in keeping with  
  \cref{con:mainConvention}. %
  \begin{enumerate}[leftmargin=*]
  \item Take $A=\CC[a]$, $Y=\Spec(A)$, and $X=\Spec(A[x^\pm]/(a(x-1),(x-1)^d))$ for $d>1\in \mathbb{Z}_{\geq{1}}$. Its generic root count is $1$, however its root count at $a=0$ is $d$.
  \item Let $A=\CC[a^{\pm},b^{\pm}]/((a-1)^2-(b-1)^3-(b-1)^2)$, $Y=\Spec(A)$, and $X=\Spec(A[x^\pm]/(x^2-b,(a-1)x-(b-1))$. Note that $X$ is the normalization of $Y$, so that $p\colon X\to Y$ is birational. This means that the generic root count is $1$. Over the point $a=1=b$, we however have two points.
  \end{enumerate}
\end{example}

When it comes to the root count of specializations, we are primarily interested in points $P\in Y$ at which the specialization $X_P$ is finite.  These are exactly the points for which the morphism $p\colon X\rightarrow Y$ is quasi-finite at every point in the preimage.

\begin{definition}
  \label{def:quasiFiniteLocus}
  We denote the \emph{quasi-finite locus} of $X$ by
  \begin{equation*}
    \mathrm{QF}(p)\coloneqq\{P\in{Y}\mid X_P \text{ finite}\}.
  \end{equation*}
\end{definition}

We now show that root counts decrease under specialization, assuming the morphism $X\rightarrow Y$ is sufficiently flat.

\begin{lemma}\label{lem:GRCDecreases2}
  Let $X$ be generically finite, and suppose that $p^{-1}(\mathrm{QF}(p))$ is contained in the flat locus of $p$.  Then $\ell_{X,\eta}\geq{\ell_{X,y}}$ for all $y\in{\mathrm{QF}(p)}$.
\end{lemma}
\begin{proof}
 Let $\mathcal{O}_{Y,y}$ be the local ring of $Y$ at $y$, and let $R$ be a valuation ring dominating $\mathcal{O}_{Y,y}$ \cite[\href{https://stacks.math.columbia.edu/tag/00IA}{Lemma 00IA}]{stacks-project}. In particular, the field of fractions of $R$ is $K(Y)$. Note that we can also assume that $R$ is a discrete valuation ring by \cite[\href{https://stacks.math.columbia.edu/tag/00PH}{Lemma 00PH}]{stacks-project}. We write $k$ for the residue field of $R$. The inclusion $\mathcal{O}_{Y,y}\subseteq R$ gives a morphism $P\colon\Spec(R)\to Y$ such that the image of the closed point is $y$, and the image of the generic point is the generic point of $Y$. We consider the Henselization $R^{h}$ of $R$, which is again a local normal domain by \cite[\href{https://stacks.math.columbia.edu/tag/06DI}{Lemma 06DI}]{stacks-project}. By composing $P\colon\Spec(R)\to Y$ with the faithfully flat map $\Spec(R^{h})\to \Spec(R)$, we may assume that $R$ is Henselian.     %

 Set $M\coloneqq B\otimes_A R$. %
 This is a quasi-finite and flat $R$-algebra by assumption. %
 Since $R$ is Henselian, we can write
 \begin{equation*}
     M=M_{\mathrm{fin}}\times{M_{\mathrm{nfin}}},
 \end{equation*}
 where $M_{\mathrm{fin}}$ is finite and $M_{\mathrm{nfin}}\otimes_{R}{k}=(0)$, see \cite[\href{https://stacks.math.columbia.edu/tag/04GG}{Lemma 04GG, part 13}]{stacks-project}.  Moreover, $M_{\mathrm{fin}}$ is flat over $R$, so it is free.  We have $\ell_{X,\eta}=\mathrm{rank}(M_{\eta})=\mathrm{rank}(M_{\mathrm{fin},\eta})+\mathrm{rank}(M_{\mathrm{nfin},\eta})$ and $\mathrm{rank}(M_{s})=\mathrm{rank}(M_{\mathrm{fin},s})$. But $\ell_{X,y}=\mathrm{rank}(M_{\mathrm{fin},s})=\mathrm{rank}(M_{\mathrm{fin},\eta})$, so we obtain the statement of the lemma.
\end{proof}

As an immediate corollary, root counts decrease for square systems:

\begin{corollary}\label{cor:GRCDecreases2}
  Suppose that $X$ is generically square. Then $p\colon X\to Y$ is flat at the quasi-finite points of $p$.
  In particular, root counts go down under specialization.
\end{corollary}
\begin{proof}
This follows from \cite[\href{https://stacks.math.columbia.edu/tag/00ST}{Lemma 00ST}]{stacks-project}, \cite[\href{https://stacks.math.columbia.edu/tag/00SW}{Lemma 00SW}]{stacks-project} and \cref{lem:GRCDecreases2}.
\end{proof}

We close this subsection with a few remarks:

\begin{remark}\label{rem:rootCounts}\
  \begin{enumerate}[leftmargin=*]
  \item First, note that points $y\in Y$ in the scheme-theoretic language are prime ideals and need not be points on the variety.  Hence \cref{lem:GRCDecreases2} also shows that for sufficiently flat families generic root counts decrease under specialization.  The root count $\ell_{X,y}$ equals the generic root count of
    \begin{math}
      X\cap p^{-1}(V(y)),
    \end{math}
    where $V(y)=\overline{\{y\}}$ denotes the corresponding closed subscheme of $Y$.
  \item Second, the proof of \cref{lem:GRCDecreases2} in fact shows that the non-flatness is the only thing that can cause errant behavior. Namely, if we assume that $R$ is a discrete valuation ring, then $M_{\mathrm{fin}}\cong R^{m}\times M_{\mathrm{fin},\mathrm{tors}}$ by the structure theorem of finitely generated modules over principal ideal domains. There is only a local increase in the finite part of the root count if $M_{\mathrm{fin},\mathrm{tors}}$ is non-trivial, which is equivalent to $M_{\mathrm{fin}}$ being non-flat. Indeed, in this case $M_{\mathrm{nfin}}$ is a finite number of copies of the fraction field of $R$ by \cite[\href{https://stacks.math.columbia.edu/tag/02ML}{Lemma 02ML}]{stacks-project}, so it is already flat over $R$.
  \item Lastly, note that for other notions of ``root count'', we can weaken the conditions under which root counts decrease under specialization:

    For example, if one ignores the ramification multiplicities of the points in the fibers $X\to Y$ (the resulting number is sometimes called the \emph{separable degree}), then, by \cite[Th\'{e}or\`{e}me 18.10.16]{EGA44}, the root counts go down under specialization for any generically finite $X$ as long as $Y$ is geometrically unibranch \cite[\href{https://stacks.math.columbia.edu/tag/0BQ2}{Definition 0BQ2}]{stacks-project} (e.g., $Y$ normal).  In \cref{exa:CounterExampleRC} (1), the separable degree is $1$ everywhere. However, in \cref{exa:CounterExampleRC} (2), the separable degree goes up because $Y$ is not geometrically unibranch.
  \end{enumerate}
\end{remark}

\subsection{A classification of the possible root counts}\label{sec:degeneracyGraph}
We now show how \cref{lem:GRCDecreases2} leads to a natural finite set of points $y_{i,j}$ with subvarieties $Z_{i,j}$ that %
quantifies the different root counts of a parametrized system if $X$ is square.  Finding more specific results on this set of $y_{i,j}$ %
seems pertinent to the problem of determining generic root counts.

\begin{lemma}\label{lem:QuasiFiniteLocusOpen}
  Suppose that $X$ is generically square. Then $\mathrm{QF}(p)$ is open and $p^{-1}(\mathrm{QF}(p))\to \mathrm{QF}(p)$ is flat.
\end{lemma}
\begin{proof}
  By \cref{cor:GRCDecreases2}, $p$ is %
  flat at every point $x$ lying over $y\in{\mathrm{QF}(p)}$. %
  The inverse image of $\mathrm{QF}(p)$ in $X$ is thus in the open flat locus $U$. Consider the restriction $p_{U}$ of $p$ to $U$, which is a flat morphism. By \cite[\href{https://stacks.math.columbia.edu/tag/02NM}{Lemma 02NM}]{stacks-project}, the locus of relative dimension zero of $p_{U}$ is open in $U$.
  The image of this locus under $p_{U}$ is exactly $\mathrm{QF}(p)$. Since flat morphisms are open, we conclude that $\mathrm{QF}(p)$ is open.
\end{proof}

Consider the morphism of schemes $p^{-1}(\mathrm{QF}(p))\to \mathrm{QF}(p)$ induced from \cref{lem:QuasiFiniteLocusOpen}.
By \cref{cor:GRCDecreases2}, root counts decrease under specialization, so that $\ell_{X,y}\leq\ell_{X,\eta}$ for every $y\in{\mathrm{QF}(p)}$. The morphism $p^{-1}(\mathrm{QF}(p))\to \mathrm{QF}(p)$ thus has universally bounded fibers \cite[\href{https://stacks.math.columbia.edu/tag/03J4}{Definition 03J4}]{stacks-project}. %
We now apply %
\cite[\href{https://stacks.math.columbia.edu/tag/07RY}{Lemma 07RY}]{stacks-project} %
and find that there are reduced closed subschemes
\begin{equation*}
    \emptyset=Z_{-1}\subseteq{}Z_{0}\subseteq{Z_{1}}\subseteq\dots\subseteq{Z_{\ell_{X,\eta}}}=\mathrm{QF}(p)
\end{equation*}
such that %
$$Z_{i}\backslash{Z_{i-1}}=\{P\in \mathrm{QF}(p):\ell_{X,P}=i\}.$$ %
Every $Z_{i}\backslash Z_{i-1}$ has finitely many generic points %
which we denote by $y_{i,j}$. Note that if $Z_{i}=Z_{i-1}$, then there are no such generic points.  %

We give an informal interpretation of these points $y_{i,j}$. Each $y_{i,j}$ gives a new parametrized system over which the generic root count is $i$. Indeed, we can consider the closure $Z_{i,j}=\overline{\{y_{i,j}\}}$ with its reduced induced subscheme structure, and then take the base change $X\times_{Y}Z_{i,j}\to Z_{i,j}$, which has generic root count $i$. If we view points as prime ideals and thus as collections of relations, then these points can be seen as containing a \emph{minimal} number of relations such that the root count becomes  exactly $i$. That is, any other prime ideals in $Z_{i}\backslash Z_{i-1}$ will also give rise to the same root count, and they contain the prime ideals corresponding to the $y_{i,j}$'s. Suppose now that $y_{i,j}$ admits a specialization %
$y_{i-1,k}$, %
so that $y_{i-1,k}\in Z_{i,j}$. Then this gives a minimal way to change the root count from $i$ to $i-1$. Suppose for instance that $X$ and $Y$ are defined over $\mathbb{C}$. We can then consider the $Z_{i,j}$'s as subvarieties %
of the parameter space, and having a specialization of the form above means that we can find a path $\gamma\colon [0,1]\to Z_{i,j}(\mathbb{C})$ such that $\gamma(t)$ is not in $Z_{i-1,j}(\mathbb{C})$ for all $t\neq{1}$ and $j$, but $\gamma(1)\in Z_{i-1,k}$. Here we used the fact that the complex points of an irreducible (and thus connected) variety form a connected space with respect to the usual complex topology. The $y_{i,j}$'s and corresponding $Z_{i,j}$'s can thus be seen as giving a road map for the different generic root counts that occur for $p:X\to Y$ when moving around the parameter space.

Combinatorially speaking, we can define this road map as follows. Viewing a scheme as a poset via specialization, we endow the set of $y_{i,j}$
  with the induced poset structure. The Hasse diagram of this poset then defines a finite graph that provides a pictorial representation of the various possible root counts and the ways one can navigate between them. Determining this graph explicitly seems challenging, even in small examples.

\begin{remark}
  Consider the square universal family $X$ with fixed monomial supports $M_{i}$ from \cref{def:UniversalFamilies}. By the BKK-Theorem (see \cref{cor:BKK}), the generic root count is the corresponding mixed volume. In Sections \ref{sec:modifications} and \ref{sec:linearDependencies} we will study the root counts of $X$ over certain non-generic points $P\neq\eta$ of $Y$. In terms of the language introduced above, if we know that the introduced tropical intersection number is lower than the mixed volume, then this shows the existence of a set of non-trivial $y_{i,j}$. Note however that we do not show that our prime ideals are minimal in the sense discussed above.   %
\end{remark}

\subsection{Berkovich spaces}\label{sec:BerkovichSpaces}
Next we go over the basics of Berkovich analytifications of schemes.  We will use $Y$ to denote the affine scheme, though our definition will of course apply to both $X$ and $Y$ from \cref{con:mainConvention}.  For more details on Berkovich spaces, we refer the reader to \cite{Gubler2013}, \cite{Temkin15} or \cite{BPR2016}.

\begin{definition}
  \label{def:analytification}
  The \emph{Berkovich analytification} of $Y$ is the set
  \begin{equation*}
    Y^{\mathrm{an}}\coloneqq \left\{ P = (P',|\cdot |_P) \bigmid
      \begin{array}{c}
        P'\in Y \text{ and } |\cdot|_P\colon k(P')\rightarrow\mathbb{R}_{\geq 0} \text{ an absolute value}\\[0.5mm]
        \text{ on the residue field } k(P') \text{ extending } |\cdot|_K
      \end{array}
    \right\}.
  \end{equation*}
  There is a natural forgetful map $\pi\colon Y^{\mathrm{an}}\to{Y}$ mapping $P=(P',|\cdot|_P)$ to $P'$, and given $P=(\pi(P),|\cdot|_P)\in Y^\an$ and $f\in A$ we will generally write %
  \begin{enumerate}
  \item $k(P)$ for $k(\pi(P))$,
  \item $f(P)$ for $f(\pi(P))\in k(\pi(P))$,
  \item $|f(P)|$ for $|f(\pi(P))|_P\in \RR_{\geq 0}$,
  \item $\mathrm{val}(f(P))$ for $-\mathrm{log}|f(P)|$, where $\log(\cdot)$ is the natural logarithm.  
  \end{enumerate}
  We say that $P\in Y^\an$ is \emph{rational} if the induced map $K\to{k(P)}$ is an isomorphism.\linebreak
  The valuation topology on $Y^{\mathrm{an}}$ is generated by the sets
  \begin{equation*}
    \mathbb{B}(r_{1},r_{2},f)=\{P\in Y^\an\mid r_{1}<|f(P)|_P<r_{2}\} \quad\text{for } 0\leq r_1<r_2 \text{ and } f\in A,
  \end{equation*}
  and $\pi$ is continuous with respect to the valuation topology on $Y^{\an}$ and the Zariski topology on $Y$.

  Note that any morphism of schemes $p\colon X\rightarrow Y$ induces a morphism on the analytifications $p^{\mathrm{an}}\colon X^{\mathrm{an}}\to{Y^{\mathrm{an}}}$ as follows:
  For any $P\in X^{\an}$, we have an injection of residue fields $k(p(\pi(P)))\to k(\pi(P))$ and we define $p(P)\in Y^{\an}$ to be the point $p(\pi(P))\in Y$ together with the restriction of the absolute value $|\cdot|_P$ on $k(\pi(P))$ to $k(p(\pi(P)))$. The induced map $p^{\an}\colon X^{\an}\to Y^{\an}$ is continuous with respect to the valuation topology on both spaces.
\end{definition}
\begin{remark}
    Although one usually assumes in the theory of Berkovich spaces that $K$ is complete, the above definition also works when $K$ is not complete. We will be explicit about requiring $K$ to be complete when citing results from the literature on Berkovich spaces.
\end{remark}

\cref{def:analytification} is taken from \cite{Nicaise2016}, and it is equivalent to the definition using multiplicative seminorms found in \cite{Gubler2013,Temkin15,BPR2016} by \cite[Remark~2.2]{Gubler2013}.

One can also think of points in $Y^{\an}$ as equivalence classes of $L$-valued points of~$Y$, where $L$ is a valued field extension of $K$. Given an $L$-valued point as a ring homomorphism $\psi\colon A\to L$, which in turn induces an injection $k(P')\to L$, the point in $Y^{\an}$ is $(\mathrm{ker}(\psi), |\cdot|_\psi)$, where $|\cdot|_\psi$ is the restriction of the absolute value on $L$ to $k(P')$.  We will regularly use this description of points in $Y^\an$ throughout the paper.

\smallskip

In fact, while we fixed $Y$ in \cref{con:mainConvention}, we will at times have to perform a base change $Y_L\rightarrow Y$ in some of the proofs, usually to make some point $P\in Y$ rational in the following sense:

\begin{remark}
  \label{rem:RationalPoints}
  Let $L$ be a valued field extension of $K$, and let $Y_L\coloneqq\Spec(A\otimes_K L)$.

  If $L=k(P)$ for some $P=(P',|\cdot|_P)\in Y^\an$, then there is a canonical point $P_{L}=(P_L',|\cdot|_{P_L})\in{Y^{\mathrm{an}}_{L}}$, where $P_L'\in Y_L$ is the ideal generated by $P'\in Y$ and $|\cdot|_{P_L}\colon k(P_L)=k(P)\rightarrow \RR_{\geq 0}$ is the same absolute value as $|\cdot|_{P}$.  Most importantly, $P_L\in Y^\an$ is rational.

  For general $L$, we use $P_L\in Y_L^\an$ to denote any point lying over $P$. This exists by applying \cite[Lemma 2.3]{Gubler2013} to $Y^{\mathrm{an}}_{\hat{L}}\to{Y^{\mathrm{an}}_{L}}\to{Y^{\mathrm{an}}}$, where $\hat{L}$ is the completion of~$L$.  If $P$ is rational, then $P_L$ is unique.

  Note that $K$ is assumed to be complete in \cite[Lemma 2.3]{Gubler2013}, but the proof works verbatim for non-complete fields as well. Namely, one needs the fact if $K\to L_{i}$ are valued field extensions for $i=1,2$, then there is a valued field $L$ that fits into a commutative diagram of valued field extensions
\begin{equation*}
    \begin{tikzcd}
        K \arrow[r] \arrow[d] & L_{1} \arrow[d] \\
        L_{2} \arrow[r] & L
    \end{tikzcd}
\end{equation*}
But this immediately follows from the case where $K$ and the $L_{i}$'s are complete. %

\end{remark}

An important result we use in this paper is the fact that dense open subsets $U$ in $Y$ give dense open subsets $U^{\an}$ in $Y^{\an}$. We record this here for the convenience of the reader.

\begin{proposition}\label{lem:DensityZariski}
  Assume that $K$ is complete. Let $X$ be a scheme that is locally of finite type over $K$, and let $U\subseteq{X}$ be a dense open set. Then $U^{\mathrm{an}}$ is open and dense.
\end{proposition}
\begin{proof}
This follows from \cite[Proposition 2.6.4]{Berkovich1993}.
\end{proof}

\subsection{Fiberwise tropicalizations and local tropical bases}
In this section, we define the necessary tropical objects of our work.  Recall that $Y=\Spec(A)$ is an integral scheme and that $X=\Spec(A[x^\pm]/I)$ is a closed subscheme of a relative torus $T=\Spec(A[x^\pm])$ over $Y$ by \cref{con:mainConvention}. %

\begin{definition}
  \label{def:fiberwiseTropicalization}
  Let $P=(P',|\cdot|_{P})\in Y^\an$. Let $X_{P'}\to T_{P'}$ be the fiber of $X\to T$ over $P'$. We consider $X_{P'}$ and $T_{P'}$ as schemes over the valued field $k(P)$. %
  We define the 
  \emph{fiberwise tropicalization} of $X$ at $P$ to be the tropicalization of $X_{P'}$ inside $T_{P'}$, %
  as in \cite[Section 3]{Gubler2013} or \cite[Definition 3.2.1]{MS15}. %
  We endow this set with the structure of a weighted polyhedral complex in 
  $\RR^n$ as in \cite[Definition 13.4]{Gubler2013} or \cite[Definition 3.4.3]{MS15}.  %

  Given a decomposition $X=X_1\cap\dots\cap X_k$, we refer to the intersection of their fiberwise tropicalizations $\bigcap_{i=1}^k \Trop(X_{i,P})$ as a \emph{fiberwise tropical prevariety} at $P$.
\end{definition}

\begin{example}
  \label{ex:runningExampleFiberwiseTropicalization}
  Let $X$ and $Y$ be as in \cref{ex:runningExampleGenericRootCount}.  Consider the decomposition $X = X_1\cap X_2$ given by the two ideals
  \begin{equation*}
    I_1\coloneqq \big(a_1x^2+a_2y^2+a_3y,\; b_1x^2+b_2y^2+b_3z\big) \quad\text{and}\quad I_2\coloneqq \big( c_1z+c_2 \big),
  \end{equation*}
  and, for $\lambda>0$, the point
  \begin{equation*}
    P_\lambda\coloneqq \Big(a_1-(1+t^\lambda), a_2-1, a_3-1, b_1-1, b_2-1, b_3-1, c_1-t^2, c_2-1\Big) \in Y.
  \end{equation*}

  The fiberwise tropicalization $\Trop(X_{2,P_\lambda})$ is independent of $\lambda$.  It is $\Trop(X_{2,P_\lambda})=(0,0,-2)+\Span(e_1,e_2)$, where $e_1,e_2,e_3$ are the unit weights in the variables $x,y,z$, respectively.

  In contrast, the fiberwise tropicalization $\Trop(X_{1,P_\lambda})$ depends on $\lambda$.  It consists of two vertices $v_1=(0,0,0)$ and $v_2=(\lambda,\lambda,\lambda)$, both connected by an edge, and each vertex connected to two rays in four distinct directions $u_1,\dots,u_4$, see \cref{fig:OSCARcomputation} for an \textsc{OSCAR} computation in the case $\lambda=4$.

  In particular, $\Trop(X_{1,P_\lambda})$ and $\Trop(X_{2,P_\lambda})$ intersect in two points (each of multiplicity $2$), of which one diverges as $\lambda$ increases, see \cref{fig:tropicalIntersection}.
  This is consistent with the observation in \cref{ex:runningExampleGenericRootCount} that the generic root count equals $\ell_{X,\eta}=4$ yet at $P=(a_1b_2-a_2b_1)$ the root count equals $2$.
\end{example}

\begin{figure}[t]
  \centering
  \begin{minipage}[t]{0.475\linewidth}
\begin{lstlisting}[basicstyle=\ttfamily\scriptsize, escapeinside={(*@}{@*)}]
julia> vertices_and_rays(TropI1)
6-element SubObjectIterator{...}:
 [0, 0, 0]    (*@ \textcolor{blue!50!black}{\# $\eqqcolon v_1$} @*)
 [4, 4, 4]    (*@ \textcolor{blue!50!black}{\# $\eqqcolon v_2$} @*)
 [1, 1, 2]    (*@ \textcolor{blue!50!black}{\# $\eqqcolon u_3$} @*)
 [0, 0, -1]   (*@ \textcolor{blue!50!black}{\# $\eqqcolon u_4$} @*)
 [-1, 0, 0]   (*@ \textcolor{blue!50!black}{\# $\eqqcolon u_5$} @*)
 [-1, -2, -2] (*@ \textcolor{blue!50!black}{\# $\eqqcolon u_6$} @*)
\end{lstlisting}
  \end{minipage}
  \begin{minipage}[t]{0.475\linewidth}
\begin{lstlisting}[basicstyle=\ttfamily\scriptsize, escapeinside={(*@}{@*)}]
julia> IncidenceMatrix(
         maximal_polyhedra(TropI1))
5x6 IncidenceMatrix
[1, 2] (*@ \textcolor{blue!50!black}{$=\mathrm{conv}(v_1,v_2)$} @*)
[2, 3] (*@ \textcolor{blue!50!black}{$=v_2+\RR_{\geq0}\cdot u_3$} @*)
[2, 4] (*@ \textcolor{blue!50!black}{$=v_2+\RR_{\geq0}\cdot u_4$} @*)
[1, 5] (*@ \textcolor{blue!50!black}{$=v_1+\RR_{\geq0}\cdot u_5$} @*)
[1, 6] (*@ \textcolor{blue!50!black}{$=v_1+\RR_{\geq0}\cdot u_6$} @*)
\end{lstlisting}
  \end{minipage}
  \caption{OSCAR output for $\Trop(X_{1,P_4})$ from \cref{ex:runningExampleFiberwiseTropicalization} (black) and their interpretation (blue).}
  \label{fig:OSCARcomputation}
\end{figure}

\begin{figure}[t]
  \centering
  \begin{tikzpicture}[x={(240:4mm)},y={(0:8mm)},z={(90:4mm)},every node/.style={font=\small}]
    \coordinate (v1) at (0,0,0);    %
    \coordinate (v2) at (4,4,4);    %
    \coordinate (p1) at (-1,-2,-2); %
    \coordinate (p2) at (4,4,-2);   %

    \coordinate (foo1) at (-2,-4,-4); %
    \coordinate (foo2) at (4,4,-3.75); %
    \draw[densely dashed,thin] (p1) -- (foo1);
    \draw[densely dashed,thin] (p2) -- (foo2);

    \fill[blue,opacity=0.2] (-1.75,-3.5,-2) -- (-1.75,6.5,-2) -- (6,6.5,-2) -- (6,-3.5,-2) node[anchor=north west,opacity=1] {$\Trop(X_{2,P_\lambda})$} -- cycle;

    \draw[thick] (v1) -- (v2);

    \draw[->,thick] (v2) -- ++(1.5,1.5,3) node[anchor=south west] {$u_3$};
    \draw[thick] (v2) -- node[right] {$\Trop(X_{1,P_\lambda})$} (p2);
    \draw[->,thick] (foo2) -- ++(0,0,-1) node[anchor=north] {$u_4$};

    \draw[->,thick] (v1) -- ++(-2.5,0,0) node[anchor=south] {$u_5$};
    \draw[thick] (v1) -- (p1);
    \draw[->,thick] (foo1) -- ++(-1,-2,-2) node[anchor=north east] {$u_6$};

    \draw[->,dashed,red!50!black] (p2) -- ++(1,1,0);
    \draw[->,thick,red!50!black] (v2) -- ++(1,1,1) node[right] {$(1,1,1)$};

    \fill (v1) circle (2pt);
    \fill (v2) circle (2pt);
    \node[anchor=south east] at (v1) {$v_1$};
    \node[anchor=south east] at (v2) {$v_2$};
    \draw[fill=white] (p1) circle (2pt);
    \draw[fill=white] (p2) circle (2pt);

  \end{tikzpicture}
  \caption{The intersection of $\Trop(X_{1,P_\lambda})$ and $\Trop(X_{2,P_\lambda})$ from \cref{ex:runningExampleFiberwiseTropicalization} for $\lambda=4$. The red arrows show how $\Trop(X_{1,P_\lambda})$ and consequently the intersection changes as $\lambda\rightarrow\infty$.}
  \label{fig:tropicalIntersection}
\end{figure}

Next we require the notion of tropical bases in our setting.

\begin{definition}\label{def:TropicalBases}
  Let $P\in Y^\an$, and let $V\subseteq Y$ be a Zariski-open set with $\pi(P)\in V$. Let $A_V$ be the induced coordinate ring, and $I_V$ be the induced ideal.  Let $f_1,\dots,f_k\in I_V\subseteq A_V[x^\pm]$ be a set of generators, say $f_i=\sum c_{i,\alpha}x^{\alpha}$ for some $c_{i,\alpha}\in A_V$.

  We say $f_1,\dots,f_k$ are a \emph{$K$-rational local tropical basis} of $X$ around $P$ if there is an open neighbourhood $U\subseteq Y^\an$ with $P\in U$ and $\pi(U)\subseteq V$ such that $f_{1,Q},\dots,f_{k,Q}\in I_Q$ is a tropical basis for all $Q\in U$, i.e., $\Trop(X_{Q})=\bigcap_{i=1}^{k} \Trop(V(f_{i})_{Q})$. Moreover, we say $f_1,\dots,f_k$ are \emph{non-degenerate} around $P$, if $c_{i,\alpha}(Q)\neq 0$ for all $Q\in U$ unless $c_{i,\alpha}(P)=0$.

  A \emph{(non-degenerate) local tropical basis} around $P$ is a (non-degenerate) $L$-rational local tropical basis of $X_L$ at $P_L$ for some valued field extension $K\to{L}$ and a point $P_{L}$ lying over $P$.
\end{definition}

\begin{example}
  \label{ex:runningExampleTropicalBasis}
  Let $X_1$ and $X_2$ be as in \cref{ex:runningExampleFiberwiseTropicalization}, i.e., given by the ideals
  \begin{equation*}
    I_1\coloneqq \big(\underbrace{a_1x^2+a_2y^2+a_3y}_{\eqqcolon f_1},\; \underbrace{b_1x^2+b_2y^2+b_3z}_{\eqqcolon f_2}\big) \quad\text{and}\quad I_2\coloneqq \big( \underbrace{c_1z+c_2}_{\eqqcolon g} \big).
  \end{equation*}
  We then have
  \begin{itemize}[leftmargin=7mm]
  \item $g$ is a local tropical basis of $X_2$ for all $P\in Y^{\an}$, and it is non-degenerate as long as $c_{1}(P)\neq 0\neq c_{2}(P)$.
  \item $f_1,f_2$ form a local tropical basis of $X_1$ for those $P\in Y^{\an}$ for which $\Trop(V(f_1)_P)$ and $\Trop(V(f_2)_P)$ intersect transversally.
  \item $f_1,f_2$ and $f_0\coloneqq b_1f_1-a_1f_2 = (a_1b_2 - a_2b_1)y^2 - a_3b_1y + a_1b_3z$ form a local tropical basis of $X_1$ for all $P\in Y^{\an}$. It is non-degenerate for all $P$ for which $a_{i}(P)\neq 0\neq b_{i}(P)$ and $(a_{1}b_{2} - a_{2}b_{1})(P)\neq 0$.
  \end{itemize}
\end{example}

Note that local tropical basis are preserved under valued field extensions:

\begin{lemma}
  \label{lem:tropicalBasisFieldExtension}
  Let $f_{1},\dots,f_k$ be a $K$-rational local tropical basis at $P$ and let $K\to{L}$ be a valued field extension. Then $f_{1},\dots,f_k$ are an $L$-rational local tropical basis at any $P_{L}$ mapping to $P$ under the morphism $Y_{L}^{\mathrm{an}}\to{Y^{\mathrm{an}}}$.
\end{lemma}
\begin{proof}
  The morphism $X_{L}^{\mathrm{an}}\to{X^{\mathrm{an}}}$ is continuous, so the neighborhood $U$ of $P$ gives an open neighborhood $U_{L}$ of $P_{L}$. The $f_{i}$'s then still form a tropical basis on $U_{L}$ since tropicalizations are invariant under field extensions by \cite[Theorem 3.2.4]{MS15}.
\end{proof}

Finally, we introduce an abbreviation for the cardinality of a finite stable intersection, and show that it is invariant under translation and taking recession fans.   The latter is a special case of \cite[Theorem 5.7]{AHR16}. For the definition and properties of stable intersections, see \cite[Section 3.6]{MS15}.

\begin{definition}
  \label{def:tropicalIntersectionNumber}
  Let $\Sigma_1,\dots,\Sigma_k$ be balanced polyhedral complexes of complementary dimension in $\RR^n$, i.e., $\codim(\Sigma_1)+\dots+\codim(\Sigma_k)=n$.  Their \emph{tropical intersection number} is the cardinality of their stable intersection, counted with multiplicity:
  \begin{equation*}
    \Sigma_1\cdot\ldots\cdot\Sigma_k\coloneqq \#\Big(\Sigma_1\cap_{\st}\dots\cap_\st\Sigma_k\Big).
  \end{equation*}
\end{definition}

\begin{lemma}
  \label{lem:tropicalIntersectionProductTranslationInvariant}
  Let $\Sigma_1,\dots,\Sigma_k$ be balanced polyhedral complexes in $\RR^n$ of complementary dimension and let $v_1,\dots,v_k\in\RR^n$.  Then
  \[ \Sigma_1\cdot\ldots\cdot \Sigma_k = (\Sigma_1+v_1)\cdot \ldots\cdot(\Sigma_k+v_k) \, . \]
\end{lemma}
\begin{proof}
  Without loss of generality, we may assume that $k=2$ and that $v_2=0$. Consider the function $m\colon [0,1] \rightarrow \ZZ$ given by $t \mapsto (\Sigma_1+t\cdot v_1)\cdot \Sigma_2$.  %
  By the alternative definition of stable intersection via perturbations in \cite[Proposition~3.6.12]{MS15}, the tropical intersection product is invariant under perturbation, hence $m$ is locally constant on $[0,1]$.  Since $[0,1]$ is connected, it follows that $m$ is constant.
\end{proof}

\begin{lemma}
  \label{lem:tropicalIntersectionNumberRecessionFan}
  Let $\Sigma_1,\dots,\Sigma_k$ be balanced polyhedral complexes of complementary dimension in $\RR^n$, and let $\rec(\Sigma_1)$ denote the recession fan of $\Sigma_1$.  Then
  \begin{equation*}
    \Sigma_1\cdot\ldots\cdot\Sigma_k = \rec(\Sigma_1)\cdot\Sigma_2\cdot\ldots\cdot\Sigma_k.
  \end{equation*}
\end{lemma}
\begin{proof}
  Without loss of generality, we may assume that $k=2$.  For $\lambda>0$ consider $\lambda\cdot\Sigma_1\coloneqq\{\lambda\cdot\sigma_1\mid\sigma_1\in\Sigma_1\}$, where $\lambda\cdot\sigma_1$ denotes linear scaling $\sigma_1$ by $\lambda$, and $\mult_{\lambda\cdot\Sigma_1}(\lambda\cdot\sigma_1)\coloneqq \mult_{\Sigma_1}(\sigma_1)$.  Note that $\lambda\cdot\Sigma_1$ describes a degeneration of $\Sigma_1=1\cdot\Sigma_1$ to $\rec(\Sigma_1)=\lim_{\lambda\to 0}\lambda\cdot\Sigma_1$ where the limit is taken with respect to the Hausdorff distance.  Observe however that locally the degeneration looks like a translation, i.e., for all $w\in \lambda\cdot\Sigma_1\cap_\st\Sigma_2$ there is a $u\in\RR^n$ such that for $\varepsilon>0$ sufficiently small $(\lambda\pm\varepsilon)\cdot\Sigma_1\cap_\st\Sigma_2=(\lambda\cdot\Sigma_1\pm\varepsilon\cdot u)\cap_\st\Sigma_2$ locally around $w$. The statement now follows from the fact that tropical intersection numbers are invariant under translation by \cref{lem:tropicalIntersectionProductTranslationInvariant}.
\end{proof}

\section{Tropical flatness and the generic tropical flatness theorem}\label{sec:FiberwiseTropicalizations}
This section revolves around the notion of tropical flatness, which can be regarded as a regularity condition on the variation of the fiberwise tropicalizations. Indeed, we will see in \cref{sec:TropGeneric} that if $X$ is tropically flat around $P$, then many properties of the fiber $X_P$ that are evident in the fiberwise tropicalization $\trop (X_{P})$ become generic properties of $X$.
Moreover, we show that under our assumption in \cref{con:mainConvention} any scheme $X$ is tropically flat for generic $P\in Y^{\an}$.

\subsection{Tropical flatness}\label{sec:tropicalFlatnessBasics}
In this section, we introduce tropical flatness, and describe around which points hypersurfaces and linear spaces are tropically flat.

\begin{definition}
  \label{def:TropicallyFlat}
  We say $X$ is \emph{tropically flat} over $P\in{Y^{\mathrm{an}}}$ if it admits a non-degenerate local tropical basis at $P$ as defined in \cref{def:TropicalBases}. The \emph{tropically flat locus} is the set of all $P\in{Y^{\mathrm{an}}}$ over which $X$ is tropically flat.
\end{definition}

We will see in \cref{sec:GenericTropicalFlatness} that the tropically flat locus is dense in $Y^{\an}$. We discuss two important examples here where the tropically flat locus can be described explicitly.  %

\begin{lemma}\label{cor:TropConstant}
  Let $X=V(f)$ for some non-constant polynomial $f\in A[x_{1}^{\pm},\dots,x_{n}^{\pm}]$, say $f=\sum_\alpha c_{\alpha}x^{\alpha}$. Let $U=\bigcap_\alpha D(c_{\alpha})$.  Then the tropically flat locus of $X$ is $U^{\an}$.
\end{lemma}
\begin{proof}
  Note that $f_P$ is a tropical basis for the ideal it generates for all $P\in Y^{\an}$ with $f_P\neq 0$ \cite[Example 2.6.4]{MS15}.  Hence $f$ is a local tropical basis around all $P\in U^{\an}$. And $f$ is non-degenerate around all $P\in U^{\an}$ by construction of $U^\an$.
\end{proof}

We now describe the tropical flat locus for linear spaces, for which we need the Pl\"ucker vectors of both the linear space as well as its orthogonal complement:

\begin{definition}\label{def:PlueckerVectors}
  Let $X=V(I)$ for some linear ideal $I\subseteq A[x_{1}^{\pm},\dots,x_{n}^{\pm}]$. Suppose that $X$ is generically of codimension $k$, and fix a set of generators $I_{\eta} = \langle f_{1,\eta},\dots,f_{k,\eta}\rangle$, say $f_{i,\eta}=\sum_{j=1}^n c_{i,j}(\eta)\cdot x_j\in K(Y)[x_{1}^{\pm},\dots,x_{n}^{\pm}]$ for some $c_{i,j}\in A$.
  Consider the coefficient matrix $C\coloneqq (c_{i,j})_{i\in[k],j\in[n]}\in A^{k\times n}$ and denote $C(\eta)\coloneqq (c_{i,j}(\eta))_{i\in[k],j\in[n]}\in K(Y)^{k\times n}$.  Let $D\in A^{(n-k)\times n}$ be a matrix such that $D(\eta)\in K(Y)^{(n-k)\times n}$ is of full rank and $C(\eta)\cdot D(\eta)^t=0$.  We denote the maximal minors of $C$ and $D$ by $p_{\Lambda}$ and $q_{[n]\setminus\Lambda}$ for $\Lambda\in\binom{[n]}{n-k}$, respectively, and refer to them as \emph{Pl\"ucker vectors}.
\end{definition}

\begin{remark}
Note that the $p_{I}$'s and $q_{I}$'s are only well defined up to a $K(Y)^{*}$-multiple, as we for instance can choose a different basis for the row space of $C$ and the row space of $D$. Moreover, we can do this independently for $C$ and $D$, so the $c_{J}$ relating the two sets of Pl\"{u}cker vectors can be general elements of $K(Y)^{*}$.  We will assume a fixed basis for both and only evaluate the $p_{I}$'s and $q_{I}$'s at points of $Y^{\an}$ where they define regular functions.
\end{remark}

\begin{lemma}\label{lem:FlatnessMatroids}
  Let $X=V(I)$ for some linear ideal $I\subseteq A[x_{1}^{\pm},\dots,x_{n}^{\pm}]$.  Let $p_{\Lambda},q_{[n]\setminus\Lambda}\in A$ be defined as in \cref{def:PlueckerVectors}, and denote $U\coloneqq \bigcap_{\Lambda\in\binom{[n]}{k}}D(p_\Lambda\cdot q_{[n]\setminus\Lambda})$. Then the tropically flat locus of $X$ contains $U^{\an}$.
\end{lemma}
\begin{proof}
  Let $P\in U^{\an}$.  As $U\subseteq \bigcap_{\Lambda\in\binom{[n]}{k}}D(p_\Lambda)$, we may assume that $f_1,\dots,f_k$ are non-degenerate around $P$.
  For $M=\{i_{k-1},\dots,i_{n}\}\subseteq [n]$ let $g_{M}\coloneqq \sum_{j=k-1}^n q_{M\setminus \{i_j\}} \cdot x_{i_j}\in A[x_1,\dots,x_n]$, so that the $g_{M,\eta}$'s form a tropical basis of $I_\eta$ by \cite[Lemma 4.3.16]{MS15}.  As $U\subseteq \bigcap_{\Lambda\in\binom{[n]}{k}}D(q_{[n]\setminus\Lambda})$, we can find $c_M\in A$ such that $f_M\coloneqq c_M\cdot g_M\in I$ and the $f_{M,Q}$ remain a tropical basis of $I_Q$ by \cite[Lemma 4.3.16]{MS15}. Then the $f_i$'s and the $f_M$'s form a non-degenerate local tropical basis around $P$, showing that $X$ is tropically flat around $P$.
\end{proof}

\begin{remark}
  If $I$ is affine linear, then we can apply the result above to the homogenization to obtain a locus over which $X$ is tropically flat.
\end{remark}

\subsection{Tropically flat morphisms are locally constant}\label{sec:GenericTropicalFlatness}

In this section we show that the tropicalization of a subscheme $X\subseteq T$ that is tropically flat around a point $P$ is locally constant after a non-archimedean field extension. This is the key property that allows us to extend local results to global results in \cref{sec:TropGeneric}. It relies on the following lemma.

\begin{lemma}\label{lem:ApproximationLemma}
  Let $g\in{A}$ and $P\in{Y^{\mathrm{an}}}$ with $r\coloneqq |g(P)|\neq{0}$. Let $L\coloneqq k(P)$ be the residue field of $P$ and let $L\to  N$ be a valued field extension. Let $P_{L}\in Y^{\mathrm{an}}_{L}$ denote the canonical point and $P_{N}\in Y^{\mathrm{an}}_{N}$ any point over $P_{L}$, as in \cref{rem:RationalPoints}.
  Then there is a non-empty open neighborhood $U\subseteq Y^{\mathrm{an}}_{N}$ around $P_{N}$ and an element $c\in{L}$ (which we view as an element of $N$ through $L\to N$) with $|c|=r$ such that for all $Q\in U$
  \begin{equation}
    \label{eq:ApproximationLemma}
    |g(Q)-c(Q)|<r \quad\text{and}\quad |g(Q)|=r.
  \end{equation}
\end{lemma}
\begin{proof}
  Set $c\coloneqq g(P)$ and consider the open ball $U\coloneqq\{Q\in{Y_{N}^{\mathrm{an}}}\mid |(g-c)(Q)|<r\}$. Note that $U$ is open since $r\neq{0}$. For any $P_{N}\in{Y_{N}^{\mathrm{an}}}$ mapping to $P$, we have $P_{N}\in{U}$ since $g(P_{N})=g(P)=c$.
  The inequality in the lemma is then satisfied by definition. Note that such a point $P_{N}$ exists by \cref{rem:RationalPoints}.

  For the equality in the lemma, we use the non-archimedean triangle inequality, except the inequality is an equality since the two involved absolute values are distinct:
  \begin{equation*}
    |g(Q)|= \max\{|g(Q)-c(Q)|,|c(Q)|\}=|c(Q)|=r. \qedhere
  \end{equation*}
\end{proof}

\vspace{0.2cm}

\begin{example}\label{exa:LocalConstancy}
  Let $K=\mathbb{C}\{\!\{t\}\!\}$ be the field of complex Puiseux series, let $A=K[x]$, and let $Y=\Spec(A)$. The Berkovich analytification $Y^{\an}$ is the infinite $\RR$-tree described in \cite[Section 2.1]{BakerRumely2010}, albeit without the point at infinity.  We will focus on the case where $g\coloneqq x\in A$ in this example.

  Consider the $K$-rational point $P\coloneqq (x-1)\in Y$ for which we have $k(P)=K$, $c\coloneqq g(P)=1$, $r\coloneqq |g(P)|=1$, and $U=\{Q\mid |x(Q)-1|<1\}$.  It is straightforward to verify that the conditions in Equation~\eqref{eq:ApproximationLemma} hold.

  Now consider the \emph{Gauss point} $P\coloneqq \zeta_{G}\in Y^\an$, which is $\zeta_G=(\eta,|\cdot |_{\zeta_G})$ where $\eta\in Y$ is the generic point and $|\cdot|_{\zeta_G}\colon L\coloneqq K(x)\rightarrow \RR_{\geq0}$ is the absolute value with
  \begin{equation*}
    \Big|\sum_{i=0}^{n}c_{i}x^{i}\Big|_{\zeta_{G}}=\max\Big\{|c_{i}|\mid 0\leq{i}\leq{n}\Big\}.
  \end{equation*}
The new coordinate algebra is then $A_{L}=L\otimes_{K}K[x]$. Note that it is a bit dangerous to identify $A_L$ with $L[x]$, since we have two distinct copies of $x$: $x\otimes 1$ and $1\otimes x$. In particular,  $1\otimes x-x\otimes 1$ is not zero. We write $\zeta_{G,L}$ for the canonical point of $Y^{\an}_{L}$ lying over $\zeta_{G}$, which can be obtained from the ring homomorphism $A_{L}\to L$ sending $1\otimes x$ and $x\otimes 1$ to $x$.

  The natural map $A\to A_{L}$ maps $g=x\in A$ to $1\otimes x$. Since $g(\zeta_{G})$ maps to $x\otimes 1$ under $L\to A_{L}$, we see that the element $g-c$ used in \cref{lem:ApproximationLemma} is $1\otimes x-x\otimes 1$. Moreover, $r\coloneqq |g(\zeta_{G})|=1$, so that the open neighborhood is %
  \begin{equation*}
    U=\Big\{Q\in Y^{\an}_{L}\mid |1\otimes x-x\otimes 1|<1\Big\}.
  \end{equation*}
  Note that $(1\otimes x-x\otimes 1)(\zeta_{G,L})=0$ so that $\zeta_{G,L}\in U$. Another point of $U$ is for instance given by the ring homomorphism $\psi_{Q}:A_{L}\to L$ sending $1\otimes x$ to $x+t$. On this open neighborhood $U$, we have that $|g(Q)|=1$. %

\end{example}

As an immediate corollary, we find that fiberwise tropicalizations that are tropically flat are locally constant after a non-archimedean base change.

\begin{corollary}\label{cor:TropicallyFlatImpliesConstant}
  Suppose that $X$ is tropically flat over a point $P\in Y^{\an}$. Then there is a valued field extension $K\subseteq N$, a point $P_{N}\in Y^\an_N$ over $P$, and an open neighborhood $U_{N}\subseteq Y^{\an}_{N}$ of $P_{N}$ such that $\Trop(X_{Q})$ is equal to $\Trop(X_{P})$ %
  for all $Q\in U_{N}$. %
\end{corollary}
\begin{proof}
  Let $f_{1},\dots,f_{k}$ be a non-degenerate local tropical basis around $P$ with field extension $K\to  L$, open neighborhood $U_{L}$ and point $P_{L}$. By \cref{lem:ApproximationLemma}, there is a valued field extension $L\to N$, a point $P_{N}$ lying over $P_{L}$, and open neighborhoods $V_{i,j,N}$ of $P_{N}$ over which the absolute values of the nonzero coefficients $c_{i,j}$'s of the $f_{i}$'s are constant. %
  We then take $W\coloneqq(\bigcap_{i,j} V_{i,j,N})\cap U_{N}$  to conclude that the tropicalizations are constant over $W$. Moreover, we have %
  $\Trop(X_{N,P_{N}})=\Trop(X_{P})$ (see \cref{lem:tropicalBasisFieldExtension}), so that all these tropicalizations are equal to  %
  $\mathrm{trop}(X_{P})$.
\end{proof}
\begin{remark}
   If $X$ is not tropically flat around $P$, then the statement of \cref{cor:TropicallyFlatImpliesConstant} is generally not true, even after a non-archimedean field extension. For instance, consider $X=V(f)$ for $f=y_{1}-y_{2}-(x-1)\in A[y_{1}^{\pm},y_{2}^{\pm}],$ where $A=K[x]$ as in \cref{exa:LocalConstancy}. Then $X$ is not tropically flat over $P=(x-1)$, and it is not constant near $P$ over any valued field extension. %
\end{remark}

\subsection{Generic tropical flatness theorem}

In this section, we prove that under the conditions in \cref{con:mainConvention} the tropically flat locus of $X=V(I)$ contains an open and dense set of $Y^\an$.  Central to our arguments are generic valuations and the following closed neighborhoods that contain them:

\begin{definition}\label{def:GenericValuation}
  Recall that $\pi\colon Y^\an\rightarrow Y, (P,|\cdot|_P)\mapsto P$ in \cref{def:analytification} denotes the natural forgetful map. A \emph{generic valuation} for $Y$ is a point $P\in Y^{\mathrm{an}}$ such that $\pi(P)$ is the generic point of $Y$.  Moreover, for a finite set $\mathcal C\in A$ we set
  \begin{equation*}
    \mathbb{B}_{P,\mathcal{C}}\coloneqq \{Q\in Y^{\an}\mid |c(Q)|=|c(P)| \text{ for all } c\in\mathcal C\}.
  \end{equation*}
\end{definition}

Before we come to the proof, note by the following two lemma that generic valuations are dense in $Y^\an$, and, if $K$ is non-trivially valued, then the closed subset $\mathbb{B}_{P,\mathcal{C}}$ contains rational points and open set neighborhoods around them:

\begin{lemma}\label{lem:GenericValuations3}
  Any basic open set $\mathbb{B}(r_{1},r_{2},h)\coloneqq \{Q\in Y^\an\mid r_1<|h(P)|<r_2\}$, where $h\in {A\backslash K}$ and $0\leq r_{1}<r_{2}$, contains a generic valuation $P\in Y^{\an}$.
\end{lemma}
    \begin{proof}
  Consider the subalgebra $K[h]\subseteq{A}$, giving a non-constant morphism $\mathrm{Spec}(A)\to\mathbb{A}^{1}_{K}$. We extend this to a transcendence basis for $A$, so that $K(A)$ is finite over $K(h_{1},\dots,h_{n})$, where $h_{1}=h$. This induces a rational map $\mathrm{Spec}(A)\to\mathbb{A}^{n}$ that is finite over an open subset $U$ of $\mathbb{A}^{n}$.

  Write $\Gamma$ for the value group of $K$. For any $v=(v_1,\dots,v_n)\in\Gamma^{n}$, we now have a natural generic valuation $P_{0}\in (\mathbb{A}^{n})^{\an}$ such that $|h_{i}(P_{0})|=v_{i}$. Explicitly, we can construct the algebra $\mathcal{A}=R[h_{1}/\varpi^{v_{1}},\dots,h_{n}/\varpi^{v_{n}}]$ over the valuation ring $R$ of $K$. Here $\varpi^{v_{i}}$ is the element obtained from a chosen splitting of $v\colon K^{*}\to \Gamma$ which exists by \cite[Lemma 2.1.15]{MS15}. Note that we used our assumption that $K$ is algebraically closed here by \cref{con:mainConvention}. The spectrum of $\mathcal{A}$ is isomorphic to $\mathbb{A}^{n}_{R}$, and the localization of $\mathcal{A}$ at the ideal $\mathfrak{m}\mathcal{A}$ is a valuation ring that contains $R$. This induces a natural valuation on $K(A)$ that has the desired properties.

  Let $r$ be a real number with $r_{1}<r<r_{2}$ such that $-\mathrm{log}(r)\in\Gamma$. This exists because $K$ is algebraically closed.
We now take $v_{1}=-\mathrm{log}(r)$ and obtain a point $P_{0}$. Any point $P$ in the preimage of $P_{0}$ under the map $Y\to (\mathbb{A}^{n})^{\an}$ then has the desired properties.
\end{proof}

\begin{lemma}\label{lem:FindingRationalPoints}
  Let $P$ be a generic valuation and let $\mathcal{C}\subseteq A$ be a finite non-empty subset. Suppose that $K$ is non-trivially valued and algebraically closed. Then there is a rational point $Q\in \mathbb{B}_{P,\mathcal{C}}$ and an open neighborhood $U\subseteq \mathbb{B}_{P,\mathcal{C}}$ around $Q$.
\end{lemma}
\begin{proof}
    We write $\Gamma\coloneqq \mathrm{val}(K)\subseteq\mathbb{R}$ for the value group of $K$.
    We extend the set $\mathcal{C}$ to a generating set $\mathcal{C}'$ of $A$.  This gives a closed embedding
  \begin{equation*}
    \phi\colon Y\to\mathbb{A}^{m}.
  \end{equation*}
  Let $v\in(\Gamma\cup\{\infty\})^m$ be the tropicalization of $\phi^{\mathrm{an}}(P)$. This lies in the tropicalization of $\phi(Y)$, so by \cite[Theorem 3.2.3]{MS15} we can find a $K$-rational point of $\phi(Y)$ that tropicalizes to $v$. This also gives a $K$-rational point $Q$ of $Y$ and it has the desired properties by construction. The last part now follows from \cref{lem:ApproximationLemma}.
\end{proof}

The overall proof of \cref{thm:TropicalGenericFlatness2} has two main intermediate steps:  Given a generic valuation $P\in Y^\an$, we show that there exists a finite subset $\mathcal C\subseteq A$ such that:

\begin{enumerate}
\item (\cref{lem:GroebnerComplexStable2}) the fiberwise Gr\"obner complex of $I$ is constant on $\mathbb{B}_{P,\mathcal C}$ if $I$ is a homogeneous polynomial ideal. The proof is done by looking at the individual Gr\"obner polyhedra of $I_P$ and requires some Gr\"obner basis arguments, and the homogeneity is necessary for the Gr\"obner complex to be well-defined.
\item (\cref{lem:fiberwiseTropicalizationLocallyConstant}) the fiberwise tropicalization of $X$ is constant on $\mathbb{B}_{P,\mathcal C}$ after extending the parameter space $Y$ using a finite covering $Y'\to Y$. The proof is done by looking at a tropical witness set from \cref{def:AlgebraicWitnesses}, and the covering is necessary in order to regard said witnesses as elements in the coordinate ring. %
\end{enumerate}

Lastly, the proof of \cref{thm:TropicalGenericFlatness2} combines all aforementioned results and shows that any basic open set $\mathbb{B}(\epsilon_{0},\epsilon_{1},h)\subseteq Y^{\an}$ intersects some $\mathbb{B}_{P,\mathcal C}$ that is contained in the tropically flat locus of $X$. %

\subsection*{Stability of Gr\"obner complexes}
Before we prove the stability of Gr\"obner complexes, we need to recall several concepts from \cite[Section 2.4 and 2.5]{MS15}.  In order for said concepts to be well defined, we will regard $I$ as a homogeneous polynomial ideal.

\begin{definition}
  Let $Q\in Y^{\mathrm{an}}$ and $w\in\RR^n$.
  A finite set $G_Q\subseteq I_Q$ is called a \emph{Gr\"obner basis} with respect to $w$, if $I_Q=\langle G_Q\rangle$ and $\initial_w(I_Q)=\langle \initial_w(g)\mid g\in G_Q\rangle$.
  The \emph{Gr\"obner polyhedron} of $I_Q$ around $w$ is $C_w(I_Q) \coloneqq \mathrm{cl}(\{w'\in\RR^n\mid\initial_{w'}(I_Q)=\initial_w(I_Q)\})$, where $\mathrm{cl}(\cdot)$ denotes the closure with respect to the Euclidean topology. The \emph{Gr\"obner complex} of $I_Q$ is denoted by $\Sigma(I_Q)\coloneqq\{C_w(I_Q)\mid w\in\RR^n \}$. If $I_Q$ is homogeneous, then $\Sigma(I_Q)$ is a finite polyhedral complex and $\mathrm{trop}(V(I_Q))$ is the support of a subcomplex of $\Sigma(I_Q)$.
\end{definition}

\begin{lemma}\label{lem:GroebnerBasisStable2}
  Let $K$ be non-trivially valued, let $I$ be a homogeneous ideal. Let $P\in Y^{\mathrm{an}}$ be a generic valuation, and $w\in\RR^n$ a weight vector. Then there is a finite subset $\mathcal G=\{g_1,\dots,g_m\}\subseteq I$ and a finite subset $\mathcal{C}\subseteq A$ such that the following hold:
  \begin{enumerate}
  \item \label{enumitem:GB1} $\mathcal G_P=\{g_{1,P},\dots,g_{m,P}\}$ is a Gr\"obner basis of $I_P$ with respect to $w$,
  \item \label{enumitem:GB2} $\mathcal G_{Q}=\{g_{1,Q},\dots,g_{m,Q}\}$ is a Gr\"obner basis of $I_{Q}$ w.r.t. $w$ for all $Q\in \mathbb{B}_{P,\mathcal{C}}$,
  \item \label{enumitem:GB3} the monomial supports of $\mathcal G_{Q}$ and $\mathcal G_P$ are equal for all $Q\in \mathbb{B}_{P,\mathcal{C}}$,
  \item \label{enumitem:GB4} the valuations of the coefficients of $\mathcal{G}_{Q}$ and $\mathcal G_P$ coincide for all $Q\in \mathbb{B}_{P,\mathcal{C}}$.
  \end{enumerate}
\end{lemma}
\begin{proof}
  Fix an ordering $\succ$ on the monomials in $x_{1},\dots,x_{n}$ and, for $Q\in Y^\an$, let $>$ denote the ordering on polynomials in $k(Q)[x]$ defined using $w$ and the valuation on $k(Q)$ as in \cite[Definition~2.3]{ChanMaclagan2019}.  Let $\mathcal{G}_{P}=\{g_{1,P},\dots,g_{m,P}\}\subseteq I_{P}$ be a Gr\"{o}bner basis with respect to $w$ as computed by \cite[Algorithm 2.9]{ChanMaclagan2019}. In particular, for any two $g_{i,P},g_{j,P}\in \mathcal G_P$ their $S$-polynomial $S(g_{i,P},g_{j,P})$ will have normal form $0$ with respect to $\mathcal G_P$.  By \cite[Algorithm 2.4]{ChanMaclagan2019} this implies that there are $h_{k,P}\in k(Y)[x_1,\dots,x_n]$ such that
  \begin{equation*}
    S(g_{i,P},g_{j,P})=\sum_{k=1}^m h_{k,P}g_{k,P} \quad\text{and}\quad h_{k,P}g_{k,P}\geq{S(g_{i,P},g_{j,P})} \text{ for } k\in[m].
  \end{equation*}
  We may assume without loss of generality that $g_i\coloneqq g_{i,P}$ and $h_i\coloneqq h_{i,P}$ are in $A[x_1,\dots,x_n]\subseteq k(Y)[x_1,\dots,x_n]$. Set $\mathcal G\coloneqq \{g_1,\dots,g_m\}$ and take $\mathcal{C}$ to be the set of all coefficients of $h_{i}$, $g_{i}$ and $S(g_{i},g_{j})$, so that the valuations of these elements are then constant on $\mathbb{B}_{\mathcal{C}}$. Then Conditions \eqref{enumitem:GB1}, \eqref{enumitem:GB3} and \eqref{enumitem:GB4} hold straightforwardly.
  Moreover, for all $Q\in{\mathbb{B}_{\mathcal{C}}}$ we still have
  \begin{equation*}
    S(g_{i,Q},g_{j,Q})=\sum_{k=1}^m h_{k,Q}g_{k,Q} \quad\text{and}\quad h_{k,Q}g_{k,Q}\geq{S(g_{i,Q},g_{j,Q})} \text{ for } k\in[m].
  \end{equation*}
  By \cite[Algorithm 2.9]{ChanMaclagan2019}, this implies that $\mathcal{G}_{Q}\subseteq I_Q$ is a Gr\"{o}bner basis with respect to $w$, hence Condition \eqref{enumitem:GB2} holds also.
\end{proof}

\begin{lemma}\label{lem:GroebnerComplexStable2}
  Let $K$ be non-trivially valued and let $I$ be a homogeneous ideal. Let $P\in{Y^{\mathrm{an}}}$ be a generic valuation. Then there is a finite subset $\mathcal{C}\subseteq A$ such that for all $Q\in \mathbb{B}_{\mathcal{C}}$
  \begin{equation*}
    \Sigma(I_{P})=\Sigma(I_{Q}).
  \end{equation*}
\end{lemma}
\begin{proof}
  As there are finitely many Gr\"obner polyhedra \cite[Theorem 2.5.3]{MS15}, we can pick $w_1,\dots,w_m\in \RR^n$ such that $\Sigma(I_P)=\{C_{w_1}(I_P),\dots,C_{w_m}(I_P)\}$.
  Recall that, by the proof of \cite[Proposition 2.5.2]{MS15}, a Gr\"obner polyhedron $C_w(I_Q)$ is uniquely determined by the monomial support and coefficient valuations of a Gr\"obner basis with respect to $w$.
  The statement hence follows from applying \cref{lem:GroebnerBasisStable2} to all $w_i$ and taking the union of all resulting $\mathcal C$.
\end{proof}

\subsection*{Stability of tropicalizations}
To show the stability of tropicalizations, we need the following lemma to construct suitable points on $X_{P}$ that tropicalize to a given set of weight vectors.  In the lemma, we exploit notation and use $\textrm{val}(\cdot)$ to denote the valuation on all fields.

\begin{lemma}\label{lem:ExtendValuationPoints}
  Let $P\in Y^{\an}$ be a generic valuation. For any finite number of tropical points $w_{1},\dots, w_{k}\in\trop(X_{P})\cap \mathrm{val}(k(P))^n$, there exists a finite extension of valued fields $k(P)\to L$ and points $z_1,\dots,z_k\in X_{P}(L)\subseteq L^n$ such that
  \begin{math}
    \mathrm{val}(z_{j})=w_{j},
  \end{math}
  with $\mathrm{val}(\cdot)$ denoting coordinatewise valuation.
\end{lemma}
\begin{proof}
  Extend the valuation of $k(P)$ to an algebraic closure $\overline{k(P)}$. By the Fundamental Theorem of Tropical Geometry \cite[Theorem 3.2.3]{MS15}, we can find points defined over $\overline{k(P)}$ with the desired properties. As these are defined over a finite extension $k(P)\to L\to \overline{k(P)}$, we obtain the desired statement by restricting the chosen valuation to $L$.
\end{proof}

\begin{definition}\label{def:AlgebraicWitnesses}
  Let $P\in Y^{\an}$ be a generic valuation.  Let $w_1,\dots,w_m\in\RR^n$ such that $\Sigma(I_P)=\{C_{w_1}(I_P),\dots,C_{w_m}(I_P)\}$.  Let $z_{1},\dots ,z_{m}\in X_{P}(L)$ be the points over the finite extension $k(P)\to L$ from \cref{lem:ExtendValuationPoints}. We call $Z\coloneqq \{z_1,\dots,z_m\}$ a \emph{witness set} for $\Trop(X_{P})$.
\end{definition}

Note that the coordinates of the witness set $Z\subseteq X_P(L)$ need not be elements of $A$ or even a localization of $A$.  In order to add $Z$ to $\mathcal C$ for the construction of a new $\mathbb{B}_{P,\mathcal{C}}$, we require the following extension:

\begin{assumption}
  \label{assumption:witnessSet}
  For the remainder of the section, fix a witness set $Z=\{z_1,\dots,z_m\}$ $\subseteq X_P(L)$.  Let $A'$ be the integral closure of $A$ in $L$. This defines a finite normalization morphism $\mathrm{norm}\colon Y'\coloneqq \Spec(A')\to Y=\Spec(A)$, and the extended valuation from \cref{lem:ExtendValuationPoints} directly gives a point $P'\in Y'^{\an}$ mapping to $P\in Y^{\an}$.  Moreover, we can find an open neighborhood $V'\subseteq Y'$ such that $z_{i,j}\in A'_{V'}$. %

  Note that, by Grothendieck's generic flatness theorem, there is open subset $V\subseteq Y$ such that $\mathrm{norm}^{-1}(V)\to V$ is flat and thus open.  By restricting to open subsets of $Y'$ and $Y$, we may therefore assume that $z_{i,j}\in A'$.
\end{assumption}

\begin{lemma}\label{lem:fiberwiseTropicalizationLocallyConstant}
  Let $K$ be non-trivially valued, and let $I\subseteq A'[x_1,\dots,x_n]$.  Let $P\in{Y'^{\mathrm{an}}}$ be a generic valuation. Then there is a finite subset $\mathcal{C}\subseteq A'$ such that for all $Q\in \mathbb{B}_{\mathcal{C}}$
  \begin{equation*}
    \mathrm{trop}(V(I)_{P})=\mathrm{trop}(V(I)_{Q}).
  \end{equation*}
\end{lemma}
\begin{proof}
  By \cite[Proposition 3.2.8]{MS15}, we may assume that $I$ is homogeneous. First, let $\mathcal C_\Sigma\subseteq A'$ be the subset from \cref{lem:GroebnerComplexStable2}, so that $\Sigma(I_P)=\Sigma(I_Q)$ for all $Q\in\mathbb{B}_{\mathcal C_\Sigma}$.
  Second, let $Z=\{z_1,\dots,z_m\}\subseteq X_P(L)$ be the witness set from \cref{assumption:witnessSet}, i.e., $z_i=(z_{i,j})_{j=1,\dots,n}$ and $z_{i,j}\in A\subseteq L$.  For each $z_i\in Z$, we distinguish between two cases.

  If $\val(z_i)\in\trop(V(I)_P)$, we define $\mathcal{C}_i\coloneqq\{z_{i,1},\dots,z_{i,n}\}\subseteq A'$ so that $w_i\in \trop(V(I)_Q)$ for all $Q\in\mathbb{B}_{\mathcal C_i}$.
  If $\val(z_i)\notin\trop(V(I)_P)$, then $\mathrm{in}_{w_i}(I_{P})$ contains a monomial, which means there is a $f_{i}\in I$ such that $\mathrm{in}_{w_i}(f_{i,P})$ is monomial. In that case, we define $\mathcal{C_i}\subseteq A'$ to be the set of coefficients of $f_{i}$ so that $w_i\notin \trop(V(I)_Q)$ for all $Q\in\mathbb{B}_{\mathcal C_i}$.
  We then obtain the statement for $\mathcal C\coloneqq \mathcal C_\Sigma\cup \mathcal C_1\cup\dots\cup\mathcal C_m$.
\end{proof}

\subsection*{Generic tropical flatness}
We combine all previous results for the main theorem of the section:

\begin{theorem}\label{thm:TropicalGenericFlatness2}
  Let $X=V(I)$ and $Y$ be as in \cref{con:mainConvention}.
  Then the tropically flat locus of $X$ contains a dense open subset of $Y^{\mathrm{an}}$. If $K$ is non-trivially valued, then this locus moreover contains a dense open subset of $Y(K)\subseteq Y^\mathrm{an}$.
\end{theorem}
\begin{proof}
  Since we allow valued field extensions in \cref{def:TropicallyFlat} of tropically flatness, we can assume that $K$ is non-trivially valued.

  We first show the statement for $Y'=\Spec(A')$ and $X'=X\times_{Y} Y'$ from \cref{assumption:witnessSet}. Consider a basic open set $\mathbb{B}(r_{1},r_{2},h)\subseteq Y'^{\mathrm{an}}$ for some $r_1,r_2\in\RR_{\geq 0}$ and $h\in{A'}$.
  Using \cref{lem:GenericValuations3}, we can find a generic valuation $P\in \mathbb{B}(h,r_{1},r_{2})$.  Let $f_1,\dots,f_m\in A'[x_1,\dots,x_n]$ such that $f_{1,P},\dots,f_{m,P}$ is a tropical basis of $I_P$, and let $\mathcal C_f\subseteq A'$ be the set of coefficients of $f_1,\dots,f_m$.
  By \cref{lem:fiberwiseTropicalizationLocallyConstant}, there is a $\mathcal C_{\trop}\subseteq A'$ such that $\trop(X'_P)=\trop(X'_Q)$ for all $Q\in\mathbb{B}_{\mathcal{C}_{\trop}}$.
  Set $\mathcal C\coloneqq \{h\}\cup\mathcal C_f\cup\mathcal C_{\trop}$.
  By \cref{lem:FindingRationalPoints}, there is a rational point $Q_{0}\in \mathbb{B}_{\mathcal{C}}$ and an open neighborhood $Q_{0}\in U\subseteq \mathbb{B}_{\mathcal{C}}$.
  We find that $X$ is tropically flat around any point $Q\in{U}$. We conclude that the tropically flat locus of $X'$ contains an open and dense subset of $Y'$ as well as an open and dense subset of $Y'(K)$.

  We now treat the general case.  Consider a basic open set $\mathbb{B}\subseteq Y^\an$ containing a generic valuation $P$.  Let $\mathbb{B}'$ be the preimage of $\mathbb{B}$ under the open normalization map $\phi\colon Y'\to Y$ from \cref{assumption:witnessSet}.  Note that the point $P'$ from \cref{assumption:witnessSet} is in~$\mathbb{B}'$. Using what we proved above, we find an open neighborhood $U'\subseteq \mathbb{B}'$ of $P'$ such that $X'$ is tropically flat over $U'$.  Set $U=\phi(U')$, which is open as $\phi$ is open.  For all $Q'\in U'$ and $Q=\phi(Q')$, we have
  \begin{equation*}
\trop(X'_{Q'})=\trop(X_{Q})=\trop(X'_{P'})=\trop(X_{P}).
  \end{equation*}
  We thus find that $X$ is tropically flat over $U\subseteq \mathbb{B}$. In particular, the tropically flat locus contains an open and dense subset.
\end{proof}
\section{Tropical intersections and generic root counts}\label{sec:TropGeneric}

We now use the material from the previous two sections to show how generic properties of morphisms of schemes can be detected using tropical geometry. We will see that many properties of a single tropical fiber over a tropically flat point propagate to a dense open subset of the parameter space. %
In \cref{sec:GRCTropicalIntersectionNumber}, we prove %
\cref{thm:MainThm1}, which expresses the generic root count as a tropical intersection product. This also gives a standalone proof of Bernstein's theorem, see %
\cref{cor:BKK}. In \cref{sec:TorusEquivariant} we study torus-equivariant and parametrically independent systems, and we prove  \cref{thm:TorusEquivariant}. Finally, we discuss extensions of the results given here to analytic families of polynomial equations.

\subsection{Generic root counts as tropical intersection numbers}\label{sec:GRCTropicalIntersectionNumber}

In this section, we show that generic root counts can be expressed in terms of tropical intersection numbers, provided that we have a tropically transverse intersection around a tropically flat point. This extends Bernstein's theorem to possibly overdetermined families of polynomial equations with non-trivial relations among the coefficients.

We first show that the generic (co-)dimension can be deduced from a tropically transverse intersection around a tropically flat point.
\begin{proposition}\label{pro:DimensionsTransverseIntersections}
    Let $X=\bigcap_{i=1}^k X_{i}$ where $X_i$ is of generic codimension $d_{i}$. Let $P\in Y^{\an}$ around which the $X_{i}$'s are tropically flat, and suppose that the $\trop(X_{i,P})$'s intersect transversally. Then $X$ has generic codimension $\sum_{i=1}^{k}d_{i}$.
\end{proposition}
\begin{proof}
  Let $V$ be a dense open subset of $Y$ over which the $X_{i}$'s and $X$'s attain their generic dimensions, see \cref{lem:GenericDimensions}.
By \cref{cor:TropicallyFlatImpliesConstant}, we can find a valued field extension $K\to M$ and an open neighborhood $U_{M}$ of a point $P_{M}$ lying over $P$ such that the tropicalizations of the $X_{i,Q}$'s for $Q\in U_{M}$ are all equal to the tropicalization of $X_{i,P}$. We now take a point $Q_{0}\in V^{\an}_{M}\cap U_{M}$, which exists by \cref{lem:DensityZariski}.

   Recall the Bieri-Groves theorem, which applied to the irreducible components of a variety implies that the dimension of a variety is equal to the dimension of its tropicalization, see \cite[Theorem 3.3.8]{MS15}.
   Since dimensions are stable under field extensions, we conclude that the tropicalization of $X_{i,Q_{0}}$ is of codimension $d_{i}$.
   By \cite[Theorem 1.2]{OssermanPayne2013} and the transversality of the $\trop(X_{i,P})=\Trop(X_{i,Q})$'s, we have that $\Trop(X_{P})=\bigcap_{i=1}^{k}\Trop(X_{i,P})=\bigcap_{i=1}^{k}\Trop(X_{i,Q_{0}})=\Trop(X_{Q_{0}})$, which is of codimension $\sum_{i=1}^{k}d_{i}$. Since $Q_{0}\in V^{\an}_{M}$, we conclude.
\end{proof}

\begin{remark}
  If the generic fiber $X_\eta$ is empty and $X_P$ is non-empty over some $P\in{Y^{\mathrm{an}}}$, then $X$ is not tropically flat around $P$ by \cref{pro:DimensionsTransverseIntersections}. More generally, if the dimension of the fiberwise tropicalization $\mathrm{trop}(X_{P})$ is higher than the generic dimension, then $X$ is not tropically flat around $P$.
\end{remark}

For the next lemma, we recall the notion of higher intersection multiplicities as in \cite[Section 6.8]{OR2013}. %
Let $X_{1}$ and $X_{2}$ be to subvarieties of complementary dimension in an $n$-dimensional torus $T_K$ over an algebraically closed field $K$ that intersect properly in a zero dimensional set. Write $X=X_{1}\cap X_{2}$. We can view the local rings $\mathcal{O}_{X_{1},x}$ and $\mathcal{O}_{X_{2},x}$ as modules over the ring $\mathcal{O}_{T_K,x}$. In particular, this allows us to define the modules $\mathrm{Tor}_{i}(\mathcal{O}_{X_{1},x},\mathcal{O}_{X_{2},x})$ over $\mathcal{O}_{T_K,x}$.
For every $x\in{X}$, we then define %
\begin{equation*}
    i(x,X_{1}\cdot X_{2})=\sum_{i=0}^{n}(-1)^{i}\mathrm{dim}_K(\mathrm{Tor}_{i}(\mathcal{O}_{{X}_{1},x},\mathcal{O}_{X_{2},x})).
\end{equation*}
Here $\dim_K(\mathrm{Tor}_{i}(\mathcal{O}_{{X}_{1},x},\mathcal{O}_{X_{2},x}))$ denotes the dimension of $\mathrm{Tor}_{i}(\mathcal{O}_{{X}_{1},x},\mathcal{O}_{X_{2},x})$ as a vector space over $K$. %
We define the full intersection number of $X_{1}$ and $X_{2}$ as
\begin{equation*}
    i(X_{1}\cdot X_{2})=\sum_{x\in{X}}i(x,X_{1}\cdot X_{2}).
\end{equation*}
\begin{definition}
Let $X_{1}$ and $X_{2}$ be two closed subschemes of complementary codimension in an $n$-dimensional torus $T_K$ over an algebraically closed field $K$. Let $x\in{X_{1}\cap X_{2}}$.
    We say that the higher intersection multiplicities vanish at $x$ if $$\mathrm{Tor}_{i}(\mathcal{O}_{{X}_{1},x},\mathcal{O}_{X_{2},x})=(0)$$ for $i>0$. Similarly, we say that the higher intersection multiplicities vanish for $X_{1}$ and $X_{2}$ if the higher intersection multiplicities at all $x\in{X}$ vanish. Note that if the higher intersection multiplicities vanish at a point $x$, then we have that
    \begin{equation*}
        i(x,X_{1}\cdot X_{2})=\mathrm{dim}_K\Big(\bigslant{\mathcal{O}_{T_K,x}}{I_{1,x}}\otimes_{K}\bigslant{\mathcal{O}_{T_K,x}}{I_{2,x}}\Big)=\mathrm{dim}_K\Big(\bigslant{\mathcal{O}_{X,x}}{(I_{1,x}+I_{2,x})}\Big),
    \end{equation*}
    where $I_{1}$ and $I_{2}$ are the ideals corresponding to $X_{1}$ and $X_{2}$, and the $I_{i,x}$'s are the localizations of these ideals at $x$.

Suppose we are given closed subschemes $X_{1},\dots,X_{k}$ of complementary codimension in an $n$-dimensional torus $T_K$ over an algebraically closed field $K$ that intersect in a zero-dimensional set. %
Let $D\colon T_K\to T_K^{k}$ be the diagonal map. Then we define
\begin{equation*}
    i(\prod_{i=1}^{k}X_{i})=i(D(T_K)\cdot (X_{1}\times \dots\times X_{k})).
\end{equation*}
We say that the higher intersection multiplicities of the $X_{i}$'s vanish if the higher intersection multiplicities of $D(T_K)$ and $X_{1}\times \dots\times X_{k}$ vanish.
\end{definition}

Using this terminology, we can now prove the following well-known lemma and subsequent main theorem.

\begin{lemma}\label{lem:IntersectionMultiplicities}
  Let $X_1,\dots,X_k$ be of complementary dimension, i.e., $\sum_{i=1}^k\codim(X_{i})=n$, such that $X=\bigcap_{i=1}^k{X_{i}}$ is zero-dimensional. If $X$ lies in the Cohen-Macaulay-locus of every $X_{i}$, then the higher intersection multiplicities vanish.
\end{lemma}
\begin{proof}
  We have to show that the higher intersection multiplicities of $D(T_K)$ and $\prod_{i=1}^k X_{i}$ vanish. The support of the intersection corresponds to $\bigcap_{i=1}^k {X_{i}}$ embedded in the product by the map $D\colon T_K\to{T_K^{k}}$. Note that %
  $\prod_{i=1}^k X_{i}$ is Cohen-Macaulay at every $D(z)$ %
  for $z\in{\bigcap_{i=1}^k{X_{i}}}$. Indeed, this follows from \cite[\href{https://stacks.math.columbia.edu/tag/0C0W}{Lemma 0C0W}]{stacks-project} and  \cite[\href{https://stacks.math.columbia.edu/tag/045T}{Lemma 045T(1)}]{stacks-project}. The lemma now follows from \cite[\href{https://stacks.math.columbia.edu/tag/0B02}{Lemma 0B02}]{stacks-project}.
\end{proof}

\begin{theorem}\label{thm:MainThm1}
  Let $X_1,\dots,X_k$ be generically Cohen-Macaulay, pure and of complementary dimension, and let $X=\bigcap_{i=1}^k X_i$.
  Suppose there is a point $P\in Y^{\an}$ over which the $X_{i}$ are tropically flat and the tropical prevariety $\bigcap_{i=1}^{k}\trop(X_{i,P})$ is bounded.
  Then $X$ is generically finite with generic root count $\ell_{X,\eta}=\prod_{i=1}^{k}\trop(X_{i,P})$.
\end{theorem}
\begin{proof}
Throughout the proof, we will freely apply valued field extensions $K\to L$ as needed. One can easily verify that this does not change the validity of our assumptions and results. For instance, the local  %
rank of $p\colon X\to Y$ provided by Grothendieck's theorem \cite[\href{https://stacks.math.columbia.edu/tag/052A}{Proposition 052A}]{stacks-project} is stable under flat base change, and
being generically Cohen-Macaulay is stable under field extensions by   %
\cite[\href{https://stacks.math.columbia.edu/tag/00RJ}{Lemma 00RJ}]{stacks-project}. %

  Since the $X_{i}$'s are tropically flat over $P$, we can find a valued field extension $K\to  L$, a point $P_{L}$ and an open neighborhood $\mathbb{B}$ of $P_{L}$  such that $\trop(X_{i,Q})=\trop(X_{i,P})$ for all $Q\in \mathbb{B}$ by \cref{cor:TropicallyFlatImpliesConstant}. As mentioned before, we can assume without loss of generality that $K=L$, $P_{L}=P$ and $\mathbb{B}\subseteq Y^{\an}$ is an open neighborhood of $P$.  
  Let $U_{1}\subseteq{Y}$ be a dense open subset over which $p$ is free  \cite[\href{https://stacks.math.columbia.edu/tag/052A}{Proposition 052A}]{stacks-project} and let %
  $U_{2}$ be a dense open subset over which $p$ is Cohen-Macaulay and pure, see Lemmas \ref{lem:GenericallyCM} and \ref{lem:LemmaPurity}. %
  The space  $U^{\mathrm{an}}_{1}\cap{U^{\mathrm{an}}_{2}}$ is dense in $Y^{\mathrm{an}}$ by \cref{lem:DensityZariski}, so that $U^{\mathrm{an}}_{1}\cap U^{\mathrm{an}}_{2}\cap \mathbb{B}$ is non-empty. Let $Q$ be an element of this subset. Note that $\trop(X_{i,P})=\trop(X_{i,Q})$ by construction, so that the tropical prevariety $\bigcap_{i=1}^{k}\trop(X_{i,Q})$ is bounded.  %
  As the %
  the tropical prevariety is bounded, the fibers $X_{i,Q}$'s are pure and their codimensions add up to $n$, we can apply \cite[Corollary 6.13]{OR2013} to find that
  the tropical intersection number is equal to the sum of the algebraic intersection numbers. By \cref{lem:IntersectionMultiplicities}, this sum is equal to the sum of the algebraic lengths. But again using the fact that free modules are stable under base change, we find that this sum is the local rank provided by Grothendieck's theorem \cite[\href{https://stacks.math.columbia.edu/tag/052A}{Proposition 052A}]{stacks-project}. This concludes the proof.
\end{proof}

\begin{remark}
 The tropical intersection number in \cref{thm:MainThm1}
 is an algebraic intersection number in a suitable toric variety: %
 Let $X(\Delta)$ be a toric variety such that $\Delta$ is a compatible compactifying fan for the $\trop(X_{i,P})$'s as in \cite[Section~3]{OR2013}. Then by \cite[Proposition 3.12]{OR2013}, we find that the closures of the $X_{i,P}$'s in $X(\Delta)$ only intersect in the dense torus. In particular, the tropical intersection number in \cref{thm:MainThm1} is equal to the algebraic intersection number  $\prod_{i=1}^{k}\overline{X}_{i,P}$. %

\end{remark}

The following example shows the necessity of tropical flatness in \cref{thm:MainThm1}:

\begin{example}\label{exa:TropicallyFlatIllustration}
  Consider $X=X_1\cap X_2$ as well as $P_\lambda\in Y$ from \cref{ex:runningExampleFiberwiseTropicalization}.  In \cref{ex:runningExampleTropicalBasis}, it is shown that $X_{1}$ and $X_{2}$ are tropically flat around $P_\lambda$, hence \cref{thm:MainThm1} states that the generic root count of $X$ equals $\trop(X_{1,P_\lambda})\cdot \trop(X_{2,P_\lambda})$.  Indeed, the first was determined to be $4$ in \cref{ex:runningExampleGenericRootCount}, and the second was determined to be $4$ in \cref{ex:runningExampleFiberwiseTropicalization}.

  Moreover, as $\lambda\rightarrow\infty$, $P_\lambda$ converges to a point $P_\infty$ with $(a_1b_2-a_2b_1)(P_\infty)=0$, where the root count drops to $\ell_{X,P_\infty}=2$.  As $X_2$ remains tropically flat around $P_\infty$, this shows that $X_1$ is not tropically flat around $P_\infty$.
\end{example}

\begin{example}
  \label{ex:genusTwoExample}
  \label{ex:nonSquare}
  For a non-square example, we intersect a family of curves of genus $2$ in $\mathbb{P}^3$ with a family of hyperplanes.  The family of curves of genus $2$ can be constructed by pushing forward the intersection of a cubic and a quadratic under the Segre map $\mathbb{P}^{1}\times \mathbb{P}^{1}\to \mathbb{P}^{3}$ \cite[Section IV, Remark 6.4.1(c))]{Hartshorne1977}.

  Consider the parameter ring
  \begin{equation*}
    A\coloneqq \CC\{\!\{t\}\!\}\Big[a_{i},b_{j},c_{k}\mid 0\leq{i}\leq{3}, 0\leq{j}\leq{2}. 0\leq k\leq 3\Big],
  \end{equation*}
  and the ideal
  \begin{align*}
    \langle a_{0}x_{0}^{3}+a_{1}x_{0}^2x_{1}+a_{2}x_{0}x_{1}^{2}+a_{3}x_{1}^3 ,\quad b_{0}y_{0}^2+b_{1}y_{0}y_{1}+b_{2}y_{1}^2 \rangle \subseteq A[x_{0},x_{1},y_{0},y_{1}].
  \end{align*}
  In other words, the $a_i$'s parametrize a cubic on the first $\PP^1$, the $b_j$'s parametrize a quadratic on the second $\PP^1$, and the ideal is their intersection in $\mathbb{P}^{1}\times \mathbb{P}^{1}$.  The parameters $c_k$'s will be used later.

  Its preimage under the Segre map $A[w_{0},\dots,w_{4}]\to A[x_{0},x_{1},y_{0},y_{1}]$, is generated by the polynomials
  \begin{align*}
    f_{1}&\coloneqq w_{1}w_{2}-w_{0}w_{3},\\
    f_{2}&\coloneqq a_{0}b_{0}w_{0}^2w_{1} + a_{0}b_{1}w_{0}w_{1}^2 + a_{1}b_{0}w_{0}w_{1}w_{2} + a_{1}b_{1}w_{0}w_{1}w_{3}+ a_{0}b_{2}w_{1}^3 + a_{1}b_{2}w_{1}^2w_{3}\\
         & \qquad + a_{2}b_{0}w_{1}w_{2}^2 + a_{2}b_{1}w_{1}w_{2}w_{3} + a_{2}b_{2}w_{1}w_{3}^2 + a_{3}b_{0}w_{2}^2w_{3} + a_{3}b_{1}w_{2}w_{3}^2 + a_{3}b_{2}w_{3}^3,\\
    f_{3}&\coloneqq a_{0}b_{0}w_{0}^3 + a_{0}b_{1}w_{0}^2w_{1} + a_{1}b_{0}w_{0}^2w_{2} + a_{1}b_{1}w_{0}^2w_{3} + a_{0}b_{2}w_{0}w_{1}^2 + a_{1}b_{2}w_{0}w_{1}w_{3} \\
    &\qquad + a_{2}b_{0}w_{0}w_{2}^2 + a_{2}b_{1}  w_{0}w_{2}w_{3} + a_{2}b_{2}w_{0}w_{3}^2 + a_{3}b_{0}w_{2}^3 + a_{3}b_{1}w_{2}^2w_{3} + a_{3}b_{2}w_{2}w_{3}^2.
  \end{align*}
  Set $Y\coloneqq\Spec(A)$ and $X_1\coloneqq\Spec(A[w^\pm]/(f_1,f_2,f_3))$.  Pick $\lambda>0$.  \cref{fig:genusTwoExample} illustrates the fiberwise tropicalization $\Trop(X_{1,P})$ for $P\in Y^{\an}$ randomly chosen with $\mathrm{val}(a_0(P))=\mathrm{val}(b_0(P))=2\lambda$ and $\mathrm{val}(a_i(P))=\mathrm{val}(b_j(P))=\mathrm{val}(c_k(P))=0$ otherwise.  It consists of four vertices arranged in a quadrilateral, and each vertex is connected to two rays.  As $f_1,f_2,f_3$ are homogeneous, the fiberwise tropicalization is invariant under translation in direction of the all-ones vector $(1,1,1,1)$.

  Consider further the polynomials
  \begin{equation*}
    g_1\coloneqq c_0w_0-c_1\quad \text{and}\quad g_2\coloneqq c_2w_0w_1+t\cdot c_3w_2w_3,
  \end{equation*}
  and $X_2\coloneqq\Spec(A[w^\pm]/(g_1,g_2))$. Then $\Trop(X_{1,P})$ and $\Trop(X_{2,P})$ intersect in two points, of which one diverges as $\lambda\rightarrow\infty$, see \cref{fig:genusTwoExample} also (the lower intersection point diverges, the upper intersection point stays fixed).

  As $P$ was chosen randomly, \cref{thm:TropicalGenericFlatness2} implies that $X_1$ is tropically flat around $P$ with high probability. Additionally, as $X_{2}$ is independent of the parameters, $X_2$ is also tropically flat around $P$.  After verifying by direct computation that $\Trop(X_{1,P})$ and $\Trop(X_{2,P})$ intersect transversally, \cref{thm:MainThm1} implies that the generic root count of $X$ is their tropical intersection product with high probabiliy.

  In general, verifying that $\Trop(X_{1,P})$ and $\Trop(X_{2,P})$ intersect transversally using direct computation can be quite difficult. This is one of the motivations we will introduce the notion of torus equivariance and parametric independence in \cref{sec:TorusEquivariant}.
\end{example}

\begin{figure}[t]
  \centering
  \begin{tikzpicture}[x={(10mm,0mm)},y={(0mm,10mm)}]
    \coordinate (v1) at (0,0);
    \coordinate (v2) at (2,0);
    \coordinate (v3) at (2,2);
    \coordinate (v4) at (0,2);
    \coordinate (r5) at (-1,0);
    \coordinate (r6) at (0,-1);
    \coordinate (r7) at (2,0);
    \coordinate (r8) at (0,1);
    \draw[thick]
    (v1) -- (v2) -- (v3) -- (v4) -- cycle;
    \draw[thick,->]
    (v1) -- ++(r5) node[anchor=east] {$u_3$};
    \draw[thick,->]
    (v1) -- ++(r6) node[anchor=north] {$u_4$};
    \draw[thick,->]
    (v2) -- ++(r6) node[anchor=north] {$u_4$};
    \draw[thick,->]
    (v2) -- node[midway] (i1) {} ++(r7) node[anchor=west] {$u_1$};
    \draw[thick,->]
    (v3) -- node[midway] (i2) {} ++(r7) node[anchor=west] {$u_1$};
    \draw[thick,->]
    (v3) -- ++(r8) node[anchor=south] {$u_2$};
    \draw[thick,->]
    (v4) -- ++(r5) node[anchor=east] {$u_3$};
    \draw[thick,->]
    (v4) -- ++(r8) node[anchor=south] {$u_2$} node[anchor=north east] {$\Trop(X_{1,P})$};
    \draw[thick,blue]
    (3,-1.5) -- (3,3.5) node[anchor=north west] {$\Trop(X_{2,P})$};
    \fill (v1) circle (2pt);
    \fill (v2) circle (2pt);
    \fill (v3) circle (2pt);
    \fill (v4) circle (2pt);
    \fill[white,draw=black] (i1) circle (2pt);
    \fill[white,draw=black] (i2) circle (2pt);
    \node[anchor=south west] at (v1) {$v_3$};
    \node[anchor=south east] at (v2) {$v_4$};
    \node[anchor=north east] at (v3) {$v_1$};
    \node[anchor=north west] at (v4) {$v_2$};
    \node[anchor=west,inner sep=0pt] at (6,1)
    {
      \begin{math}
        \begin{aligned}
          v_1 &=\lambda\cdot (0,0,0,0)=0\\
          v_2 &=\lambda\cdot (1,1,-1,-1)\\
          v_3 &=\lambda\cdot (2,0,0,-2)\\
          v_4 &=\lambda\cdot (1,-1,1,-1)\\
          u_1 &=(-1,-1,1,1)\\
          u_2 &=(-1,1,-1,1)\\
          u_3 &=(1,1,-1,-1)\\
          u_4 &=(1,-1,1,-1)\\
        \end{aligned}
      \end{math}
    };
  \end{tikzpicture}\vspace{-3mm}
  \caption{The fiberwise tropicalizations of \cref{ex:genusTwoExample} and their intersection.}
  \label{fig:genusTwoExample}
\end{figure}

We close this subsection with a few easy corollaries of \cref{thm:TropicalGenericFlatness2} and \cref{thm:MainThm1} combined.
\cref{cor:CombineGenericandMain} states that tropical flatness for \cref{thm:MainThm1} is not necessary, as long as the other condition of having a bounded prevariety holds in an open set.  This is significant, as tropical flatness can be difficult to test.

\begin{corollary}\label{cor:CombineGenericandMain}
  Let $X_1,\dots,X_k$ be generically Cohen-Macaulay, pure and of complementary dimension, and let $X=\bigcap_{i=1}^k X_i$.
  Suppose that $\bigcap_{i=1}^{k}\trop(X_{i,P})$ is bounded for $P$ in a non-empty open subset of $Y^{\mathrm{an}}$.
  Then there is a non-empty open subset $U$ of $Y^{\mathrm{an}}$ such that the generic root count $\ell_{X,\eta}$ is $\prod_{i=1}^{k}\trop(X_{i,Q})$ for all $Q\in{U}$.
\end{corollary}

The next \cref{cor:NoBoundedTrop} shows that if the generic root count does not attain the tropical intersection number, then fiberwise tropicalizations will generically not %
intersect in a bounded set.
This why we will be turning to tropical modifications in \cref{sec:modifications}.

\begin{corollary}\label{cor:NoBoundedTrop}
  Let $X_1,\dots,X_k$ be generically Cohen-Macaulay, pure and of complementary dimension, and let $X=\bigcap_{i=1}^k X_i$.  Suppose that $\ell_{X,\eta}\neq \prod_{i=1}^{k}\trop(X_{i,Q})$ for $P$ in a dense open subset $W$ of $Y^{\mathrm{an}}$. Then there is a dense open subset $U\subseteq Y^{\mathrm{an}}$ such that the tropical prevariety $\bigcap_{i=1}^{k}\trop(X_{i,Q})$ is unbounded for $Q\in U$.
\end{corollary}
\begin{proof}
  Consider the intersection $V=(\bigcap_{i=1}^{k} V_{i})\cap W$ of the dense open loci from \cref{thm:TropicalGenericFlatness2} and $W$. This is again dense, and the tropical fibers of $X$ over $V$ are necessarily unbounded by \cref{thm:MainThm1}.
\end{proof}

Finally, as an easy consequence of \cref{thm:MainThm1}, we obtain another proof of the Bernstein-Koushnirekno Theorem:

\begin{corollary}\label{thm:SquareBounded}
Let $X_{i}=V(f_{i})$ be $n$ hypersurfaces in an $n$-dimensional relative torus given by polynomials $f_{i}=\sum_\alpha c_{i,\alpha}x^{\alpha}$, and let $X=\bigcap_{i=1}^{n} X_i$.
Suppose that the tropical prevariety $\bigcap_{i=1}^{n}\trop(V(f_{i})_P)$ is bounded for some $P\in{Y^{\mathrm{an}}\backslash\bigcup_{\alpha,i}{V(c_{i,\alpha})}}$.
Then the generic root count of $X$ is the normalized mixed volume $\mathrm{MV}(f_{1},\dots,f_{n})$.
\end{corollary}
\begin{proof}
  The $V(f_{i})$'s are tropically flat (\cref{cor:TropConstant}), %
  generically Cohen-Macaulay and %
  pure of relative dimension $n-1$ over the given locus. %
  We can thus use \cref{thm:MainThm1} to conclude that the generic root count is the tropical intersection number. By \cite[Theorem 4.6.8]{MS15}, %
  this is the mixed volume, so we obtain the statement of the corollary.
\end{proof}

\begin{corollary}[Bernstein-Kushnirenko]\label{cor:BKK}
Let $X$ be a square universal family with fixed monomial supports. Then the generic root count equals $\ell_{X,\eta}=\mathrm{MV}(f_{1},\dots,f_{n})$.
\end{corollary}
\begin{proof}
  One easily finds a point $P\in{Y^{\mathrm{an}}}$ for which the tropicalizations intersect in finitely many points.   The statement then follows from \cref{thm:SquareBounded}.

  The generic finiteness of the tropical prevariety will be explained in greater generality in \cref{sec:TorusEquivariant} using the notion of torus-equivariance, see \cref{prop:TranslationProperCM}. Here, we note that the $X_{i}$'s are indeed all torus-equivariant, since the coefficients in front of the monomials are all free and independent.
\end{proof}

\subsection{Torus-equivariant systems and generic root counts}\label{sec:TorusEquivariant}
In this section, we introduce the notion of torus-equivariance and parametric independence, which give a natural condition under which \cref{thm:MainThm1} holds. It will be used in Sections \ref{sec:modifications} and \ref{sec:linearDependencies} to obtain generic root counts for certain classes of systems.

\begin{definition}

  Recall that $T$ is an $n$-dimensional torus over $Y$, i.e., $T=T_K\times Y$ where $T_K$ is the $n$-dimensional torus over $K$ with $K$-valued points $T_{K}(K)=(K^\ast)^n$.
  Let
  \begin{math}
    m\colon T_K\times T_K\to T_K
  \end{math}
  be the natural multiplication map on $T_K$.
  We say that $X\subseteq T$ is \emph{torus-equivariant}, if there is a group action on the parameter space $\rho\colon T_K\times Y\to Y$ such that, under the two morphisms
  \begin{center}
    \begin{tikzpicture}
      \node (Z1) {$T_K\times T_K\times Y$};
      \node[anchor=base west, xshift=20mm] (Z2) at (Z1.base east) {$ T_K\times T_K\times Y$};
      \draw[->] (Z1.north east) to[out=30,in=150] node[above] {$h_1$} (Z2.north west);
      \draw[->] (Z1.south east) to[out=330,in=210] node[below] {$h_2$} (Z2.south west);
      \node[anchor=south,yshift=10mm] (txPtop) at (Z1.north) {$(t,x,P)$};
      \node[anchor=south,yshift=10mm] (txPtargettop) at (Z2.north) {$(t,m(t,x),P)$};
      \draw[|->] (txPtop) -- (txPtargettop);
      \node[anchor=north,yshift=-10mm] (txPbot) at (Z1.south) {$(t,x,P)$};
      \node[anchor=north,yshift=-10mm] (txPtargetbot) at (Z2.south) {$(t,x,\rho(t,P))$};
      \draw[|->] (txPbot) -- (txPtargetbot);
    \end{tikzpicture}
  \end{center}
  we have $h_{1}(T_K\times X)=h_{2}(T_K\times X)$ as closed subschemes of $T_K\times T_K\times Y$.   Here the top and bottom maps are defined on $R$-valued points of $T_{K}\times T_{K}\times Y$, where $R$ is a $K$-algebra\footnote{This uniquely determines the desired morphisms $h_{i}$ by the Yoneda lemma. Alternatively, we can define it on $K$-valued points and use that this defines morphisms uniquely for maps of varieties over $K$. }.   If the actions are clear from context, we will also simply write $t\cdot x:=m(t,x)$ and $t\cdot P\coloneqq\rho(t,P)$.
\end{definition}

\begin{example}
  Consider the parameter ring $A=K[a_{0},a_{1},a_{2}]$ and the polynomial $f=a_{0}x_1+a_{1}x_{1}x_{2}+a_{2}x_{2}\in A[x_{1}^{\pm},x_{2}^{\pm}]=A_{0}$. We write $C=A_{0}[t_{1}^{\pm},t_{2}^{\pm}]$ for the coordinate ring of $Z\coloneqq T_K\times T_K\times Y$. The map $h_{1}\colon Z\to Z$ corresponds to the following map on the coordinate rings
  \begin{equation*}
    h^{*}_{1}\colon C\to C,\quad a_i\mapsto a_i, \quad t_i\mapsto t_i, \quad x_i\mapsto t_{i}x_i.
  \end{equation*}
  Let
  \begin{equation*}
    g\coloneqq \dfrac{a_{0}}{t_{1}}x_1+\dfrac{a_{1}}{t_{1}t_{2}}x_{1}x_{2}+\dfrac{a_{2}}{t_{2}}x_{2},
  \end{equation*}
  so that $h^{*}_{1}(g)=f$ and thus $V(g)=h_{1}(T_K\times X)$. We now similarly define a map $h^{*}_{2}:C\to C$  on the level of coordinate rings by
  \begin{align*}
    h^{*}_{2}\colon C\to C,\quad a_{i}\mapsto
    \begin{cases}
      t_{1}a_{0}&\text{for }i=0\\
      t_{1}t_{2}a_{1}&\text{for }i=1\\
      t_{2}a_{2}&\text{for }i=2\\
    \end{cases}
    , \quad t_i\mapsto t_i, \quad x_i\mapsto x_i.
  \end{align*}
  Note that this is induced by a group action $\rho$ of $T$ on $Y$. We then similarly have $h^{*}_{2}(g)=f$, so that $h_{1}(T\times X)=h_{2}(T\times X)$, and thus $X$ is torus-equivariant with respect to $\rho$. %
\end{example}

\begin{remark}
  Let $K\to L$ be a field extension and let $(t,P)$ be an $L$-valued point of $T_K\times Y$. We then have an equality of closed subschemes
  \begin{equation*}
    t\cdot X_{P}=X_{t\cdot P}.
  \end{equation*}
  In other words, we can obtain toric translates of the fiber $X_P$ through an appropriate action on the parameter $P$. In particular, if we represent
  a point in $T^{\an}\times Y^{\an}$ by an $L$-valued point $(t,P)\in T_{K}(L)\times Y(L)$, then both $X_{P}$ and $X_{t\cdot P}$ can be considered as schemes over $L$. This allows us to consider $(t\cdot X_{P})^{\an}$, which can be identified in a natural way with $t\cdot X^{\an}_{P}$ and $X^{\an}_{t\cdot P}$.   %
\end{remark}

To prove the main result of this subsection, we will need the following \cref{prop:TranslationProperCM} that essentially states that toric translations are sufficient to guarantee the requirement for \cref{thm:MainThm1}.

\begin{lemma}\label{lem:GenericStableIntersection}
  Let $\Sigma_{1}$ and $\Sigma_{2}$ be two balanced  polyhedral complexes in $\mathbb{R}^{n}$. There is a dense open subset $U\subseteq{\mathbb{R}^{n}\times\mathbb{R}^{n}}$ such that for all $(\lambda_{1},\lambda_{2})\in{U}$ the translates $\lambda_{1}+\Sigma_{1}$ and $\lambda_{2}+\Sigma_{2}$ intersect transversally. This set moreover contains $\{0\}\times{U_{1}}$ and $U_{2}\times{\{0\}}$ for suitable dense open subsets $U_{i}\subseteq \mathbb{R}^{n}$. %
\end{lemma}
\begin{proof}
  This follows from the proof of \cite[Proposition 3.6.12]{MS15}.
\end{proof}

\begin{lemma}\label{prop:TranslationProperCM}
  Let $X_{1}, \dots, X_{k}$ be pure closed subschemes of complementary dimension in $T_K$.
  Then there is a non-empty open subset $U\subseteq (T^{\an}_K)^{k-1}$ such that for $t_1\coloneqq 1\in T^{\an}_K$ and $t=(t_{2},\dots ,t_{k})\in{U}$, we have
  \begin{enumerate}
  \item the intersection $\bigcap_{i=1}^k t_{i}X^{\an}_{i}$ is finite and lies in the Cohen-Macaulay locus of each $t_iX^{\an}_i$,
  \item the tropicalizations of the $t_{i}X^{\an}_{i}$'s intersect transversally.
  \end{enumerate}
\end{lemma}
\begin{proof}
  Let $Z_i$ be the non-Cohen-Macaulay locus of $X_i$, which are proper closed subsets of the $X_{i}$'s by \cite[\href{https://stacks.math.columbia.edu/tag/00RG}{Lemma 00RG}]{stacks-project}. Consider the stable intersection of the tropicalizations of $X_{1}, X_{2}, \dots, X_{k-1}$ and $Z_{k}$. By \cite[Theorem 3.6.10]{MS15} and the assumption on the codimensions, this is empty. The other $k-2$ combinations where $X_{i}$ is replaced by $Z_{i}$ also give empty stable intersections. For each of these, we obtain a dense open subset $V_{j}\subseteq\mathbb{R}^{n(k-1)}=(\mathbb{R}^n)^{k-1}$ such that
  \begin{equation*}
   \trop(X_{1})\cap  (\lambda_{j}+\trop(Z_{j}))\cap\bigcap_{i\neq{1,j}}(\lambda_{i}+  \trop(X_{i}))=\emptyset
  \end{equation*}
  for $(\lambda_{2},\dots,\lambda_{k-1})\in V_{j}$ by \cref{lem:GenericStableIntersection}. %
We intersect these dense open subsets to obtain a dense open subset $V$. Now let $V'$ be the dense open subset obtained from \cref{lem:GenericStableIntersection} applied to $\trop(X_{1}),\dots,\trop(X_{k})$. Then $V'\cap V$ is again a dense open subset and   %
we immediately find that $U=\trop^{-1}(V\cap V')\subseteq (T^{\an}_{K})^{k-1}$ has the desired properties.
\end{proof}

Finally, we require the notion parametric independence, which will ensure that we can torically translate the fibers independently of each other.

\begin{definition}\label{def:ParametricallyIndependent}
  Let $X_1,\dots,X_k$ be closed subschemes of $T$. We say that $X_1,\dots,X_k$ are \emph{parametrically independent} if there are parameter spaces $Y_i$ and closed subschemes $X_i'\subseteq T_{Y_i}\coloneqq T_K\times Y_i$ such that $Y = \prod_{i=1}^k Y_i$ and $X_i = X'_{i}\times_{Y_{i}}Y$.
\end{definition}

\begin{proposition}\label{thm:TorusEquivariant}
   Let $X_1,\dots,X_k$ be closed subschemes of $T$ and let $X=\bigcap_{i=1}^k X_i$. Suppose the $X_{i}$'s %
  are parametrically independent, generically pure of complementary dimension, and that $X_2,\dots,X_k$ are torus-equivariant. Then $X$ is generically finite with generic root count $\ell_{X,\eta}=\prod_{i=1}^{r}\trop(X_{i,P})$ for $P\in U^{\an}$, where $U\subseteq Y$ is a Zariski dense open subset of $Y$. %
\end{proposition}
\begin{proof}
  For every $X_{i}$, we take a set of generators $f_{i,j}$ of the corresponding ideal. We write $U_{0}$ for the open subset of $Y$ over which the coefficients of the monomials of the $f_{i,j}$'s are non-zero. Let $U_{1}\subseteq{Y}$ be the open subset provided by Grothendieck's generic freeness theorem and let $U_{2}$ be the dense open subset over which the $X_{i}$ are pure.

  Let $U=U_{0}\cap{}U_{1}\cap{U_{2}}$. We will show as in the proof of \cref{thm:MainThm1} that the root count of $X$ over any $Q\in U^{\an}$ is equal to the tropical intersection number of the $\trop(X_{i,Q})$'s.  %
  Let $Q\in U^{\an}$. %
  By taking a non-archimedean field extension, we can assume that $Q$ is rational and consider $Q$ as a closed point of $U$. Since the family is parametrically independent, we can find points $Q_{i}\in{Y^{\mathrm{an}}_{i}}$ that give rise to $Q$. As above, we will consider the $Q_{i}$'s as closed points of $Y_{i}$.
  We write $X_{i,Q_{i}}$ for the fibers of the families and $Z_{i,Q_{i}}$ for their non-CM-loci. We note here that the $Q_{i}$'s (and thus the $X_{i,Q_{i}}$'s) are fixed for the remainder of the proof.

Recall that $T_{K}$ is the $n$-dimensional torus over $K$.
  By \cref{prop:TranslationProperCM}, there is an open subset $V$ of $(t_{1},\dots,t_{k})\in(T_{K}^{\mathrm{an}})^k$ such that %
  the $t_{i}X^{\an}_{i,Q}$'s meet in the CM-locus of each. Moreover, their tropicalizations meet transversally in finitely many points.
  Consider the torus-action
  \begin{equation*}
    T_{K}\times{Y_{i}}\to{Y_{i}}
  \end{equation*}
  for each $i$. Applying the projection maps $Y\to{Y_{i}}$ to the dense open subset $U$, we obtain dense open subsets $U_{i}\subseteq{Y_{i}}$.

  Consider the map $T_{K}\to{Y_{i}}$ sending $(t_{i})\mapsto (t_{i}\cdot Q_{i})$. Here we used the action of $T_{K}$ on every $Y_{i}$.   %
  The inverse image of $U_{i}$ under this map is open and non-empty, and thus dense by \cref{lem:DensityZariski}. Its analytification thus intersects the projection $V^{\an}_{i}$ of $V^{\an}$. By doing this for all $i$, we obtain a new point $Q'\in V^{\an}$ with induced $Q'_{i}\in Y^{\an}_{i}$ such that the $X_{i,Q'_{i}}$'s meet in the CM-locus of each.

  The $\trop(X_{i,Q'_{i}})$'s are translates of the $\trop(X_{i,Q_{i}})$'s so that the tropical intersection number is unchanged.  Using \cref{lem:IntersectionMultiplicities}, we then see that the root count of $X$ over $Q'$ is the same as the global intersection multiplicity of the $X_{i,Q'}$'s. But by \cite[Corollary 6.13]{OR2013}, this is the tropical intersection number of the $\trop(X_{i,Q'})$'s, which is the tropical intersection number of the $\trop(X_{i,Q})$'s, as desired.  %
\end{proof}

\begin{example}
  All examples so far satisfy the requirements for \cref{thm:TorusEquivariant}:
  Both $X_1$ and $X_2$ from \cref{ex:runningExampleFiberwiseTropicalization} as well as $X_1$ and $X_2$ from \cref{ex:nonSquare} are parametrically independent because the parameters in the definitions of the ideals of $X_1$ and $X_2$ are disjoint.  Moreover, in both cases, $X_2$ is torus-equivariant.
\end{example}

We conclude this section with two remarks on \cref{thm:TorusEquivariant}.

\begin{remark}
  Torus-equivariance is needed in \cref{thm:TorusEquivariant} for avoiding the non-Cohen-Macaulay loci.  For systems that are Cohen-Macaulay everywhere, such as square systems, it already suffices if the tropicalizations are torus-equivariant, i.e., if for generic $P\in Y$ and every $t\in T_K$, there is a point $Q\in Y$ such that $\trop(X_{Q})=\trop(t\cdot X_{P})=\trop(t)+\trop(X_{P})$.
\end{remark}

\begin{remark}\label{rem:TropicallyFlatTorusEquivariant}
In this remark, we explain how the Zariski dense open subset from  \cref{thm:TorusEquivariant} can be interpreted in terms of tropically flat points, and in terms of ordinary $K$-valued points, even when $K$ is trivially valued.

Suppose the $X_{i}$'s are tropically flat around $P$. Then the generic root count is the tropical intersection number of the $\trop(X_{i,P})$. Indeed, the tropicalizations of the $X_{i}$'s are locally constant by \cref{cor:TropConstant}, and thus they will intersect the dense open subset from \cref{thm:TorusEquivariant}. We thus see that any tropically flat point $P$ automatically gives rise to the generic root count.

    Note that \cref{thm:TorusEquivariant} does not require $K$ to be non-trivially valued. In particular, we find the following. Let $U\subseteq Y$ be the Zariski open dense subset of \cref{thm:TorusEquivariant}. Then for any $P\in U(K)\subseteq U^{\an}$, we have that the conclusion of \cref{thm:TorusEquivariant} holds. These $P$ form an open dense subset of the variety $Y$ in the classical sense, so that the generic root count is realized as a tropical intersection product by a generic set of classical points.
\end{remark}

\begin{remark}\label{rem:AnalyticFamilies}
  We note here that the results in Section \ref{sec:TropGeneric} can be generalized to analytic families of polynomial equations. For many applications in practice this is important, since the functions in the parameters naturally contain analytic functions. For instance, one can consider the system
  \begin{align*}
    f_{1}&=\sin(a_{1}+a_{2})x^2+\sin(a_{1}a_{2})y^2+\cos(a_{1})x+a_{4}y+a_{5},\\
    f_{2}&=\cos(b_{1})x^2+\cos(b_{1}+b_{2})y^2+b_{3}x+b_{4}y+b_{5}
  \end{align*}
  over the ring $A=\mathbb{C}[[a_{i},b_{i}]]$, where $\mathbb{C}$ is trivially valued. This is a $\mathbb{C}$-affinoid domain, so the material in \cite[Section 2]{Berkovich1993} is again applicable. Note that the coefficients of the monomials in these types of equations can satisfy algebraic relations that are not always apparent. The tropical material presented here is however directly applicable, and we can find a ring homomorphism $A\to\mathbb{C}[[t]]$ with corresponding point $P$ such that the fiberwise tropical prevariety over $P$ is finite. This then implies that the generic root count is the mixed volume of the Newton polytopes of the polynomials, which is $4$.
  More generally, one can consider polynomial equations defined over affinoid $K$-algebras $\mathcal{A}$, where $K$ is any (complete) non-archimedean field. The material in this paper directly extends to this more general scenario.
\end{remark}
\section{Generic root counts of square systems}\label{sec:modifications}
In this section, we focus on square systems for which we can express the generic root count in terms of tropical intersection numbers.  These systems include the steady-state equations of chemical reaction networks and the birational intersection indices $[L_{1},\dots,L_{n}]$ studied by Kaveh and Khovanskii in \cite{KavehKhovanskii2010,KavehKhovanskii2012}, two applications which will be discussed in more detail in \cref{sec:linearDependencies}.

As seen in \cref{cor:NoBoundedTrop}, tropical hypersurfaces of square systems whose generic root count is below the mixed volume will not intersect each other transversally. This shows that the individual equations are bad for our tropical techniques, which is why we turn to appropriate reembeddings, also referred to as \emph{tropical modifications} \cite{Mikhalkin2006,Kalinin15}.

Recall \cref{con:mainConvention}, namely that we work with an integral affine parameter space $Y=\Spec(A)$, a subscheme $X=\Spec(A[x_1^\pm,\dots,x_n^\pm]/I)$ of a relative torus $T=\Spec(A[x_1^\pm,\dots,x_n^\pm])$ over $Y$. Similar to the works of Kaveh and Khovanskii \cite{KavehKhovanskii2010,KavehKhovanskii2012}, we may need to restrict to a dense open subset $U\subseteq T$ and consider the generic root count of $X\cap U$ instead of $X$.

\begin{definition}\label{def:UniversalModification}
  Let $f_1,\dots,f_n\in A[x_1^\pm,\dots,x_n^\pm]$. Choose $p_{i,j}\in{A}$ and $q_{j}\in K[x_{1}^{\pm},\dots,x_{n}^{\pm}]$ such that
  \begin{equation}\label{eq:polynomialRepresentation}
    f_i = \sum_{j=1}^m p_{i,j}\cdot q_{j}.
  \end{equation}
  Note that some $p_{i,j}$ may be zero and the $q_j$'s are not necessarily pairwise distinct.

  For any open subset $U$ of $T$, we write $X\cap U\coloneqq X\cap{U}$, which we again consider as a $Y$-scheme. Let
  \begin{equation*}
    \widehat C \coloneqq A\Big[x_i^\pm, w_{j}^\pm \mid i\in[n], j\in[m]\Big],
  \end{equation*}
  and consider in $\widehat C$
  \begin{equation*}
    \hat{f}_{i}\coloneqq\sum_{j=1}^{m} p_{i,j}w_{j} \text{ for } i\in[n]\qquad\text{and}\qquad \hat{h}_{j}\coloneqq w_{j}-q_{j} \text{ for } j\in[m].
  \end{equation*}
  Let $\widehat T\coloneqq\Spec(\widehat C)$ be a larger relative torus over $Y$, and consider the subscheme $\widehat X\coloneqq\Spec(\widehat B)\subseteq \widehat T$ where $\widehat B\coloneqq \widehat C/\langle\hat f_{i}, \hat h_{j}\mid i\in[n],j\in[m]\rangle$.  We have a natural decomposition $\widehat X=\widehat X_\lin\cap\widehat X_\nlin$ for $\widehat X_{\lin} = V(\widehat I_{\lin})$ and $\widehat X_{\nlin} = V(\widehat I_{\nlin})$ where $\widehat I_{\lin}=\langle \hat f_1,\dots,\hat f_n\rangle$ and $\widehat I_{\nlin} = \langle \hat h_{1},\dots,\hat h_m\rangle$. Note that $\widehat X_{\lin}$ is linear, and that $\widehat X_{\nlin}$ is constant over $Y$, which makes $\widehat X_{\lin}$ and $\widehat X_{\nlin}$ parametrically independent. We refer to $\widehat{X}$ as the \emph{modification} of $X$ derived from the representation in \cref{eq:polynomialRepresentation}.
\end{definition}

Note that the modification $\widehat X$ depends on the chosen $p_{i,j}$'s and $q_j$'s. If the subschemes $\hat{X}_{\mathrm{lin}}$ and $\hat{X}_{\mathrm{nlin}}$ satisfy the conditions of \cref{thm:TorusEquivariant}, then we immediately obtain a formula for its generic root count:

\begin{corollary}\label{cor:genericRootCountModification}
  Suppose $\widehat{X}_{\mathrm{lin}}$ is torus-equivariant and of generic codimension $n$. Then $\widehat{X}$ is generically finite and $\ell_{\widehat X,\eta}=\trop(\widehat{X}_{\mathrm{lin},P})\cdot\trop(\widehat{X}_{\mathrm{nlin},P})$ for $P\in Y^{\mathrm{an}}$ generic.
\end{corollary}
\begin{proof}
  To use \cref{thm:TorusEquivariant}, we only have to verify that $\widehat{X}_{\mathrm{lin}}$ and $\widehat{X}_{\mathrm{nlin}}$ are generically pure and of complementary dimensions.  Generic purity is straightforward for the linear $\widehat{X}_{\mathrm{lin}}$ and it follows for $\widehat{X}_{\mathrm{nlin}}$ from the fact that is isomorphic to an open subset of the $n$-dimensional torus over $Y$.  As for the complimentary dimensions, we have $\codim(\widehat{X}_{\mathrm{lin}})+\codim(\widehat{X}_{\mathrm{nlin}})=n+m=\dim(\widehat T)$.
\end{proof}

\noindent
We now focus on three aspects of \cref{cor:genericRootCountModification}
which will serve as guides for the remainder of this section:
\begin{enumerate}
\item[\customlabel{enumitem:genericCodimension}{(a)}] The assumption that $\widehat{X}_{\mathrm{lin}}$ is of generic codimension $n$.
\item[\customlabel{enumitem:torusEquivariance}{(b)}] The assumption that $\widehat{X}_{\mathrm{lin}}$ is torus-equivariant.
\item[\customlabel{enumitem:compatibleSubset}{(c)}] The fact that \cref{cor:genericRootCountModification} gives a formula for $\ell_{\widehat X,\eta}$ and not $\ell_{X,\eta}$.
\end{enumerate}
In general, Assumptions \ref{enumitem:genericCodimension} and \ref{enumitem:torusEquivariance} need not be satisfied. %
However, \cref{lem:GenFiniteImpliesGenCI} shows that Assumption \ref{enumitem:genericCodimension} is satisfied in all cases of interest and \cref{lem:TorusEquivariantModifications} shows that Assumption \ref{enumitem:torusEquivariance} is guaranteed by a set of natural algebraic conditions on the $p_{i,j}$. For \ref{enumitem:compatibleSubset}, we will see that $\ell_{X\cap U,\eta}=\ell_{\widehat{X},\eta}$ for a dense open subset $U\subseteq{T}$. This open subset also plays an important role in \cite{KavehKhovanskii2010,KavehKhovanskii2012}.

\begin{lemma}\label{lem:GenFiniteImpliesGenCI}
  Suppose that $X\cap U$ is generically finite for some dense open $U$.  Then $\widehat{X}_{\mathrm{lin}}$ is of generic codimension $n$.
\end{lemma}
\begin{proof}
  Suppose that $\widehat{X}_{\mathrm{lin}}$ is not of generic codimension $n$. Then there is a linear relation over the function field of $Y$ among the $\hat{f}_{i}$'s. But this implies that there is a linear relation among the $f_{i}$'s, contradicting the fact that $X\cap U$ is generically finite.
\end{proof}

The following lemma offers an easy criterion under which the resulting $\widehat{X}_{\mathrm{lin}}$ is torus-equivariant.

\begin{lemma}\label{lem:TorusEquivariantModifications}
  Suppose that $Y=\mathrm{Spec}(A)$, where $A=K[a_{1},\dots,a_{m}]$. Suppose that the $p_{i,j}$'s in \cref{def:UniversalModification} are linear and that there are subrings $A_{1},\dots,A_m\subseteq{A}$ generated by linear polynomials with $A=A_{1}\otimes_K\dots\otimes_KA_m$ and $p_{i,j}\in{A_{j}}$.  Then $\widehat{X}_{\mathrm{lin}}$ is torus-equivariant.
\end{lemma}
\begin{proof}
  To define the action, we consider a set of independent linear generators $z_{j,k}$ of $A_{j}$.
  The torus action sends a vector $(w_{i})$ to $(\lambda_{i}w_{i})$. We then send $z_{j,k}$ to $z_{j,k}/\lambda_{k}$. One easily checks that $\widehat{X}_{\mathrm{lin}}$ is torus-equivariant with respect to this action.
\end{proof}

\begin{definition}\label{rem:TropicallyModifiable}
  If there exist a presentation as in \cref{eq:polynomialRepresentation} such that the resulting $\widehat{X}_{\mathrm{lin}}$ is torus-equivariant and of generic codimension $n$, then we say that $X$ is \emph{tropically rectifiable}.   %
\end{definition}

We now address point \ref{enumitem:compatibleSubset}.
As with the previous two, there is no guarantee that $\ell_{\widehat{X},\eta}=\ell_{X,\eta}$.
Instead, we will see that generic root count of the modification $\widehat X$ is the generic root count of $X\cap U$ for a dense open $U\subseteq T$. In the following the we will show what $U$ is, why $U$ is necessary, and when $X\cap U=X$.

\begin{theorem}\label{thm:GRCModification}
  Suppose that $X$ is tropically rectifiable and let $U\coloneqq \bigcap_{i=1}^{m} D(q_{i})$ for $q_{i}$ as in \cref{eq:polynomialRepresentation}. Then $X\cap U$ is generically finite and  $\ell_{X\cap U,\eta}=\ell_{\widehat{X},\eta}=\trop(\widehat{X}_{\mathrm{lin},P})\cdot \trop(\widehat{X}_{\mathrm{nlin},P})$ for $P\in Y^{\mathrm{an}}$ generic.
\end{theorem}
\begin{proof}
  Let $k(P)\to L$ be an algebraic closure of $k(P)$. Then the
  $L$-valued points of $\widehat{X}_P$ correspond exactly to the $L$-valued points of $(X\cap U)_P$ since the $q_{j}$'s are nonzero. Moreover, the multiplicities are preserved by the linear nature of the modification, hence the generic root counts coincide.
\end{proof}

\begin{example}\label{exa:LocalizeGRC}
  Note that without passing to $X\cap U$, \cref{thm:GRCModification} is not true in general. For instance, consider the system
  \begin{align*}
    f_{1}=a_{1}(x-1)+a_{2}(y-1) \qquad \text{and} \qquad
    f_{2}=a_{3}(x-1)+a_{4}(y-1)
  \end{align*}
  over $Y=\mathrm{Spec}(A)$ with $A=K[a_{1},a_{2},a_{3},a_{4}]$. The unique solution of this system for every choice of parameters with $a_{1}a_{4}-a_{2}a_{3}\neq{0}$ is $(1,1)$. The tropical intersection number of the modification is however $0$, which gives the number of solutions with $x\neq{1}$ or $y\neq{1}$.
\end{example}

We would now like to extend the open subset from \cref{thm:GRCModification} so that the tropical intersection number derived from $\widehat{X}_{\mathrm{lin}}$ and $\widehat{X}_{\mathrm{nlin}}$ still gives a valid generic root count. We start with a basic tool in extending this root count.

\begin{lemma}\label{lem:CombineOpenSetsGRC}
  Suppose that $U_{i}\supseteq{U_{0}}\coloneqq {\bigcap_{j=1}^{m}D(q_{j})}$ are open dense subsets of $T$ such that the $X\cap U_i$'s are generically finite with  $\ell_{X\cap U_i,\eta}=\ell_{X\cap U_0,\eta}$. Let $U\coloneqq\bigcup_{i} U_{i}$. Then $X\cap U$ is generically finite with $\ell_{X\cap U,\eta}=\ell_{X\cap U_0,\eta}$.
\end{lemma}
\begin{proof}
  Any point in the generic fiber of $X\cap U\to{Y}$ lies in the generic fiber of $X\cap U_i\to Y$ for some $i$. But these all lie in the generic fiber of $X\cap U_0\to Y$, which quickly gives the desired equality.
\end{proof}

It now follows that there is a largest open subset $U\supseteq{\bigcap_{j=1}^{m}D(q_{j})}$ such that $X\cap U$ has generic root count $\ell_{\widehat{X}/{Y}}$.
\begin{definition}
  Let $U$ be the union of all open subsets $U_{i}\supseteq \bigcap_{j=1}^{m}D(q_{j})$ such that $\ell_{X\cap U_i,\eta}=\ell_{\widehat{X},\eta}$. We call this the \emph{(maximal) rectifiable locus} of the modification.
\end{definition}

\cref{exa:LocalizeGRC} showed that $U$ can be a strict subset of $T$. In \cref{pro:ExtendGRCOpenSubset}, we will give a criterion that allows us to extend ${\bigcap_{j=1}^{m}D(q_{j})}$ to a larger open set $U$ with $\ell_{X\cap U,\eta}=\ell_{\widehat{X},\eta}$. We first give some preliminary definitions.

\begin{definition}
  Let $Z\coloneqq \mathrm{Spec}(A[x_{i}^{\pm},w_{j}\mid i\in[n],j\in[m]])$. Any subset $J\subseteq [m]$ gives rise to a toric stratum $H_J = \bigcap_{j\notin J}D(w_{j})\cap \bigcap_{j\in J} V(w_{j})$ of $Z$.
  Note that the $H_{J}$'s are mutually disjoint and  isomorphic to standard tori over $Y$. Let $\widehat{I}_{\mathrm{lin},a}$ and $\widehat{I}_{\mathrm{nlin},a}$ be the ideals generated by the $\widehat{f}_{i}$'s and $\widehat{h}_{j}$'s in $A[x_{i}^{\pm},w_{j}\mid 1\leq{i}\leq{n},1\leq{j}\leq{m}]$. We write $\widehat{X}_{\mathrm{lin},a}=V(\widehat{I}_{\mathrm{lin},a})$ and $\widehat{X}_{\mathrm{nlin},a}=V(\widehat{I}_{\mathrm{nlin},a})$ for the corresponding subschemes in $Z$.
\end{definition}

\begin{proposition}\label{pro:ExtendGRCOpenSubset}
  Suppose that $X$ is tropically rectifiable. Let $J\subseteq [m]$ be a subset and suppose that
  \begin{equation}
    \codim(\widehat{X}_{\mathrm{lin},a,\eta}\cap H_{J',\eta},H_{J',\eta})+\codim(\widehat{X}_{\mathrm{nlin},a,\eta}\cap{H}_{J',\eta},H_{J',\eta})>\mathrm{dim}(H_{J',\eta})
  \end{equation}
  for every $J'$ with $J'\cap J=\emptyset$. %
  Then we have that $\ell_{\widehat{X},\eta}=\ell_{X\cap U_J,\eta}$ for $U_J=\bigcap_{j\in J} D(q_j)$.
\end{proposition}
\begin{proof}
  We first note that every $\widehat{X}_{\mathrm{lin},a}\cap{H_{J'}}$ is torus-equivariant.  We now use \cite[Theorem 3.6.10]{MS15}, \cref{pro:DimensionsTransverseIntersections} and \cref{thm:TropicalGenericFlatness2} to conclude that $(\widehat{X}_{\mathrm{lin},a}\cap{H}_{J'})\cap (\widehat{X}_{\mathrm{nlin},a}\cap{H}_{J'})$ is generically empty.
  Suppose that there is a point $x$ in $(X\cap U_J)_{\eta}$ with $q_{j}(x)=0$ for $j\notin J$. By definition, we have $q_{i}(x)\neq{0}$ for all $i\in{J}$.  %
  Let $J'=\{i\in[m]:q_{i}(x)=0\},$ which is a non-empty subset of $[m]$ with $J'\cap J=\emptyset$. %
  Then $x$ gives rise to a point of %
  $(\widehat{X}_{\mathrm{lin},a,\eta}\cap{H}_{J',\eta})\cap (\widehat{X}_{\mathrm{nlin},a,\eta}\cap{H}_{J',\eta})$. This contradicts the fact that $(\widehat{X}_{\mathrm{lin},a}\cap{H}_{J'})\cap (\widehat{X}_{\mathrm{nlin},a}\cap{H}_{J'})$ is generically empty.
\end{proof}

\begin{remark}
  Note that $\codim(\widehat{X}_{\mathrm{lin},a,\eta}\cap{H_{J,\eta}},H_{J,\eta})$ is relatively easy to compute.
  Computing $\codim(\widehat{X}_{\mathrm{nlin},a,\eta}\cap{H_{J,\eta}},H_{J,\eta})$ on the other hand generally involves calculating a Gr\"{o}bner basis, which can be unfeasible.
  In certain key cases we can still easily give a good lower bound for $\codim(\widehat{X}_{\mathrm{nlin},a,\eta}\cap{H_{\eta}},H_{J,\eta})$ however, so that we can again apply \cref{pro:ExtendGRCOpenSubset}.
\end{remark}

\begin{corollary}\label{cor:MonomialMCS}
  Suppose $X$ is tropically rectifiable and that $q_{j}$ is monomial for all~$j$. Then $\ell_{X,\eta}=\ell_{\widehat{X},\eta}$.
\end{corollary}
\begin{proof}
  In this case $\widehat{X}_{\nlin,a}\cap H_{J'}$ is empty for all $J'$, so that the statement follows from \cref{pro:ExtendGRCOpenSubset}.  %
\end{proof}

\begin{corollary}\label{cor:OpenSetExtendGRC}
  Suppose $X$ is tropically rectifiable. Let $J\subseteq [m]$ be a subset such that the $n\times{|J|}$-matrix $(p_{i,j})$ for $j\in{J}$ is of row rank $n$ over $K(A)$.  Then we have that $\ell_{\widehat{X},\eta}=\ell_{X\cap U_J,\eta}$ for $U_J=\bigcap_{j\in J}D(q_{j})$.
\end{corollary}
\begin{proof}
  Let $J'$ be a subset with $J\cap J'=\emptyset$. The condition implies that the codimension of the linear space $\widehat X_{\lin,a}\cap H_{J'}$ is $n$. %
  We now note that $\mathrm{dim}(H_{J'})=n+m-|J'|$, and that $\codim(\widehat{X}_{\mathrm{nlin},a}\cap{H}_{J'},H_{J'})>m-|J'|+1$. Indeed, the latter follows since the polynomials $\hat{h}_{j}=w_{j}-q_{j}$'s for $j\notin{J'}$ give a torus of codimension $m-|J'|$, and the remaining polynomials give a space of codimension at least $1$ by Krull's Hauptidealsatz.
\end{proof}

\begin{example}
  Consider the system $X=V(f_{1},f_{2})$ over $Y=\mathrm{Spec}(\mathbb{C}[a_{1},a_{2},a_{3},a_{4}])$ given by the polynomials
  \begin{align*}
    f_{1}=a_{1}(x-1)+a_{2}(y-2)\qquad \text{and}\qquad
    f_{2}=a_{3}(x-1)+a_{4}(y-4).
  \end{align*}
  The open subsets we obtain from \cref{cor:OpenSetExtendGRC} are $D(x-1)$ and $D(y-2)\cap{D(y-4)}$. Moreover, their union $U$ is not the entire torus.
  Using \cref{pro:ExtendGRCOpenSubset} with $J=\emptyset$, we however easily see that the root count is also valid over $T$. We thus see that the maximal rectifiable locus can be strictly larger than the ones obtained from \cref{cor:OpenSetExtendGRC}.
\end{example}

We conclude this section with a few comments on the linear scheme $\widehat{X}_{\mathrm{lin}}$. By combining \cref{thm:MainThm1} with \cref{lem:FlatnessMatroids}, we obtain the following explicit root count formula: %
\begin{lemma}\label{lem:PlueckerCriterion}
  Consider the non-zero Pl\"{u}cker vectors $p_{I}$ and $q_{I}$ of $\widehat{X}_{\mathrm{lin}}$ as elements of the parameter ring $A$, see \cref{def:PlueckerVectors}. For any $P\in{Y^{\mathrm{an}}}$ such that $p_{I}(P)q_{I}(P)\neq{0}$ for all $I\subseteq [n]$, we have that %
  \begin{equation*}
    \ell_{\widehat{X},\eta}=\trop(\widehat X_{\mathrm{lin},P})\cdot\trop(\widehat X_{\mathrm{nlin},P}).
  \end{equation*}
\end{lemma}
\begin{proof}
  By \cref{lem:FlatnessMatroids}, $\hat{X}_{\lin,P}$ is tropically flat over $P$. Moreover, $\hat{X}_{nlin,P}$ is automatically tropically flat over $P$. We can then find a valued field extension $K\to L$, a point $P_{L}$ lying over $P$ and an open neighborhood $U_{L}$ of $P_{L}$ such that the tropicalizations of these schemes are constant and equal to the tropicalization over $P$ by \cref{cor:TropicallyFlatImpliesConstant}. But the open subset $U_{L}$ intersects the dense open subset
  \cref{thm:TorusEquivariant}, which gives the statement of the lemma.
\end{proof}

\begin{remark}
  \label{rem:matroidalDegree}
  Let $\Sigma$ be a fixed pure balanced polyhedral complex.  Let $\Sigma_\lin$ be a tropical linear space of complementary dimension, and let $\text{rec}(\Sigma_\lin)$ denote its recession fan, which is the Bergman fan of the underlying matroid of $\Sigma_\lin$ by \cite[Theorem 4.4.5]{MS15}.  As $\Sigma\cdot\Sigma_\lin=\Sigma\cdot\mathrm{rec}(\Sigma_\lin)$ by \cref{lem:tropicalIntersectionNumberRecessionFan}, we have $\Sigma\cdot\Sigma_\lin = \Sigma\cdot\Sigma_\lin'$ whenever the two tropical linear spaces $\Sigma_\lin, \Sigma_\lin'$ share the same underlying matroid.

  Consequently, the tropical intersection number $\trop(\widehat X_{\lin,P})\cdot\trop(\widehat X_{\nlin,P})$ in \cref{thm:GRCModification} can be regarded as a \emph{matroidal degree}.  Whereas the classically the degree of $\widehat X_{\nlin,P}$ equals the number of intersection points of $\widehat X_{\nlin,P}$ with a generic linear space of complementary dimension, the tropical intersection number equals the number of intersection points of $\widehat X_{\nlin,P}$ with a generic linear space with a fixed matroid.

  Hence, whether the generic root count of \cref{thm:GRCModification} decreases upon further specialization as discussed in \cref{sec:degeneracyGraph} depends primarily on whether or not it changes the matroid of $\widehat X_{\lin,P}$ for $P$ generic.
\end{remark}

\section{Applications to systems with linear dependencies}\label{sec:linearDependencies}
In this section, we explore square systems with linear dependencies between the coefficients of their polynomials. We will %
focus on two special types of dependencies which we call  vertical and horizontal dependencies. The first is inspired by the steady-state equations of chemical reaction networks \cite{Dickenstein16}, and the second is inspired by equations with fixed polynomial supports \cite{KavehKhovanskii2012}. %

\begin{assumption}\label{ass:SquareSystemsParameterSpace}
  Throughout this section, the parameter space $Y$ will be the $m$-dimensional affine space $\mathbb{A}^{m}=\Spec(K[a_1,\dots,a_m])$  %
  and the coefficients of the polynomials  $f_1,\dots,f_n\in K[a_1,\dots,a_m][x_1^\pm,\dots,x_n^\pm]$ will be linear and homogeneous in  $a_1,\dots,a_m$.
\end{assumption}

\subsection{Square systems with vertical parameter dependencies}\label{sec:verticalDependencies}
In this section, we consider a class of parametrized polynomial systems inspired by the steady-state equations of chemical reaction networks \cite{Dickenstein16}.

\begin{definition}\label{def:verticalDependencies}
  Let $x^{\alpha_1},\dots,x^{\alpha_m}$ be the monomials of $f_1,\dots,f_n$.
  We say $f_1,\dots,f_n$ have \emph{vertical parameter dependencies}, if there is a decomposition $A=A_{1}\otimes_K\dots\otimes_KA_m$ such that each $f_i$ is of the form
  \[ f_i=\sum_{j=1}^m p_{i,j}\cdot x^{\alpha_j} \qquad \text{with } p_{i,j}\in A_j. \]
  In other words, the coefficient matrix $(p_{i,j})_{i\in[n],j\in[m]}\in A^{n\times m}$ has algebraic dependencies along its columns, but not rows.
\end{definition}

\begin{definition}\label{def:verticalDependenciesModification}
  Let $f_1,\dots,f_n$ have vertical parameter dependencies. Choosing $q_j\coloneqq x^{\alpha_j}$ in \cref{def:UniversalModification} we obtain in $\widehat C \coloneqq A[x_i^\pm, w_j^\pm \mid i\in[n], j\in[m]]$:
  \begin{equation*}
    \hat{f}_{i}\coloneqq\sum_{j=1}^m p_{i,j}w_j \text{ for } i\in[n] \qquad \text{and}\qquad
    \hat{h}_{j}\coloneqq w_{j}-x^{\alpha_j} \text{ for } j\in[m].
  \end{equation*}
  Let $\widehat X_\lin \coloneqq V(\langle \hat f_1,\dots, \hat f_n\rangle)$ and $\widehat X_\nlin\coloneqq V(\langle \hat h_1,\dots, \hat h_m\rangle)$.  By \cref{ass:SquareSystemsParameterSpace} and \cref{lem:TorusEquivariantModifications}, $\widehat X_\lin$ is torus-equivariant. Moreover, $\trop(\widehat X_{\nlin,P})$ is a linear subspace in $\RR^n$ independent of $P$ since the $q_{j}$'s are independent of the parameters, by construction. We will refer to $\widehat X\coloneqq \widehat X_\lin\cap\widehat X_{\nlin}$ as the \emph{modification} for vertical dependencies.
\end{definition}

Before we turn to generic root counts, we would like to highlight a particular class of nice tropical varieties:
\begin{definition}
  \label{def:tropicalCompleteIntersection}
  Let $\Sigma_0$ be a balanced polyhedral complex in $\RR^n$.  We say $\Sigma_0$ is a \emph{tropical complete intersection}, if there are tropical hypersurfaces $\Sigma_1,\dots,\Sigma_k$ with
  \begin{equation*}
    \Sigma_k=\Sigma_1\cap_\st\dots\cap_\st\Sigma_k.
  \end{equation*}
\end{definition}

For vertical dependencies, $\trop(\widehat X_{\nlin,P})$ is a tropical complete intersection, which is easier to work with than an arbitrary tropical variety:

\begin{proposition}\label{prop:verticalDependenciesRootCount}
  Let $f_1,\dots,f_n$ have vertical parameter dependencies. Then, for generic $P\in Y^\an$ as described in \cref{lem:PlueckerCriterion} we have
  \begin{equation*}
      \ell_{X,\eta}=\trop(\widehat X_{\lin,P})\cdot \prod_{j=1}^m \trop(V(\hat{h}_{j})_{P}).
  \end{equation*}
\end{proposition}
\begin{proof}
  Follows from \cref{thm:GRCModification} and \cref{cor:MonomialMCS}.
\end{proof}

\cref{prop:verticalDependenciesRootCount} allows us to compute the generic root count via mixed volumes, eliminating the need to compute any set-theoretic intersections of tropical varieties:

\begin{remark}\label{rem:verticalDependenciesRootCount}
  As explained in the first paragraph of \cref{rem:matroidalDegree}, the intersection product with $\trop(\widehat X_{\lin,P})$ equals the intersection product with $\trop(M)$, where $\trop(M)$ denotes the Bergman fan of the underlying matroid of $\widehat X_{\lin,P})$.  For some matroids $M$, we further have $\trop(M)=\trop(V(\hat \ell_1))\cap_\st\dots\cap_\st\trop(V(\hat \ell_n))$ for linear polynomials $\hat \ell_i$, making $\ell_{X,\eta}=\MV(\hat \ell_1,\dots, \hat \ell_n,\hat h_1,\dots, \hat h_m)$ by \cite[Theorem 4.6.9]{MS15}.  These matroids are dual to the so-called \emph{transversal matroids} \cite[Proposition~3.4 and Corollary~5.6]{FinkRinconStiefel2015}. If $M$ is not dual to a transversal matroid, then it can be expressed as a signed linear combination of cotransversal matroids, see \cite{Hampe2017}.   Consequently, the generic root count coincides with a signed linear combination of mixed volumes.
\end{remark}

We conclude this section with an example that showcases how our techniques can be applied to steady-state equations of reaction networks.

\begin{example}\label{ex:CRN}
  Consider the following reaction network:
  \[\small X_{i-1}+E \uset{d_{1,i}}{\oset{a_{1,i}}{\xrightleftharpoons{\hspace{7mm}}}} Y_{1,i} \oset{k_{1,i}}{\xrightarrow{\hspace{7mm}}} X_{i}+E \quad X_{i}+F \uset{d_{2,i}}{\oset{a_{2,i}}{\xrightleftharpoons{\hspace{7mm}}}} Y_{2,i} \oset{k_{2,i}}{\xrightarrow{\hspace{7mm}}} X_{i-1}+F \quad\text{for } i\in[n]. \]
  It describes the standard model of $n$-sites phosphorylation, see for example \cite[Section 2]{FeliuRendallWiuf2020}. Under the laws of mass-action kinematics, the corresponding evolution equations are given by
  \allowdisplaybreaks
  \begin{align*}
    \dot{x}_{i}&=f_i\coloneqq -a_{1,i+1}x_i x_{ E} -a_{2,i}x_ix_{ F}+d_{1,i+1}y_{1,i+1}+d_{2,i}y_{2,i}  +k_{1,i}y_{1,i}+k_{2,i+1}y_{2,i+1}\\ &\hspace{90mm}\text{for }i\in \{0\}\cup [n], \\
    \dot{y}_{1,i}&=f_{1,i}\coloneqq a_{1,i}x_{i-1} x_{ E}-(d_{1,i}+k_{1,i})y_{1,i} \hspace{24mm}\text{for }i\in [n],\\
    \dot{y}_{2,i}&=f_{2,i}\coloneqq a_{2,i}x_i x_{ F}-(d_{2,i}+k_{2,i})y_{2,i} \hspace{28mm}\text{for } i\in [n], \\
    \dot{x}_{E}&=f_E\coloneqq \sum_{i=1}^{n} -a_{1,i}x_{i-1} x_{ E} + (d_{1,i}+k_{1,i})y_{1,i}, \\
    \dot{x}_{F}&=f_F\coloneqq \sum_{i=1}^n -a_{2,i}x_i x_{ F} + (d_{2,i}+k_{2,i})y_{2,i},
  \end{align*}
  where $a_{2,0}=d_{2,0}=k_{1,0}=0$ and $a_{1,n+1}=d_{1,n+1}=k_{2,n+1}=0$, the remaining $a_{j,i}$'s, $d_{j,i}$'s, $k_{j,i}$'s are parameters, and the $x$'s and the $y$'s are variables.
  The solutions of the evolution equations describe a three-dimensional set of steady states. In order to obtain finitely many solutions, observe that the following quantities are conserved:
  \begin{align*}
    E_{tot}&=g_E\coloneqq x_{E}+\sum_{i=1}^{n} y_{1,i}-c_E, \qquad F_{tot}=g_F\coloneqq x_{F}+\sum_{i=1}^n y_{2,i}-c_F,\\
    X_{tot}&=g_X\coloneqq \sum_{i=0}^n x_i+\sum_{i=1}^{n} (y_{1,i}+ y_{2,i})-c_F
  \end{align*}
  To make the system square, we can omit a suitable subset of the evolution equations depending on the conservation equations. In the system above, we may omit $f_0,f_E,f_F$ and consider the system consisting of $f_i,f_{1,i},f_{2,i}$ for $i=1,\dots,n$ and $g_E,g_F,g_X$.  In other words, we have $Y=\Spec(A)$ and $X=\Spec(A[x^\pm,y^\pm]/I)$ where
  \begin{align*}
    A&\coloneqq \CC\big[a_{j,i},d_{j,i},k_{j,i}\mid j\in[2], i\in [n]\big],\\
    A[x^\pm,y^\pm]&\coloneqq A\big[x_0,x_i,x_E,x_F,y_{j,i}\mid j\in[2], i\in [n]\big], \quad \text{and }\\
    I&\coloneqq \langle f_i,f_{1,i},f_{2,i},g_E,g_F,g_X\mid i\in [n]\rangle.
  \end{align*}

  If the conservation equations $g_{E}$, $g_{F}$ and $g_{X}$ were generic, in the sense that each of its monomials comes with a unique parameter as in the evolution equations, then the modifications in \cref{def:verticalDependenciesModification} would be immediately applicable, so that the generic root count is given by the tropical intersection number in \cref{prop:verticalDependenciesRootCount}.

  To address non-generic conservation equations, we can adjust the construction \cref{def:verticalDependenciesModification} as follows:

  Let the modification on the $f$'s be exactly as in \cref{def:verticalDependenciesModification}, namely for $i=1,\dots,n$ and using $z$ to denote the newly introduced variables:
  \begin{equation*}
    \begin{array}{c}
      \hat f_i\coloneqq -a_{1,i+1}z_{iE} -a_{2,i}z_{iF}+d_{1,i+1}y_{1,i+1}+d_{2,i}y_{2,i}  +k_{1,i}y_{1,i}+k_{2,i+1}y_{2,i+1}, \\[2mm]
      \hat f_{1,i}\coloneqq a_{1,i}z_{i-1,E}-(d_{1,i}+k_{1,i})y_{1,i}, \hspace{11mm}
      \hat f_{2,i}\coloneqq a_{2,i}z_{i,F}-(d_{2,i}+k_{2,i})y_{2,i},\\[2mm]
      \hat h_{0E}\coloneqq z_{0E} - x_0x_E, \hspace{13mm}
      \hat h_{iE}\coloneqq z_{iE} - x_ix_E, \hspace{13mm}
      \hat h_{iF}\coloneqq z_{iF} - x_ix_F\\[2mm]
      \in A[x^\pm,y^\pm,z^\pm]\coloneqq A\big[x_0,x_i,x_E,x_F,y_{j,i},z_{0E},z_{iE},z_{iF}\mid j\in[2], i\in [n]\big]
    \end{array}
  \end{equation*}
  and keep $\hat g_E\coloneqq g_E$, $\hat g_F\coloneqq g_F$, $\hat g_X\coloneqq g_X \in A[x^\pm,y^\pm,z^\pm]$. Set
  \begin{align*}
    \widehat{X}_\lin\coloneqq V(\hat f_i,\hat f_{1,i},\hat f_{2,i}\mid i\in[n]), \quad
    \widehat{X}_{\text{con}}\coloneqq V(\hat g_E,\hat g_F,\hat g_X),\quad
    \widehat{X}_\nlin\coloneqq V(\hat h_1,\dots,\hat h_m).
  \end{align*}
  As $\trop(\widehat{X}_{\nlin,P})$ and $\trop(\widehat{X}_{\text{con},P})$ intersect transversally for generic~$P$, and $\widehat X_\lin$ is equivariant and parametrically independent to both $\widehat{X}_{\nlin}$ and $\widehat{X}_{\text{con}}$, we obtain for generic $P$
  \begin{equation*}
    \label{eq:crn}
    \begin{array}{rcl}
      \ell_{X,\eta}&\overset{\text{Prop. \ref{thm:TorusEquivariant}}}{=}& \trop(\widehat{X}_{\mathrm{lin},P})\cdot \trop(\widehat{X}_{\mathrm{nlin},P}\cap \widehat{X}_{\mathrm{con},P})\\[2mm]
                   &=& \trop(\widehat{X}_{\mathrm{lin},P})\cdot \trop(\widehat{X}_{\mathrm{nlin},P})\cdot \trop(\widehat{X}_{\mathrm{con},P}).
    \end{array}
  \end{equation*}
  We thus see that the generic root count is expressible as a tropical intersection product between a tropical linear space $\trop (\widehat{X}_{\mathrm{lin},P})\cap \trop (\widehat{X}_{\mathrm{con},P})$ and a linear space $\trop (\widehat{X}_{\mathrm{nlin},P})$. The above technique works for arbitrary chemical reaction networks, provided that the generic codimension of $\widehat{X}_{\mathrm{lin}}$ is as expected as in \cref{lem:GenFiniteImpliesGenCI}.

  Note however that obtaining the actual tropical intersection numbers, such as $\trop(\widehat{X}_{\mathrm{lin},P})\cdot \trop(\widehat{X}_{\mathrm{nlin},P})\cdot \trop(\widehat{X}_{\mathrm{con},P})=3$ for $n=2$ and $\trop(\widehat{X}_{\mathrm{lin},P})\cdot \trop(\widehat{X}_{\mathrm{nlin},P})\cdot \trop(\widehat{X}_{\mathrm{con},P})=5$ for $n=3$, remains a challenging computational task in its own right.  Expressing the generic root count as a tropical intersection numbers thus does not outright solve the difficult task of computing the generic root count, rather it gives us a new combinatorial approach for tackling it %
  \cite{HelminckHenrikssonRen2024}.
\end{example}

\subsection{Square systems with horizontal parameter dependencies}\label{sec:horizontalDependencies}
In this section, we consider a class of parametrized square systems inspired by the work of Kaveh and Khovanskii~\cite{KavehKhovanskii2012}.

\begin{definition}\label{def:horizontalDependencies}
  Let $x^{\alpha_1},\dots,x^{\alpha_s}$ be the monomials of $f_1,\dots,f_n$.
  We say $f_1,\dots,f_n$ have \emph{horizontal parameter dependencies}, if there is a decomposition of the parameter ring $A=\bigotimes_{i=1}^{n} A_{i}$, such that each $f_i$ is of the form
  \[ f_i=\sum_{j=1}^s p_{i,j}\cdot x^{\alpha_j} \qquad \text{with } p_{i,j}\in A_{i}. \]
  In other words, the coefficient matrix $(p_{i,j})_{i\in[n],j\in[m]}\in A^{n\times s}$ has algebraic dependencies along its rows, but not columns. We will assume that the individual $A_{i}$ are again free polynomials rings over $K$ with a fixed choice of linear generators. %
  By expressing the $p_{i,j}$'s in terms of these generators and grouping together monomials in the $f_{i}$'s, we can then write %
\begin{equation}\label{eq:HorizontalDependenciesEquation} f_i=\sum_{j=1}^m a_{i,j}\cdot q_{j} \qquad \text{with } a_{i,j}\in A_{i}.
  \end{equation}
  Note that the non-zero $a_{i,j}$'s are algebraically independent by construction.
\end{definition}

\begin{definition}\label{def:horizontalDependenciesModification}
  Let $f_1,\dots,f_n$ have horizontal parameter dependencies. Choosing the $q_j$'s in \cref{def:UniversalModification} as in \cref{eq:HorizontalDependenciesEquation}, we obtain in $\widehat C \coloneqq A[x_i^\pm, w_{j}^\pm \mid i\in [n], j\in[m]]$:
  \begin{equation*}
    \hat{f}_{i}\coloneqq\sum_{j=1}^m a_{i,j}w_{j} \quad \text{for } i\in[n]\qquad\text{and}\qquad \hat{h}_{j}\coloneqq w_{j}-q_{j} \quad \text{for } j\in[m].
  \end{equation*}
  By \cref{ass:SquareSystemsParameterSpace} and \cref{lem:TorusEquivariantModifications}, $\widehat X_\lin\coloneqq V(\langle \hat f_1,\dots, \hat f_n\rangle)$ is torus-equivariant. Moreover, $\widehat X_{\nlin,P}\coloneqq V(\langle \hat h_1,\dots, \hat h_m\rangle)$ is independent of $P$. We will refer to $\widehat X\coloneqq \widehat X_\lin\cap\widehat X_{\nlin}$ as the \emph{modification} for horizontal dependencies.
\end{definition}

Recall that for vertical dependencies, $\trop(\widehat X_\lin)$ is a general tropical linear space while  $\trop(\widehat X_\nlin)$ is a tropical complete intersection. For horizontal dependencies, the situation is reversed and $\trop(\widehat X_\lin)$ is a tropical complete intersection while $\trop(\widehat X_\nlin)$ can be a more general tropical variety. That is, the $\hat{h}_{j}$'s need not give a tropical basis, and the stable intersection of the $\trop(V(\hat{f}_{i}))$'s might give different tropical intersection numbers.  %

\begin{proposition}\label{prop:horizontalDependenciesRootCount}
  Let $f_1,\dots,f_n$ have horizontal parameter dependencies and let $U$ be the rectifiable locus of the modification. Then, for $P\in Y^\an$ such that %
  $a_{i,j}(P)\neq 0$ unless $a_{i,j}=0$, we have
  \begin{equation*}
    \ell_{X\cap U,\eta} = \trop(\widehat X_{\nlin,P})\cdot \prod_{i=1}^n \trop(V(\widehat{f}_{i})_{P}).
  \end{equation*}
\end{proposition}
\begin{proof}
    Note that these $P$ lie in the tropically flat locus by \cref{cor:TropConstant}. Moreover, the $\trop(V(\widehat{f}_{i}))$'s are torus-equivariant and independent by construction, so that the result follows from \cref{thm:TorusEquivariant}.
\end{proof}

\begin{remark}[Comparison to the works of Kaveh and Khovanskii]
  Let $\mathcal X$ be an $n$-dimensional irreducible variety over the complex numbers, and let $L_{1},\dots,L_n\subseteq\mathbb{C}(\mathcal X)$ be linear subspaces of the function field.
  In \cite{KavehKhovanskii2010}, Kaveh and Khovanskii define the \emph{birational intersection index} $[L_{1},\dots ,L_{n}]$, which records the generic number of solutions of $h_{1}=\dots=h_{n}=0$ in $\mathcal X$ for $h_{i}\in L_i$ generic.
  In \cite{KavehKhovanskii2012}, a suitable higher-rank valuation $\mathbb{C}(X)\backslash\{0\}\to\mathbb{Z}^{n}$ is used to  attach a convex body $\Delta_{L}$, the \emph{Newton-Okounkov body}, to any subspace $L$. In \cite[Theorem 4.9]{KavehKhovanskii2012} it is then proved that
  \begin{equation*}
    [L,\dots ,L]=\dfrac{n! \cdot  \mathrm{deg}(\Phi_{L})}{\mathrm{ind}(A_{L})}\cdot \mathrm{vol}(\Delta_{L}),
  \end{equation*}
  where $\Phi_{L}$ is the Kodaira map and $\mathrm{ind}(A_{L})$ is the index of a certain subgroup. %
  If the Kodaira map is a birational, then both $\mathrm{deg}(\Phi_{L})$ and $\mathrm{ind}(A_{L})$ are equal to $1$, so that the formula reduces to $[L,\dots ,L]=n! \cdot \mathrm{vol}(\Delta_{L})$.

  In comparison, our paper always assumes that the ambient variety $\mathcal X$ is a torus, though it may be over any field $K$. Let $q_{i,1},\dots,q_{i,k_i}$ be a basis of $L_{i}$ and consider
  \begin{equation*}
    f_{i}=\sum_{j=1}^{k_i} a_{i,j}q_{i,j} \qquad \text{for } i=1,\dots,n.
  \end{equation*}
   These give a square system with horizontal parameter dependencies. The open subset used in \cite[Definition 4.5 (1)]{KavehKhovanskii2012} is a subset of our rectifiable locus. Namely, if we write $Z_{i}=\bigcap_{j=1}^m V(q_{i,j})$, then $Z_{\mathbf{L}}$ from \cite[Section 4.2]{KavehKhovanskii2012} is $\bigcup_{i=1}^{n} Z_{i}$. Its complement is then easily seen to be a union of open subsets $U$ as in \cref{cor:OpenSetExtendGRC}.

  In \cite[Definition 4.5 (2)]{KavehKhovanskii2012}, it is furthermore required that the solutions are non-degenerate in the sense that the generic fiber of the morphism $X\to{Y}$ is \'{e}tale. This will generally not be the case if the characteristic of the base field $K$ is positive.
  For example, the parametrized polynomial $f=a_{1}+a_{2}x^{p}$ over the ring $A=\mathbb{F}_{p}[a_{1},a_{2}]$ has generic root count $p$, but the number of generic solutions where the morphism is \'{e}tale is zero. If $K=\mathbb{C}$, then the proof of \cite[Proposition 5.7]{KavehKhovanskii2010} implies that $X\cap U\to{Y}$ is generically \'{e}tale for the rectifiable locus $U$, from which we obtain
  \[ [L_1,\dots,L_n] = \trop(\widehat X_{\nlin,P})\cdot \prod_{i=1}^n \trop(V(\widehat{f}_{i})_{P}) = \ell_{X\cap U,\eta} \]
  using \cref{prop:horizontalDependenciesRootCount} and our earlier observation on the open subset used in the definition of $[L_{1},\dots ,L_{n}]$.
\end{remark}

We end this section with two examples from the literature, in which we highlight two different ideas to simplify the tropical intersection product in \cref{prop:horizontalDependenciesRootCount}:
\begin{enumerate}
    \item In \cref{ex:Kuramoto}, $\widehat X_\nlin$ is quasi-linear in the sense that it is the preimage of a linear space under a finite toric morphism.
    \item In \cref{exa:DuffingOscillators}, $\widehat X_\nlin$ is a tropical complete intersection and we showcase how one can simplify the resulting mixed volume.
\end{enumerate}

\begin{example}[Kuramoto model]\label{ex:Kuramoto}
Consider the following polynomials from \cite[Equation F3]{Kuramoto2019}, which describe the stationary equations of the Kuramoto model for a simple graph $G$ with vertex set $[N]$ and edge set $E(G)$:
\begin{equation}\label{eq:Kuramoto}
  f_{i}\coloneqq \sum_{\{ij\}\in E(G)} a_{i,j}(x_{i}x_{j}^{-1}-x_{j}x_{i}^{-1})-b_{i}\quad \text{for } i\in 1,\dots,N-1 \text{ and } x_N\coloneqq 1.
\end{equation}
Note that the parameters were renamed to $a_{i,j}$ and $b_i$, and some constants were omitted that do not change the generic root count.  The modification in \cref{def:horizontalDependenciesModification} yields in $\widehat C \coloneqq A[x_i^\pm, w_{i,j}^\pm \mid i\in [N-1], \{ij\}\in E(G)]$:
\allowdisplaybreaks
  \begin{align*}
    \hat{f}_{i}&\coloneqq\sum_{\{ij\}\in E(G)} a_{i,j}w_{i,j}-b_i &&\text{for } i\in [N-1],\\
    \hat{h}_{ij}&\coloneqq w_{i,j}-(x_{i}x_{j}^{-1}-x_{j}x_{i}^{-1}) &&\text{for }\{ij\}\in E(G).
  \end{align*}
  By \cref{prop:horizontalDependenciesRootCount}, the generic root count equals $\trop(\widehat X_{\nlin,P})\cdot \prod_{i=1}^{N-1} \trop(V(\widehat{f}_{i})_{P})$.  We will now describe the tropical intersection product in terms of the graph $G$, using $\trop(\Gamma)$ to denote the Bergman fan of a graphic matroid of a graph $\Gamma$.  For more information on Bergman fans of graphic matroids, see \cite[Example 4.2.14]{MS15}.

  First note that $\trop(V(\widehat{f}_{i})_{P}) = \trop(\Star(G,i))$, where $\Star(G,i)$ is the subgraph of $G$ consisting vertex $i$ as well as all vertices and edges adjacent to it.

  Moreover, let $\hat{T}\coloneqq \Spec(\hat{C})$, and consider the automorphism $\kappa_{1}\colon \hat T\rightarrow\hat T$ and the Kummer map $\kappa_2\colon\hat T\rightarrow\hat T$ that are defined by the following maps on the level of coordinate rings:
  \begin{align*}
      \kappa^{*}_{1}\colon& \hat C\rightarrow\hat C,\qquad a_{i,j}\mapsto a_{i,j}, \quad x_{i}\mapsto x_{i}, \quad w_{i,j}\mapsto  x_{i}x_{j}w_{i,j},\\
      \kappa^{*}_{2}\colon& \hat C\rightarrow\hat C,\qquad a_{i,j}\mapsto a_{i,j}, \quad x_{i}\mapsto x_{i}^2, \quad w_{i,j}\mapsto  w_{i,j}.
  \end{align*}
  Their composition $\kappa=\kappa_{2}\circ \kappa_{1}$ is a finite map of tori of degree $2^{N-1}$ and $\hat X_{\mathrm{nlin}}$ is the inverse image of the linear space $\bigcap_{\{ij\}\in E(G)} V(w_{i,j}-(x_{i}-x_{j}))$ under $\kappa$.  One can show that $\trop(\bigcap_{\{ij\}\in E(G)} V(w_{i,j}-(x_{i}-x_{j})))$ is the Bergman fan $\trop(\hat G)$, where $\hat G$ is the cone graph over $G$.  As $\kappa$ is monomial, we then have $\trop(\widehat X_{\nlin,P})=\kappa^{\trop}(\trop(\hat G))$, where $\kappa^{\trop}\colon \RR^{|w|+|x|}\rightarrow\RR^{|w|+|x|}$ is the map which scales all $x$-coordinates by $2$.

  The rectifiable locus is the entirety of $T$, so that we obtain the formula
  \begin{equation*}
    \ell_{X,\eta} = \kappa^{\trop}(\trop(\hat G))\cdot \prod_{i\in[N-1]} \trop(\Star(G,i)).
  \end{equation*}
  As in \cref{ex:CRN}, obtaining the actual intersection number, such as $\kappa^{\trop}(\trop(\hat G))\cdot \prod_{i\in[N-1]} \trop(\Star(G,i))=6$ for $G$ the complete graph on $N=3$ vertices, is a challenging task in its own right.  However, the formula allows for a new combinatorial approach for computing the generic root count.
\end{example}

\begin{example}\label{exa:DuffingOscillators}
  Fix $N>0$, and consider the following polynomials for $i\in[N]$:
  \begin{align*}
    f_{i}&=a_{1,i}u_{i}(u_{i}^{2}+v_{i}^{2})+a_{2,i}u_{i}+a_{3,i}v_{i}+a_{4,i}+\sum_{j\neq{i}} c_{j,i}v_{j},\\
    g_{i}&=b_{1,i}v_{i}(u_{i}^{2}+v_{i}^{2})+b_{2,i}u_{i}+b_{3,i}v_{i}+b_{4,i}+\sum_{j\neq{i}} d_{j,i}u_{j}.
  \end{align*}
  Here, the $a_{j,i}$'s, $b_{j,i}$'s, $c_{j,i}$'s, $d_{j,i}$'s are the parameters and the $u_{i}$'s, $v_{i}$'s are the variables. Note that there are $2N$ equations in $2N$ variables.   This polynomial system describes the steady states of coupled Duffing oscillators. In \cite{BMMT2022}, Breiding, Michałek, Monid and Telen used %
  Newton-Okounkov bodies and Khovanskii bases to show that the generic root count of this system is $5^{N}$. %
  We will show here how the same root count can be obtained from our results. %

  Applying our modifications in \cref{def:horizontalDependenciesModification} yields:
  \allowdisplaybreaks
  \begin{align*}
    \hat{f}_{i}=&a_{1,i}w_{i,1}+a_{2,i}u_{i}+a_{3,i}v_{i}+a_{4,i}+\sum_{j\neq{i}} c_{j,i}v_{j},& \hat h_{i,1}=&w_{i,1}-u_{i}(u_{i}^{2}+v_{i}^{2}),\\
    \hat{g}_{i}=&b_{1,i}w_{i,2}+b_{2,i}u_{i}+b_{3,i}v_{i}+b_{4,i}+\sum_{j\neq{i}} d_{j,i}u_{j},& \hat h_{i,2}=&w_{i,2}-v_{i}(u_{i}^{2}+v_{i}^{2}),
  \end{align*}
  which we can reformulate to the following generating set
  \begin{equation}\label{eq:duffingOscilattors}
    \begin{aligned}
      \hat{f}_{i}=&a_{1,i}w_{i,1}+a_{2,i}u_{i}+a_{3,i}v_{i}+a_{4,i}+\sum_{j\neq{i}} c_{j,i}v_{j},& \hat h_{i}=&w_{i,1}-u_{i}(u_{i}^{2}+v_{i}^{2}),\\
      \hat{g}_{i}=&b_{1,i}w_{i,2}+b_{2,i}u_{i}+b_{3,i}v_{i}+b_{4,i}+\sum_{j\neq{i}} d_{j,i}u_{j},& \hat \rho_{i}=&v_{i}w_{i,1}-u_{i}w_{i,2}.
    \end{aligned}
  \end{equation}
  Note that there are now $4N$ equations in $4N$ variables.
  As before, the rectifiable locus is $T$, so that $\ell_{X,\eta}=\ell_{\widehat{X},\eta}$.

  We will first show that the $\trop(V(\hat h_i)_P)$'s and $\trop(V(\hat \rho_i)_P)$'s intersect transversally, which combined with \cref{prop:horizontalDependenciesRootCount} implies that the generic root count is the mixed volume of the polynomials in System~\eqref{eq:duffingOscilattors}. We then use a result by Bihan and Soprunov \cite{BihanSoprunov2019} to show that this mixed volume is $5^N$.

  To see that the aforementioned tropical hypersurfaces intersect transversally, observe that $\trop(V(\hat h_i)_P)$ consists of three maximal cells, while $\trop(V(\hat \rho_i)_P)$ consists of only one. Letting $\sigma_i$ denote a maximal cell of $\trop(V(\hat h_i)_P)$ and $\tau_i$ %
  the maximal cells of $\trop(V(\hat \rho_i)_P)$, we have
  \begin{equation*}
    \sigma_i\subseteq
    \begin{cases}
      (e_{w_{i,1}}-3e_{u_i})^\perp \text{ or}\\
      (e_{w_{i,1}}-e_{u_i}-2e_{v_i})^\perp \text{ or}\\
      (2e_{u_i}-2e_{v_i})^\perp,
    \end{cases}
    \qquad\text{and}\qquad \tau_i=(e_{v_i}+e_{w_{i,1}}-e_{u_i}+e_{w_{i,2}})^\perp.
  \end{equation*}

  It is straightforward to check that, regardless of the choice of $\sigma_i$'s, the normal vectors of $\sigma_1,\tau_1,\dots,\sigma_{N},\tau_{N}$ specified above will always be linearly independent, which in turn implies that the cells intersect transversally. This can for example be done by constructing a matrix of normal vectors, where the rows are indexed by the maximal cells and the columns are indexed by the unit vectors in the following ordering:
  \begin{center}
    \begin{tikzpicture}[add paren/.style={left delimiter={(},right delimiter={)}}]
      \matrix (m) [matrix of math nodes, row sep=1mm, column sep=3mm]
      { & & e_{w_{1,2}} & \cdots & e_{w_{N,2}} & e_{w_{1,1}}\; e_{u_1}\; e_{v_1} & \cdots & e_{w_{N,1}}\; e_{u_N}\; e_{v_N} \\
                 & &   & &   &   & & \\
        \tau_1   & & 1 & &   &   & & \phantom{1} \\
        \vdots   & &   & &   &   & & \\
        \tau_N   & &   & & 1 &   & & \phantom{1} \\
        \sigma_1 & &   & &   & 1 & & \phantom{1} \\
        \vdots   & &   & &   &   & & \\
        \sigma_N & &   & &   &   & & 1 \\ };
      \node[anchor=base west] at (m-3-3.base east) {$\ast$};
      \node[anchor=base west] (eol3) at (m-3-8.base east) {$\ast$};
      \node[anchor=base west] at (m-5-5.base east) {$\ast$};
      \node[anchor=base west] (eol5) at (m-5-8.base east) {$\ast$};
      \node[anchor=base west] at (m-6-6.base east) {$\ast$};
      \node[anchor=base west] (eol6) at (m-6-8.base east) {$\ast$};
      \node[anchor=base west] (eol8) at (m-8-8.base east) {$\ast$};
      \draw[loosely dotted] (m-3-3.south) -- (m-5-5.north);
      \draw[loosely dotted] (eol3.south) -- (eol5.north);
      \draw[loosely dotted] (m-6-6.south) -- (m-8-8.north);
      \draw[loosely dotted] (eol6.south) -- (eol8.north);
      \draw[dashed] (m-3-3.south west) rectangle (eol3.north east);
      \draw[dashed] (m-5-5.south west) rectangle (eol5.north east);
      \draw[dashed] (m-6-6.south west) rectangle (eol6.north east);
      \draw[dashed] (m-8-8.south west) rectangle (eol8.north east);
      \node[fit=(m-3-3) (eol8), add paren, inner sep=0pt] (submatrix) {};
    \end{tikzpicture}
  \end{center}
  Regardless of the choice of $\sigma_i$'s, the matrix of normal vectors will always be in row-echelon form, and hence of full rank.

  To show that the mixed volume of the Newton polytopes of the polynomials in Equation~\eqref{eq:duffingOscilattors} is $5^N$, recall \cite[Proposition 3.2]{BihanSoprunov2019} which states that for polytopes $Q_1,\dots,Q_n$ and $P_1\subseteq Q_1$ in $\RR^n$:
  \begin{align}
      &\MV(P_1,Q_2,\dots,Q_n) = \MV(Q_1,Q_2,\dots,Q_n) \nonumber \\
      &\quad \Longleftrightarrow \quad \forall u\in\RR^n \text{ with } \MV(Q_2^u,\dots,Q_n^u)>0:\; P_1\cap Q_1^u\neq\emptyset. \label{eq:BihanSoprunov}
  \end{align}
  Here $P^u$ denotes the face of $P$ minimizing scalar product by $u$ and the polytopes in $\{Q_2^u,\dots,Q_n^u\}$ are considered to be polytopes of $u^\perp\cong\RR^{n-1}$. We will use this result to show that the Newton polytopes of $\hat f_i$ and $\hat g_i$ can be replaced by the Newton polytopes of
  \begin{equation*}
    a_{1,i}w_{i,1}+a_{2,i}u_{i}+a_{3,i}v_{i}+a_{4,i}\qquad \text{and}\qquad
    b_{1,i}w_{i,2}+b_{2,i}u_{i}+b_{3,i}v_{i}+b_{4,i}
  \end{equation*}
  without changing the mixed volume. A quick computation then reveals the mixed volume to be $5^N$.

  Let $(\dutchcal{u},\dutchcal{v},\dutchcal{w})$ be a vector whose minimum on the Newton polytope of $\hat f_i$ is uniquely attained at a vertex corresponding to a monomial in $\sum_{j\neq i}c_{i,j}v_j$. Here, $\dutchcal{u}_i$'s, $\dutchcal{v}_i$'s, $\dutchcal{w}_{i,j}$'s represent weights on the variables $u_i$'s, $v_i$'s, $w_{i,j}$'s. Without loss of generality, we may assume that $i=1$ and that the monomial is $v_2$, so that the assumption implies:
  \begin{equation*}\label{eq:f1assumption}
    \dcv_2<\dcw_{1,1},\quad \dcv_2<\dcu_1,\quad \dcv_2<0,\quad \text{and} \quad \dcv_2<\dcv_j \text{ for } j\neq 2.
  \end{equation*}
  We will show that the mixed volume in Expression \eqref{eq:BihanSoprunov} is zero. This is done by assuming that none of the polytopes in it are vertices and proving that it either leads to a contradiction or to mixed volume $0$.

    In the following, %
    we will use ($f$) as a shorthand for a tropical equation derived from polynomial $f$. In other words, ($f$) means that the face of the Newton polytope of $f$ minimizing $(\dutchcal{u},\dutchcal{v},\dutchcal{w})$ is not a vertex, or equivalently that the minimum of $\trop(f)(\dutchcal{u},\dutchcal{v},\dutchcal{w})$ is attained at least twice.
  We distinguish between three cases:
  \begin{description}[leftmargin=5mm]
  \item[$\dutchcal{u}_2<\dutchcal{v}_2$] From $(\hat h_2)$ and $(\hat f_2)$ we obtain $\dcw_{2,1}=3\dcu_{2}$ and $\dcw_{2,1}=\dcu_2$, respectively. Together, they imply $\dcu_2=0$, which contradicts $\dcu_2<\dcv_2<0$.
  \item[$\dutchcal{u}_2>\dutchcal{v}_2$] From $(\hat \rho_2)$ we obtain $\dcw_{2,2}=3\dcv_2<\dcv_2$, the last inequality simply following from $\dcv_2<0$. By $(\hat g_2)$ we then get that $\dcw_{2,2}=\dcu_{j}$ for some $j\neq 2$. We therefore have $\dcu_{j}<\dcv_2$. Considering $(\hat f_j)$, we thus obtain $\dcw_{j,1}=\dcu_j<\dcv_j$. From $\dcu_j<\dcv_j$ and $(\hat h_j)$ we get $\dcw_{j,1}=3\dcu_j$, which together with the previous $\dcw_{j,1}=\dcu_j$ implies $\dcu_j=0$. This contradicts $\dcu_j<\dcv_2<0$.
  \item[$\dutchcal{u}_2=\dutchcal{v}_2$] From $(\hat \rho_2)$ and $(\hat h_2)$ we obtain $\dcw_{2,1}=\dcw_{2,2}$ and $\dcw_{2,1}\geq 3\dcu_2=\dcu_2+2\dcv_2$ respectively. We again distinguish between three cases:
    \begin{description}[leftmargin=3mm]
    \item[$\dcu_2=\dcv_2>\dcw_{2,1}=\dcw_{2,2}$] The minimum of $\trop(f_2)(\dutchcal{u},\dutchcal{v},\dutchcal{w})$ is attained uniquely at $\dcw_{2,1}$, hence the mixed volume is $0$.
    \item[$\dcu_2=\dcv_2<\dcw_{2,1}=\dcw_{2,2}$] Suppose that $\dcv_2<\dcu_i$ for all $i\neq 2$. Then %
      the minimum of both tropical polynomials $\trop(f_2)$ and $\trop(g_2)$ evaluated at $(\dutchcal{u},\dutchcal{v},\dutchcal{w})$ is attained at the monomials $u_2$ and $v_2$, hence the mixed volume is $0$.

      Suppose that $\dcv_2\geq\dcu_i$ for some $i\neq 2$. Then $\dcu_i\leq \dcv_2<\dcv_i$ and from $(\hat h_i)$ we obtain $\dcw_{i,1}=3\dcu_{i}<\dcu_{i}$. This implies that the minimum of $\trop(\hat f_i)$ evaluated at $(\dutchcal{u},\dutchcal{v},\dutchcal{w})$ is uniquely attained at the monomial $w_{i,1}$, contradicting $(\hat f_i)$.

    \item[$\dcu_2=\dcv_2=\dcw_{2,1}=\dcw_{2,2}$] From $(\hat h_2)$ we obtain $\dcw_{2,1}\geq 3\dcu_2$. Combined with the assumptions $\dcu_2=\dcw_{2,1}$ and $\dcw_{2,1}=\dcv_2<0$, this implies $\dcw_{2,1}>3\dcu_2$. Hence the minimum of $\trop(\hat h_2)(\dutchcal{u},\dutchcal{v},\dutchcal{w})$ is attained uniquely at $u_2^3$ and $u_2v_2^2$. Our assumptions on $\hat f_1$ imply that the minimum in $\trop(\hat f_2)(\dcu,\dcv,\dcw)$ is attained at $w_{2,1}$, $u_2$, and $v_2$. We again distinguish between three cases:

    If the minimum in $\trop(\hat{g}_{2})(\dcu,\dcv,\dcw)$ is attained at $v_2$ and $u_{j}$ for $j\neq{2}$, then from our initial assumption, we obtain $\dcu_{j}<\dcv_{j}$ and thus $\dcw_{j,1}=3\dcu_{j}<\dcu_{j}$. But then the minimum in $\trop(\hat f_j)(\dcu,\dcv,\dcw)$ is uniquely attained at $w_{j,1}$, contradicting $(\hat f_j)$.

    If the minimum in $\trop(\hat{g}_{2})(\dcu,\dcv,\dcw)$ is attained at $u_{j}$ and $u_{k}$ for $j,k\neq{2}$ and $j\neq k$, then the minimum in $\trop(\hat f_j)(\dcu,\dcv,\dcw)$ is attained at $w_{j,1}$ and $u_{j}$. As before, this implies that $\dcw_{j,1}=3\dcu_{j}=0$, which contradicts $(\hat f_j)$.

    The remaining case is where the minimum in $\trop(\hat{g}_{2})(\dcu,\dcv,\dcw)$ is attained at $w_{2,2}$, $u_2$, and $v_2$. But then the mixed volume of the Newton polytopes of  $\initial(h_{i,1})=u_{2}(u_{2}^2+v_{2}^2)$, $\initial(\rho_{2})$, $\initial(f_{2})$ and  $\initial(g_{2})$ is zero, since they contain a common lineality space.
 \end{description}
 \end{description}
\end{example}
 
\renewcommand*{\bibfont}{\small}
\printbibliography

\end{document}